\documentclass[10pt,reqno]{amsart}

\usepackage[latin2]{inputenc}
\usepackage{amsmath}
\usepackage{graphicx}
\usepackage{amssymb}
\usepackage{esint}
\usepackage{color}
\usepackage{amsthm}
\usepackage{epsfig}
\usepackage{enumitem}
\usepackage[english]{babel}
\allowdisplaybreaks[2]

\newtheorem{theorem}{Theorem}
\newtheorem{proposition}[theorem]{Proposition}
\newtheorem{lemma}[theorem]{Lemma}
\newtheorem{corollary}[theorem]{Corollary}
\newtheorem*{theorem*}{Theorem}

\theoremstyle{definition}
\newtheorem{definition}[theorem]{Definition}

\theoremstyle{remark}
\newtheorem{remark}[theorem]{Remark}

\definecolor{Yellow}{rgb}{0.95,0.9,0.0} 
\definecolor{Red}{rgb}{0.8,0.1,0.1}
\definecolor{Green}{rgb}{0.1,0.65,0.2}
\definecolor{Blue}{rgb}{0.1,0.1,0.8}
\definecolor{Purple}{rgb}{0.7,0.1,0.7}
\definecolor{Grey}{rgb}{0.6,0.6,0.6}

\newcommand{\Prob}{\mathbf{P}}
\newcommand{\E}{\mathbf{E}}
\newcommand{\Hsp}{\mathbf{H}}
\newcommand{\Lsp}{\mathbf{L}}
\newcommand{\Filt}{\mathbf{F}}
\newcommand{\F}{\mathcal{F}}
\newcommand{\R}{\mathbf{R}}

\newcommand{\Ccpt}{C^\infty_{\mathrm{cpt}}}

\DeclareMathOperator*{\supp}{supp}
\newcommand{\dx}{\,\mathrm{d}x}
\newcommand{\dt}{\,\mathrm{d}t}
\newcommand{\dy}{\,\mathrm{d}y}
\newcommand{\ds}{\,\mathrm{d}s}
\newcommand{\dB}{\,\mathrm{d}B}

\newcommand{\da}{\,\mathrm{d}a}
\newcommand{\db}{\,\mathrm{d}b}
\newcommand{\ddt}{\frac{\mathrm{d}}{\mathrm{d}t}}
\newcommand{\Tweak}{T^*}
\newcommand{\Text}{T_{\mathrm{extinct}}}
\newcommand{\hatText}{\hat T_{\mathrm{extinct}}}
\newcommand{\Jplus}{\mathcal{J}^{2,+}}
\newcommand{\n}{\mathbf{n}}

\usepackage[colorlinks,linkcolor=black,citecolor=blue,urlcolor=blue]{hyperref}

\begin{document}

\title[Finite time extinction for SPME with transport noise]
{Finite time extinction for the 1D stochastic porous medium equation
with transport noise}

\author{Sebastian Hensel}
\address{Institute of Science and Technology Austria (IST Austria), Am~Campus~1, 3400 Klosterneuburg, Austria}
\email{sebastian.hensel@ist.ac.at}

\thanks{This project has received funding from the European Union's Horizon 2020 research and 
innovation programme under the Marie Sk\l{}odowska-Curie Grant Agreement No.\ 665385 
\begin{tabular}{@{}c@{}}\includegraphics[width=3ex]{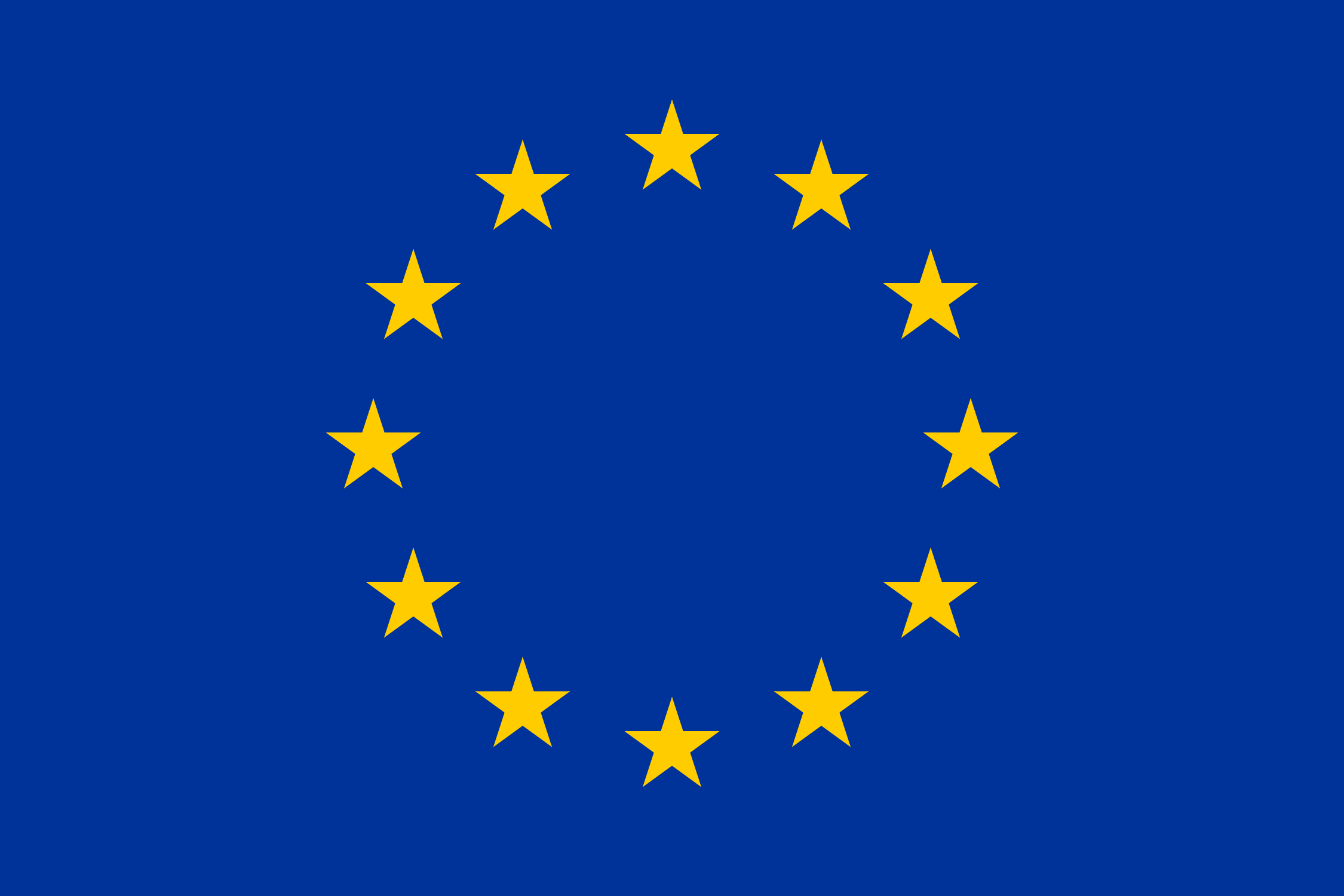}\end{tabular}.
I am very grateful to M.~Gerencs\'{e}r for the suggestion of this problem. I want to thank 
J.~Maas and M.~Gerencs\'{e}r for helpful discussions on the subjects of this paper, as well
as valuable feedback on an earlier version of this manuscript. Special thanks go to 
F.~Cornalba for suggesting the additional $\kappa$-truncation in Proposition~\ref{prop:L1contraction}.}%

\begin{abstract}
We establish finite time extinction with probability one for weak solutions of the
Cauchy--Dirichlet problem for the 1D stochastic porous medium equation with Stratonovich transport noise 
and compactly supported smooth initial datum. Heuristically, this is expected to hold
because Brownian motion has average spread rate $\smash{O(t^\frac{1}{2})}$ 
whereas the support of solutions to the deterministic PME grows only with rate 
$\smash{O(t^{\frac{1}{m{+}1}})}$. The rigorous proof relies on a contraction
principle up to time-dependent shift for Wong--Zakai type approximations, 
the transformation to a deterministic PME with two copies of a Brownian path 
as the lateral boundary, and techniques from the theory of viscosity solutions. 
\end{abstract}

\maketitle

\section{Introduction}
We study the stochastic porous medium equation with Stratonovich transport noise
\begin{align}
\label{eq:SPME}
\mathrm{d}u = \Delta u^m\,\mathrm{d}t + \nu\nabla u\circ\mathrm{d}B_t
\end{align}
in the space-time cylinder $Q=D\times (0,T)$
for some bounded convex domain $D\subset\R^d$ with $C^2$ boundary $\partial D$ and time horizon $T\in (0,\infty]$.
Here $u\geq 0$ denotes the so-called \textit{density variable}, $B$ is
a $d$-dimensional standard Brownian motion and $m>1$ as well as $\nu>0$ are parameters. 
Moreover, the Stratonovich differential $\nu\nabla u\circ\mathrm{d}B$
is an abbreviation for the expression $\frac{1}{2}\nu^2\Delta u\dt + \nu\nabla u\cdot\mathrm{d}B$.

We are in particular interested in studying the corresponding Cauchy--Dirichlet problem, 
i.e., we supplement the stochastic PDE \eqref{eq:SPME} with
\begin{align}
\label{eq:IC} u(x,0) &=  u_0(x),\quad\hspace{0.25cm} x\in\bar{D}, \\
\label{eq:lateralBC} u(x,t) &= 0,\qquad (x,t)\in\partial D\times (0,T].
\end{align}
The initial condition is given by a non-negative smooth function
$ u_0\in\Ccpt(D)$ with compact support in $D$. In particular, \eqref{eq:IC}
and \eqref{eq:lateralBC} are compatible to any order.

The main result of the present work is that weak solutions (in the precise sense of Definition~\ref{def:weakSolutions} below) 
to the 1D Cauchy--Dirichlet problem of the SPME \eqref{eq:SPME}--\eqref{eq:lateralBC}
become extinct in finite time on a set of full probability. For a mathematically precise statement 
we refer the reader to Theorem~\ref{theorem:mainResult} below. To the best of the author's
knowledge, this is the first result establishing finite time extinction in the \textit{slow diffusion regime} $m>1$
for the stochastic porous medium equation. In contrast, finite time extinction in the \textit{fast diffusion regime} 
$0<m<1$ is well known in several cases. For instance, Barbu, Da Prato and R\"{o}ckner establish
finite time extinction for porous medium type models with multiplicative noise in the works 
\cite{Barbu:2009}, \cite{Barbu:2009a} and~\cite{Barbu:2012}. Gess treats the case of 
stochastic sign fast diffusion equations in~\cite{Gess:2015}. In a very recent work of Turra \cite{Turra:2019}, 
finite time extinction in the fast diffusion regime is established for transport noise. The slow diffusion
case with transport noise, however, is left open.

The underlying idea of the present work is to exploit that the support of solutions to the deterministic 
porous medium equation grows only with rate $\smash{O(t^{\frac{1}{m{+}1}})}$ whereas Brownian motion spreads 
on average with a faster rate $\smash{O(t^\frac{1}{2})}$. Hence one may expect that the effect of the random advection term
in \eqref{eq:SPME} eventually dominates the non-linear diffusion term, thus leading to finite time extinction 
on a set of full probability. Let us briefly mention the three main ingredients which provide a rigorous
justification for this heuristic (for a more detailed and mathematical account on the strategy, 
we refer the reader to Section~\ref{sec:strategy}). Along the way we provide connections to the---for the purposes 
of this work---relevant parts of the by now extensive literature for the stochastic porous medium equation.

In a first step, we recover the unique weak solution (in the sense of Definition~\ref{def:weakSolutions} below) 
to the Cauchy--Dirichlet problem of the SPME \eqref{eq:SPME}--\eqref{eq:lateralBC} by means of a
suitable Wong--Zakai approximation. We then perform, on the level of the approximate solutions,
a simple stochastic flow transformation yielding smooth solutions to a deterministic porous medium equation, 
but now with two copies of a mollified Brownian path as the lateral boundary. In principle, we would then like to remove
the mollification parameter to formally obtain a solution to the deterministic PME with two copies of a Brownian
path as the lateral boundary. We may view this solution---again at least on a formal level---as a solution 
to the full space porous medium equation by trivially extending it outside of the domain. We finally
put a Barenblatt profile above the solution at the initial time. It follows by the comparison principle 
that finite time extinction happens once the domain with two copies of a Brownian path as the lateral boundary 
is pushed outside the support of such a Barenblatt solution. Since the support of the Barenblatt solution has 
finite speed of propagation of order $O(t^\frac{1}{m+1})$, this actually occurs with probability one.

However, two technical issues being linked to each other arise in the above strategy. The first is concerned 
with the proposed limit passage after the stochastic flow transformation. The problem is that the domains on which
the approximate and transformed solutions are supported are not monotonically ordered as one sends the
mollification scale to zero. This fact together with the desire to make use of the comparison principle
motivated us to perform this limit passage in the framework of viscosity theory by means of the technique of 
semi-relaxed limits. This allows us to rigorously construct a \textit{maximal subsolution}
for the deterministic porous medium equation with two copies of a Brownian path as the lateral boundary (to be
understood in the precise sense of Definition~\ref{def:subsolutionCauchyDirichlet} below).
To the best of the author's knowledge, the present work seems to be the first instance to make use of 
viscosity theory to study the Dirichlet problem for the stochastic porous medium equation after stochastic flow transformation.

A notion of viscosity solution for the full space deterministic porous medium equation
has been developed by Caffarelli and V\'{a}zquez~\cite{Caffarelli:1999} (see also the work of 
V\'{a}zquez and Br\"{a}ndle~\cite{Vazquez:2005}). However, as a consequence of the degeneracy of the
porous medium equation the usual notion of supersolutions from \cite{Crandall:1992} had to be adapted.
But since we only need to talk about subsolutions for our purposes, this is actually of no concern for us.

However, by taking a semi-relaxed limit it is by no means clear anymore how this maximal
subsolution and the transformed unique weak solution to  \eqref{eq:SPME}--\eqref{eq:lateralBC} 
relate to each other. In order to close the argument,
we need to make sure that the unique weak solution of the original problem \eqref{eq:SPME}--\eqref{eq:lateralBC}
is dominated after stochastic flow transformation by the constructed maximal subsolution in a suitable sense.
It turns out that we can justify this once we provided convergence of the Wong--Zakai approximations 
in a strong enough topology like~$L^1$. Let us mention at this point the 
recent work of Fehrman and Gess~\cite{Fehrman:2019}, who prove rough path well-posedness 
for nonlinear gradient type noise using a kinetic formulation which 
in particular implies Wong-Zakai type results, but with the equation being posed on the torus.

There has been recently a lot of focus on providing well-posedness for stochastic porous medium equations 
in an $L^1$ framework. The corresponding notion of solution is the one of \textit{entropy solutions}
first put forward in the work of Dareiotis, Gerencs\'{e}r and Gess~\cite{Dareiotis:2018},
and the theory is based on a quantitative $L^1$ contraction principle for such solutions.
In the work~\cite{Dareiotis:2018}, the authors consider the case of nonlinear multiplicative noise 
with periodic boundary conditions. The subsequent work of Dareiotis and Gess~\cite{Dareiotis:2019} 
treats the regime of nonlinear conservative gradient noise (again with periodic boundary conditions), 
whereas the recent work~\cite{Dareiotis:2019a} by Dareiotis, Gess and Tsatsoulis provides the 
corresponding theory of \cite{Dareiotis:2018} for the Dirichlet problem on smooth domains.

Following ideas developed in \cite{Dareiotis:2018}, we roughly speaking aim to establish $L^1$ convergence of the sequence
of Wong--Zakai approximations. However, as we are concerned in this work with the Cauchy--Dirichlet 
problem for the SPME with transport noise, the actual implementation differs in some aspects to the 
above mentioned works. Most importantly, since two different approximate solutions are advected with
different speeds (as a consequence of the two associated different mollification scales) it is
natural to introduce a time-dependent, random shift function and to compare the Wong--Zakai
approximations only after shifting. We then essentially establish quantitative $L^1$ contraction for
this sequence of shifted densities, see Proposition~\ref{prop:L1contraction} below for the mathematically 
precise formulation. This result may be of independent interest.

\subsection*{Notation}
Throughout the paper, we fix a filtered probability space $(\Omega,\F,\Filt,\Prob)$ 
with filtration $\Filt=(\F_t)_{t\in [0,\infty)}$ which is subject to the usual conditions, i.e.,
the filtration $\Filt$ is right-continuous and $\F_0$ is $\Prob$-complete. By $B=(B_t)_{t\in [0,\infty)}$
we denote a standard $d$-dimensional $\Filt$-Brownian motion. For a given time horizon $T\in (0,\infty]$ we let $\mathcal{P}_T$ be
the predictable $\sigma$-field on $\Omega_T:=\Omega\times[0,T]$. 

Let $D\subset\R^d$ be an open domain. (In the one-dimensional 
setting, we will instead use the notation $I\subset\R$ for a bounded open interval.)
The space of all smooth functions with compact support in $D$
is denoted by $C^\infty_{\mathrm{cpt}}(D)$. The space $H^1_0(D)$ is the usual space of Sobolev functions with 
zero trace on the boundary. Its topological dual space will be denoted by $H^{-1}$. The space of 
$H^{-1}$ valued, continuous and $\Filt$-adapted stochastic processes intersected with 
$L^q(\Omega_T,\mathcal{P}_T;L^q(D))$ is given by $\Hsp^{-1}_q(D)$. The space of $L^2(D)$ valued, 
continuous and $\Filt$-adapted stochastic processes intersected with the space $L^2(\Omega_T,\mathcal{P}_T;H^1_0(D))$)
(resp.\ the space $L^2(\Omega_T,\mathcal{P}_T;H^1(D))$) is $\Lsp^2_0(D)$ (resp.\ $\Lsp^2(D)$). 
If we want to specify the target values for a function space, we will include this in the notation.
E.g., $C_{\mathrm{cpt}}^\infty(D;[0,\infty))$ is the space of all non-negative smooth functions with compact support in $D$.
Finally, the space of functions of bounded variation is denoted by $BV$.

The Lebesgue measure on $\R^d$ is denoted by $\mathcal{L}^d$, the $s$-dimensional Hausdorff measure by $\mathcal{H}^s$.
We write $\chi_A$ for the indicator function of a measurable set $A$ (with respect to any measure space). For 
two numbers $a,b\in\R$ we abbreviate $a\vee b$ for their maximum respectively $a\wedge b$
for their minimum. Finally, we define $a_+:=a\vee 0$ and $\mathrm{sign}_+(a):=\chi_{(0,\infty)}(a)$ for $a\in\R$.

\section{Main result}
The main result of the present work establishes that weak solutions of the one-dimensional  
stochastic porous medium equation with Stratonovich transport noise \eqref{eq:SPME}--\eqref{eq:lateralBC} 
almost surely have the finite time extinction property. Before we provide the precise statement, let us 
introduce the notion of weak solutions. 

\begin{definition}
\label{def:weakSolutions}
Let $D\subset\R^d$ be a bounded domain with $C^2$ boundary, $\Tweak\in (0,\infty]$ be a time horizon
as well as $u_0\in\Ccpt(D;[0,\infty))$ be an initial density. A non-negative stochastic 
process $u\in\Hsp^{-1}_{m+1}(D)$ is called a \emph{weak solution to the Cauchy--Dirichlet problem}
\eqref{eq:SPME}--\eqref{eq:lateralBC} \emph{of the stochastic porous medium equation (SPME)
with initial density $u_0$ on the space time cylinder} $Q = D\times (0,\Tweak)$ if for all 
$\phi\in C^\infty_{\mathrm{cpt}}(D)$ with probability one it holds
\begin{equation}\label{eq:weakSol}
\begin{aligned}
&\int_D u(x,T)\phi(x) \dx
-\int_D  u_0(x)\phi(x)\dx
\\&
=\int_0^T\int_D \Big(u^m(x) {+} \frac{1}{2}\nu^2u(x)\Big)\Delta\phi(x) \dx \dt
-\int_0^T\int_D \nu u(x)\nabla\phi(x) \dx \dB_t
\end{aligned}
\end{equation}
for all $T\in (0,\Tweak)$.
\end{definition}

\begin{remark}
Existence and uniqueness of weak solutions in the sense of Definition~\ref{def:weakSolutions}
was established, for instance, in \cite[Theorem 3.3]{Dareiotis:2017}. In fact, the authors consider a much more
general class of degenerate quasilinear stochastic PDE, allowing for more general initial conditions and,
in particular, for signed solutions.
\end{remark}

Being equipped with the notion of weak solutions we may now formulate
the main result of the present work, which is finite time extinction 
with probability one for weak solutions of the 1D stochastic porous 
medium equation with Stratonovich transport noise. 

\begin{theorem}
\label{theorem:mainResult}
Let $I\subset\R$ be a bounded open interval, and let $u\in\Hsp^{-1}_{m+1}(I)$ be the 
unique weak solution to the SPME \emph{\eqref{eq:SPME}--\eqref{eq:lateralBC}} 
with initial density $ u_0\in\Ccpt(I;[0,\infty))$ on the space-time cylinder 
$I\times (0,\infty)$ in the sense of Definition~\ref{def:weakSolutions}. Define
\begin{align} 
\label{def:extinctionTime}
\Text := \inf\{T\geq 0\colon u(\cdot,T)=0\text{ almost everywhere in } I\}.
\end{align}
Then there exists a constant $M=M(u_0)>0$ depending only on the initial density $u_0$ such that 
\begin{align}\label{eq:boundExtinctionTime}
\Prob(\Text\leq T) \geq \Prob\big(\inf\big\{t\geq 0\colon |B_t| \geq \mathcal{L}^1(I){+}Mt^\frac{1}{m+1}\big\}\leq T\big)
\end{align}
for all $T\in (0,\infty)$. Moreover, for
$\hatText := \inf\{T\geq 0\colon |B_T| \geq \mathcal{L}^1(I){+}MT^\frac{1}{m+1}\}$
it holds $\Prob(\Text<\infty) = \Prob(\hatText<\infty) = 1$. Finally, for all $T\in (0,\infty)$ 
we have almost surely on $\{T\geq\hatText\}$ that $u(\cdot,T)=0$ is satisfied almost everywhere in $I$.
\end{theorem}

\section{Outline of the strategy}
\label{sec:strategy}
Let us comment on the strategy for the proof of the main result, Theorem~\ref{theorem:mainResult}. 
In a nutshell the argument works as follows. By stochastic flow transformation,
which in our case of pure transport noise is of particularly easy form, we may consider first
the transformed densities
\begin{align*}
v(x,t) := u(x{-}\nu B_t,t),\quad (x,t)\in \bigcup_{t\in (0,\infty)}(\nu B_t{+}I)\times\{t\}.
\end{align*}

Formally, the transformed density $v$ is then subject to the deterministic porous medium equation
\begin{align*}
\partial_t v = \partial_{xx}v^m
\quad\text{in}\quad \bigcup_{t\in (0,\infty)}(\nu B_t{+}I)\times\{t\},
\end{align*}
with initial density $v(\cdot,0)= u_0$ and Dirichlet boundary conditions on the
lateral boundary $\bigcup_{t\in [0,\infty)}(\nu B_t{+}\partial I)\times\{t\}$.
Extending $v$ trivially to the entire space-time domain $\R\times (0,\infty)$,
and denoting this extension by $\bar v$, we obtain, at least formally, a solution to the Cauchy problem of
the deterministic porous medium equation $\partial_t\bar v=\partial_{xx}\bar v^m$
with initial density $\bar v(\cdot,0) =  u_0 \geq 0$. 

Now, we may choose a Barenblatt profile $\mathcal{B}$ with free boundary 
$\partial I$ at $t=0$ and which strictly dominates the initial density,
i.e., $\{\mathcal{B}(\cdot,0)=0\}=\partial I$ and $ u_0<\mathcal{B}(\cdot,0)$ on $I$.
By comparison it follows $\bar v \leq \mathcal{B}$ on $\R\times (0,\infty)$.
Since the support of the Barenblatt solution has finite speed of propagation of order $O(T^\frac{1}{m+1})$
(with the implicit constant only depending on the initial density), we deduce
the validity of \eqref{eq:boundExtinctionTime}. However, since Brownian motion
has average spread rate of order $O(T^\frac{1}{2})$ and we assume that $m>1$, this entails
that finite time extinction happens with probability one: $\Prob(\Text<\infty) = 1$.

To make this argument rigorous, we rely on a suitable approximation procedure.
To this end, we first consider the $d$-dimensional setting and fix a finite time horizon $T^*\in (0,\infty)$. Moreover,
choose $\rho\in C^\infty_{\mathrm{cpt}}((0,1);[0,\infty))$ with $\int_{\R} \rho(r)\,\mathrm{d}r = 1$
and define $\rho_\varepsilon := \frac{1}{\varepsilon}\rho(\frac{\cdot}{\varepsilon})$
for $\varepsilon>0$. We then introduce the smooth approximations $B^\varepsilon=(B^\varepsilon_t)_{t\in [0,\infty)}$ 
to Brownian motion defined via
\begin{align}
\label{eq:approxBrownianMotion}
B^\varepsilon_t := \int_{0}^\infty \rho_\varepsilon(t{-}s)B_s\,\mathrm{d}s. 
\end{align}

Since the mollifier $\rho$ is supported on positive times, we note that $B^\varepsilon$
is $\Filt$-adapted. We also have almost surely that 
\begin{align}
\label{eq:uniformBoundDyadicApproximation}
\sup_{\varepsilon > 0}\|B^\varepsilon\|_{C^0([0,T^*])}<\infty
\end{align}
as well as
\begin{align}
\label{eq:uniformConvergenceDyadicApproxiamtion}
B^\varepsilon \to B \text{ uniformly on } [0,T^*] \text{ as } \varepsilon \to 0.
\end{align}
Fix $\alpha\in (\frac{1}{3},\frac{1}{2})$. For purely technical reasons, we make also use of the classical fact 
that there is a square integrable random variable $\mathcal{C}_\alpha$ 
such that with probability one
\begin{align}
\label{eq:HoelderContBM}
|B_t-B_s| \leq \mathcal{C}_\alpha|t-s|^\alpha
\text{ for all } s,t\in [0,T^*].
\end{align}
In particular, for each integer $M\geq 1$ and each $\delta>0$
we can find an $\varepsilon'=\varepsilon'(M,\delta)>0$ 
such that for all $\varepsilon\leq\varepsilon'$ it holds
\begin{align}
\label{eq:supportAfterShift}
\sup_{t\in [0,T^*]}|B_t-B_t^\varepsilon| \leq \delta
\text{ almost surely on } \{\mathcal{C}_\alpha\leq M\}.
\end{align}
We then proceed by considering the inhomogeneous Cauchy--Dirichlet problem
\begin{align}
\label{eq:densityPME2}
\mathrm{d}u_\varepsilon &= \Delta u^m_\varepsilon\,\mathrm{d}t 
+ \nu\nabla u_\varepsilon\cdot\mathrm{d}B^\varepsilon_t,
\quad (x,t) \in D\times (0,T^*),
\\
\label{eq:densityPMEic2}
u_\varepsilon(x,0) &=  u_0(x) + \varepsilon,
\quad\hspace{2.42cm} x\in \bar D,
\\
\label{eq:densityPMElateral2} 
u_\varepsilon(x,t) &= \varepsilon,
\quad\hspace{3.13cm} (x,t)\in \partial D\times (0,T^*].
\end{align}

Note that the regularization not only comes from the Wong--Zakai approximation
but also by choosing strictly positive initial and boundary data. In this
way we circumvent the degeneracy at $u=0$ of the porous medium operator
at the level of the approximations, i.e., we can solve the problem \eqref{eq:densityPME2}--\eqref{eq:densityPMElateral2}
in a classical and pathwise sense such that the maximum principle applies. 
Proofs will be provided in the subsequent Section~\ref{sec:solutionTheory}.

\begin{lemma}
\label{lem:classicalSolution}
For each $\varepsilon>0$, there exist $u_\varepsilon\in\varepsilon{+}\Lsp^2_0(D)$
such that the following holds true on a set with probability one 
(e.g., so that \eqref{eq:uniformBoundDyadicApproximation} holds):

For all $\varepsilon>0$ the map $u_\varepsilon$ is the unique classical solution of the 
Cauchy--Dirichlet problem \emph{\eqref{eq:densityPME2}--\eqref{eq:densityPMElateral2}} 
in the sense that $u_\varepsilon\in C^{2,1}_{x,t}(\bar{D}{\times}[0,T^*])\cap C^\infty(D{\times}(0,T^*))$ 
and the equations \emph{\eqref{eq:densityPME2}--\eqref{eq:densityPMElateral2}} are satisfied pointwise everywhere.
From the maximum principle, we have the bounds
\begin{align}
\label{boundMaximumPrinciple}
\varepsilon \leq u_\varepsilon(x,t)  
\leq \varepsilon + \|u_0\|_{L^\infty(D)}
\end{align}
for all $\varepsilon>0$ and all $(x,t)\in\bar{D}{\times}[0,T^*]$. There is a constant
$C=C(T^*,\nu,u_0)>0$ and some $\beta>0$ such that the a priori estimates
\begin{align}
\label{eq:energyInequalityApprox}
&\sup_{T\in [0,T^*]}\int_D \frac{1}{2}|u_\varepsilon|^2(T) \dx
+\int_0^{T^*}\int_D mu_\varepsilon^{m-1}|\nabla u_\varepsilon|^2 \dx \dt
\leq \int_D \frac{1}{2}|u_\varepsilon(0)|^2 \dx, 
\\
\label{eq:boundTimeDerivativeApprox}
&\sup_{T\in [0,T^*]}\E\int_D \frac{T}{2}|\nabla u_\varepsilon^m|^2(T) \dx
+ \E\int_0^{T^*}\int_D \frac{t}{2}mu_\varepsilon^{m-1}
|\partial_t u_\varepsilon|^2 \dx \dt
\leq \varepsilon^{-\beta}C,
\\
\label{eq:boundLaplacianApprox}
&\,\E\int_0^{T^*}\int_D t|\Delta u_\varepsilon^m|^2 \dx \dt
\leq \varepsilon^{-\beta}C
\end{align}
hold true almost surely for all $\varepsilon>0$. 
\end{lemma}

A key ingredient for the rigorous justification of the comparison argument outlined
at the beginning of this section is that the Wong--Zakai type approximations $u_\varepsilon$ 
from Lemma~\ref{lem:classicalSolution} recover the unique weak solution $u$ to the Cauchy--Dirichlet 
problem \eqref{eq:SPME}--\eqref{eq:lateralBC} in a certain pointwise sense. We obtain this
by establishing a sort of contraction principle for the sequence of $u_\varepsilon$ by means of 
Kru\v{z}kov's device of variable doubling~\cite{Kruzkov1970}. To this end, the proof loosely
follows the strategy of Dareiotis, Gerencs\'{e}r and Gess~\cite{Dareiotis:2018} or 
Dareiotis and Gess~\cite{Dareiotis:2019} who study entropy solutions for stochastic porous medium type equations
posed on the torus. The Cauchy--Dirichlet problem in the framework of entropy solutions
was recently studied in the work of Dareiotis, Gess and Tsatsoulis~\cite{Dareiotis:2019a}. 
However, they do not consider the case of gradient type noise. Let us also mention the 
recent work of Fehrman and Gess~\cite{Fehrman:2019}, who prove rough path well-posedness 
for nonlinear gradient type noise using a kinetic formulation which 
in particular implies Wong-Zakai type results, but with the equation being posed on the torus.. 

As we are concerned in this work with the Cauchy--Dirichlet problem for the SPME with transport noise,
and roughly speaking aim to establish $L^1$ convergence for the Wong--Zakai type approximations $u_\varepsilon$,
the actual implementation of the doubling of variables technique differs in some aspects to the above mentioned works.
First, since two different solutions $u_\varepsilon$ and $u_{\hat\varepsilon}$ are advected at different
speeds, it seems to be natural to introduce for our purpose a (time-dependent, random) shift function and 
to compare the solutions only after shifting, i.e., we are led to study $L^1$ convergence for the sequence
\begin{align*}
u^\leftarrow_\varepsilon(x,t) := u_\varepsilon(x{+}\nu(B_t{-}B^\varepsilon_t),t).
\end{align*}
Second, since the introduction of this 
shift in turn changes the domain on which the equation for $u^\leftarrow_\varepsilon$ is posed 
it is necessary to create a boundary layer in order to apply the doubling of variables
method up to the boundary. This is done by means of an additional truncation as follows.   

Let $\kappa>0$ and $\delta\in (0,\frac{\kappa}{2})$ be fixed. We then 
choose a smooth and convex map $\zeta\colon\R\to [0,\infty)$
such that $\zeta(r)=0$ for $r\leq 0$, $\zeta(r)=r-1$ for
$r\geq 2$ and $\zeta(r)\leq r\vee 0$ for all $r\in\R$. 
Define $\zeta_\delta(r):=\delta\zeta(\frac{r}{\delta})$.
There is a constant $C>0$ independent of $\delta$ such that
\begin{align}
\label{eq:auxUniformBound1}
&\sup_{r\in\R} |(\zeta_\delta)'(r)| + |r|(\zeta_\delta)''(r) \leq C,
\\\label{eq:auxConvergence1}
&(\zeta_\delta)'(r) \to \mathrm{sign}_+(r) \text{ and }
\zeta_\delta(r) \nearrow r_+ := \max\{r,0\} \text{ as } \delta\to 0,
\\\label{auxConvergence21}
&|\zeta_\delta(r){-}r_+|\leq C\delta\text{ for all } r\in\R,\text{ and }
\zeta_\delta(r) = r-\delta\text{ for all } r\geq 2\delta.
\end{align}

Let $\zeta^\delta_\kappa(r):= \kappa + \zeta_\delta(r{-}\kappa)$.
The idea then is to study basically $L^1$ contraction for the sequence $\zeta^\varepsilon_\kappa\circ u^{\leftarrow}_\varepsilon$
for $\varepsilon\leq\frac{\kappa}{2}$. Note that for each time slice $t\in [0,T^*]$ the function 
$\zeta^\varepsilon_\kappa\circ u^{\leftarrow}_\varepsilon(\cdot,t)$ is $C^\infty$ in the 
open domain $D-\nu(B_t{-}B_t^\varepsilon)$, and for
$\varepsilon\leq\frac{\kappa}{2}$ it is actually constant in a neighborhood of $\partial D-\nu(B_t{-}B_t^\varepsilon)$.
In particular, by introducing such a boundary layer there is no jump of the Neumann data for 
$\zeta^\varepsilon_\kappa\circ u^{\leftarrow}_\varepsilon$ across the boundary
$\partial D-\nu(B_t{-}B_t^\varepsilon)$. This turns out to be absolutely essential in order to 
apply the doubling of variables technique up to the boundary.

\begin{proposition}[Contraction principle up to time-dependent shift for truncated
Wong--Zakai type approximations]
\label{prop:L1contraction}
For each $\varepsilon > 0$ let $u_\varepsilon$ denote the unique classical solution to 
\emph{\eqref{eq:densityPME2}--\eqref{eq:densityPMElateral2}} in the sense of 
Lemma~\ref{lem:classicalSolution}. We extend $u_\varepsilon$ to a function defined on $\R^d\times [0,T^*]$
by setting it equal to $\varepsilon$ outside of $D\times [0,T^*]$. Denoting this
extension again by $u_\varepsilon$ we then define for all $\varepsilon >0$ the shifted densities
\begin{align}\label{eq:shift}
u^\leftarrow_\varepsilon(x,t) := u_\varepsilon(x{+}\nu(B_t{-}B^\varepsilon_t),t),
\quad (x,t)\in \R^d\times [0,T^*].
\end{align}

Then for all $\kappa\in (0,T^*\wedge 1)$ and all compact sets $K\subset D$ 
there exists a constant $\bar C>0$ independent of $\kappa$ and $K$, some $\vartheta>0$ 
and some $\varepsilon_0=\varepsilon_0(\kappa,K)$ such that for all 
$\varepsilon,\hat\varepsilon\leq\varepsilon_0$ it holds
\begin{equation}
\begin{aligned}
\label{eq:L1contraction}
&\sup_{T\in [\kappa,T^*]} \E\int_K 
|\kappa\vee u^\leftarrow_\varepsilon(x,T){-}\kappa\vee u^\leftarrow_{\hat\varepsilon}(x,T)|\dx
\\&
\leq \bar C(\varepsilon\vee\hat\varepsilon)^\vartheta + \bar C\kappa
+\E\int_D |\kappa\vee(u_0(x){+}\varepsilon){-}\kappa\vee(u_0(x){+}\hat\varepsilon)|\dx.
\end{aligned}
\end{equation}
\end{proposition}

We may lift this quantitative stability estimate for the truncated and shifted densities 
$\kappa\vee u^\leftarrow_\varepsilon$ to qualitative $L^1$ convergence of the
shifted densities $u^\leftarrow_\varepsilon$.

\begin{corollary}
\label{cor:L1convergence}
Let the assumptions and notation of Proposition~\ref{prop:L1contraction} be in place.
Then the sequence of shifted densities $u_\varepsilon^\leftarrow$
is Cauchy in $C([\tau,T^*];L^1(\Omega{\times}D,\Prob\otimes\mathcal{L}^d))$
for all positive times $\tau \in (0,T^*)$. The sequence of shifted densities is 
moreover Cauchy in $L^1([0,T^*];L^1(\Omega{\times}D,\Prob\otimes\mathcal{L}^d))$.
Let $u\in L^1([0,T^*];L^1(\Omega{\times}D,\Prob\otimes\mathcal{L}^d))$ 
denote the corresponding limit in this space. Then, $u$ is also the weak limit
of the Wong--Zakai type approximations $u_\varepsilon$ from Lemma~\ref{lem:classicalSolution}
in the space $L^{m+1}(\Omega_{T^*},\mathcal{P}_{T^*};L^{m+1}(D))$
\end{corollary}

The final issue concerning the Wong--Zakai type approximations $u_\varepsilon$
from Lemma~\ref{lem:classicalSolution} is the identification of the limit object 
$u$ in Corollary~\ref{cor:L1convergence} as the unique weak solution to the Cauchy--Dirichlet 
problem \eqref{eq:SPME}--\eqref{eq:lateralBC} of the stochastic porous medium equation.

\begin{proposition}
\label{prop:weakSolutionByWongZakaiApproximation}
Let $u\in L^{m+1}(\Omega_{T^*},\mathcal{P}_{T^*};L^{m+1}(D))$ 
be the limit of the Wong--Zakai approximations $u_\varepsilon$ in the sense of Corollary~\ref{cor:L1convergence}.
Then it holds that $u\in\Hsp^{-1}_{m+1}(D)$, and $u$ is the unique weak solution 
to the Cauchy--Dirichlet problem \emph{\eqref{eq:SPME}--\eqref{eq:lateralBC}} 
with initial density $u_0\in C^\infty_{\mathrm{cpt}}(D;[0,\infty))$
in the sense of Definition~\ref{def:weakSolutions}. Moreover, $u$ satisfies the bounds
$0 \leq u \leq \|u_0\|_{L^\infty(D)}$.
\end{proposition}

From now on we will restrict ourselves to the one-dimensional setting $d=1$. 
We still have to make rigorous the outlined comparison argument.
The key ingredient for this will be provided in Section~\ref{sec:maximalSubsolution}.
It consists of the construction of a \textit{maximal subsolution} $\bar p_{\mathrm{max}}$ (in the sense of 
viscosity theory~\cite{Crandall:1992}, precise definitions will follow in Section~\ref{sec:maximalSubsolution}) 
for the Cauchy--Dirichlet problem of the porous medium equation after stochastic flow transformation
\begin{align}
\label{eq:pressPMErough}
\partial_t \bar p &= (m{-}1)\bar p\partial_{xx} \bar p + |\partial_x \bar p|^2,
\quad (x,t) \in \bigcup_{t\in (0,T^*)}(\nu B_t{+}I)\times\{t\}, 
\\
\label{eq:pressPMEic} 
\bar p(x,0) &= \bar p_0(x),
\quad\hspace{2.98cm} x\in \bar{I}, 
\\
\label{eq:pressPMElateral} 
\bar p(x,t) &= 0, 
\quad\hspace{3.06cm} (x,t)\in \bigcup_{t\in (0,T^*]}(\nu B_t{+}\partial I)\times\{t\}.
\end{align}
The \textit{pressure variable} $p$ is obtained from the density variable $u$ via the transformation
$g(u)$ with $g\colon [0,\infty)\to [0,\infty)$ given by $g(r):=\frac{m}{m-1}r^{m-1}$. To the best
of the author's knowledge, the present work seems to be the first instance to make use of the pressure
formulation to study the stochastic porous medium equation after stochastic flow transformation in the
setting of viscosity theory.

The main difficulty for solving \eqref{eq:pressPMErough}--\eqref{eq:pressPMElateral}  
comes from the fact that the lateral boundary consists of two translates of a 
Brownian trajectory. To overcome the lack of regularity of the lateral boundary,
we first consider the approximate initial-boundary value problem (see Lemma~\ref{lem:subsolutionPressureApprox})
\begin{align}
\label{eq:pressPMErough2}
\partial_t \bar p_\varepsilon &= (m{-}1)\bar p_\varepsilon\partial_{xx} \bar p_\varepsilon 
+ |\partial_{x} \bar p_\varepsilon|^2,
\quad (x,t) \in \bigcup_{t\in (0,T^*)}(\nu B^\varepsilon_t{+}I)\times\{t\}, 
\\
\label{eq:pressPMEic2} 
\bar p_\varepsilon(x,0) &= \bar p_{0,\varepsilon}(x),
\quad\hspace{3.21cm} x\in\bar{I}, 
\\
\label{eq:pressPMElateral2} 
\bar p_\varepsilon(x,t) &= \frac{m}{m-1}\varepsilon^{m-1}, 
\quad\hspace{1.90cm} (x,t)\in  \bigcup_{t\in (0,T^*]}(\nu B^\varepsilon_t{+}\partial I)\times\{t\},
\end{align}
and then pass to the limit $\varepsilon\to 0$ by means of the technique of
semi-relaxed limits, see for instance~\cite[Section~6]{Crandall:1992}. 
In this way we obtain a maximal subsolution to the problem \eqref{eq:pressPMErough}--\eqref{eq:pressPMElateral} 
in the sense of viscosity theory~\cite{Crandall:1992}, see Proposition~\ref{prop:maxSubsolution}. 
The main motivation for working in the framework of viscosity theory is the
non-monotonicity of the underlying space-time domains $\bigcup_{t\in (0,T^*)}(\nu B^\varepsilon_t{+}I)\times\{t\}$
as $\varepsilon\to 0$ which necessitates the usage of a relaxed limit. In this way, however,
the interpretation of the lateral boundary condition \eqref{eq:pressPMElateral} in
a strong sense is lost in the limit. It is well-known that boundary regularity for solutions to 
stochastic PDEs with gradient type noise and given Dirichlet data proves to be a delicate issue. This
already shows up in the linear case, see, for instance, the works by Krylov \cite{Krylov:2003} and \cite{Krylov:2003a}.
For a recent work in the semilinear regime, we refer the reader to \cite{Gerencser:2019}.
On the other side, ``continuity up to the lateral boundary'' is inessential for our purposes and anyway not expected, 
if at all, to be obtained by the methods in this work.

Note that the maximal subsolution $\bar p_{\max}$ of course depends on the realization of Brownian motion, and
is therefore random. However, since the construction of $\bar p_{\mathrm{max}}$ is ultimately a purely 
deterministic consequence of the probabilistic facts 
\eqref{eq:uniformBoundDyadicApproximation}--\eqref{eq:supportAfterShift}, 
we obtain the maximal subsolution in a pathwise sense on a set of full probability.

The proof of Theorem~\ref{theorem:mainResult}, which is the content of 
Section~\ref{sec:proofMainResult}, then roughly speaking proceeds as follows.
Denoting by $p$ the map which we obtain from the unique weak solution $u$ of 
\eqref{eq:SPME}--\eqref{eq:lateralBC} in the sense of Definition~\ref{def:weakSolutions}
by first applying a stochastic flow transformation and then a density-to-pressure transformation,
we have the estimate $p\leq \bar p_{\mathrm{max}}$, see Proposition~\ref{prop:comparisonTransformedWeakSol}. 
This bound is essentially a combination of the following facts:
\begin{itemize}[leftmargin=0.7cm]
\item[i)] The solution to \eqref{eq:pressPMErough2}--\eqref{eq:pressPMElateral2} 
					may in fact be obtained from the Wong--Zakai type approximations $u_\varepsilon$ 
					of \eqref{eq:densityPME2}--\eqref{eq:densityPMElateral2}  
					by first applying a stochastic flow transformation and then a 
					density-to-pressure transformation.
\item[ii)] The Wong--Zakai approximations $u_\varepsilon$ (or more precisely, their
           shifted counterparts \eqref{eq:shift}) converge on each positive time slice in $L^1$ to
					 the unique weak solution of the Cauchy--Dirichlet problem \eqref{eq:SPME}--\eqref{eq:lateralBC},
					 see Corollary~\ref{cor:L1convergence}.
\item[iii)] The maximal subsolution to \eqref{eq:pressPMErough}--\eqref{eq:pressPMElateral}
						dominates the upper semi-relaxed limit (with respect to parabolic
						space-time cylinders) of the transformed $u_\varepsilon$. However,
						taking a semi-relaxed limit allows to compare with the transformed density $u$
						by means of the previous two items.
\end{itemize} 

Finally, we compare the maximal subsolution $\bar p_{\max}$ to a Barenblatt profile (written in
the pressure variable) as outlined in the heuristic argument. We make use of the comparison principle
in the framework of viscosity solutions for the deterministic porous medium equation
as developed by Caffarelli and V\'{a}zquez~\cite{Caffarelli:1999} resp.\
V\'{a}zquez and Br\"{a}ndle~\cite{Vazquez:2005}. The remaining argument after comparing with the Barenblatt profile, 
in particular the derivation of \eqref{eq:boundExtinctionTime}, then works as already sketched before.

\section{Recovering weak solutions by Wong--Zakai type approximations}
\label{sec:solutionTheory}

\subsection{Proof of Lemma~\ref{lem:classicalSolution} {\normalfont (Wong--Zakai type approximation)}}
We make use of a usual trick of avoiding the degeneracy of the porous medium operator,
see for instance~\cite[Proof of Theorem 5.5]{Vazquez:2007}. Let $\varepsilon > 0$ be fixed
and choose a bounded smooth function $a_\varepsilon\colon\R\to [m(\frac{\varepsilon}{2})^{m-1},\infty)$ such that
it holds $a_\varepsilon(r)=mr^{m-1}$ for all $r\in [\varepsilon,\varepsilon{+}\|u_0\|_{L^\infty(D)}]$.
By the choice of $a_\varepsilon$ and since $B^\varepsilon_t$ as defined in \eqref{eq:approxBrownianMotion} is
smooth on $[0,T^*]$ almost surely, we can make use of standard quasilinear theory~\cite{Ladyzhenskaya:1968} 
to solve the Cauchy--Dirichlet problem \eqref{eq:densityPME2}--\eqref{eq:densityPMElateral2} 
in a classical sense, i.e., we obtain with probability one a classical 
solution $u_\varepsilon\in C^{2,1}_{x,t}(\bar{D}{\times}[0,T^*])\cap C^\infty(D{\times}(0,T^*))$ such that
the equations \emph{\eqref{eq:densityPME2}--\eqref{eq:densityPMElateral2}} are satisfied pointwise everywhere.

We can infer from the maximum principle that $\varepsilon\leq u_\varepsilon(x,t)\leq\varepsilon+\|u_0\|_{L^\infty(D)}$ 
holds true for all $(x,t)\in \bar{D}{\times}[0,T^*]$ as it is asserted in \eqref{boundMaximumPrinciple}.
The derivation of the energy estimate \eqref{eq:energyInequalityApprox} is standard: multiply the 
equation with $u_\varepsilon$, integrate over $D$ and use the regularity of $u_\varepsilon$ to 
integrate by parts in the spatial differential operators. Note in this respect that as a consequence 
of \eqref{boundMaximumPrinciple} and \eqref{eq:densityPMElateral2} it holds $\n_{\partial D}\cdot \nabla u_\varepsilon^m\leq 0$ 
on $\partial D$, where $\n_{\partial D}$ is the exterior unit normal vector field along the $C^2$ manifold $\partial D$. 
This is the only reason for the inequality sign in \eqref{eq:energyInequalityApprox} as we may compute
for the second term
\begin{align*}
\int_D \frac{\mathrm{d}}{\mathrm{d}t}B^\varepsilon_t\cdot u_\varepsilon\nabla u_\varepsilon \dx
&= \int_D \frac{\mathrm{d}}{\mathrm{d}t}B^\varepsilon_t\cdot u_\varepsilon\nabla (u_\varepsilon{-}\varepsilon) \dx
\\&
= -\int_D \frac{\mathrm{d}}{\mathrm{d}t}B^\varepsilon_t\cdot (u_\varepsilon{-}\varepsilon)\nabla u_\varepsilon 
= -\int_D \frac{\mathrm{d}}{\mathrm{d}t}B^\varepsilon_t\cdot 
\frac{1}{2}\nabla|u_\varepsilon{-}\varepsilon|^2 = 0.
\end{align*}    

We proceed with the bound \eqref{eq:boundTimeDerivativeApprox} for the time derivative.
Multiplying the equation~\eqref{eq:densityPME2} with $\partial_t u_\varepsilon^m$, integrating
over $D$, performing an integration by parts in the term with the porous medium operator
and estimating the transport term by H\"{o}lder's and Young's inequality yields for all $t\in (0,T^*)$ 
the estimate
\begin{align*}
&\int_D mu_\varepsilon^{m-1}(t)|\partial_t u_\varepsilon|^2(t) \dx
+ \frac{\mathrm{d}}{\mathrm{d}t}\int_D \frac{1}{2} |\nabla u_\varepsilon^m|^2(t) \dx
\\&
\leq \frac{1}{2}\int_D mu_\varepsilon^{m-1}(t)|\partial_t u_\varepsilon|^2(t) \dx
+\frac{\nu^2}{2}\int_D \Big|\frac{\mathrm{d}}{\mathrm{d}t} B^\varepsilon_t\Big|^2 
mu_\varepsilon^{m-1}(t)|\nabla u_\varepsilon|^2(t) \dx.
\end{align*}
Multiplying this bound with $t$ and integrating the resulting estimate over $(0,T)$
we may infer using also \eqref{eq:energyInequalityApprox} and \eqref{boundMaximumPrinciple}
\begin{align*}
&\int_0^T\int_D \frac{t}{2}mu_\varepsilon^{m-1}|\partial_t u_\varepsilon|^2 \dx \dt
+ \frac{T}{2}\int_D |\nabla u_\varepsilon^m|^2(T) \dx
\\&
\leq\frac{1}{2}\int_0^T\int_D |\nabla u_\varepsilon^m|^2 \dx \dt
+\sup_{0\leq t \leq T^*} \Big|\frac{\mathrm{d}}{\mathrm{d}t} B^\varepsilon_t\Big|^2
\frac{T^*\nu^2}{2}\int_0^T\int_D mu_\varepsilon^{m-1}|\nabla u_\varepsilon|^2 \dx \dt
\\&
\leq\Big(\frac{m}{2}(\varepsilon{+}\|u_0\|_{L^\infty})^{m-1}
{+}\frac{T^*\nu^2}{2}\sup_{0\leq t \leq T^*} 
\Big|\frac{\mathrm{d}}{\mathrm{d}t} B^\varepsilon_t\Big|^2\Big)
\int_D\frac{1}{2}|u_\varepsilon|^2(0) \dx
\end{align*}
for all $T\in (0,T^*)$. Moreover, it follows from \eqref{eq:approxBrownianMotion} and Doob's maximal inequality
that $\E\sup_{0\leq t \leq T^*}|\frac{\mathrm{d}}{\mathrm{d}t} B^\varepsilon_t|^2\leq C\varepsilon^{-\beta}T^*\E|B_{T^*}|^2$
for some absolute constant $C>0$. This establishes the estimate \eqref{eq:boundTimeDerivativeApprox}.
The bound \eqref{eq:boundLaplacianApprox} is now a consequence of plugging in the equation \eqref{eq:densityPME2},
then using the triangle inequality, estimating the term with the time derivative by means of \eqref{eq:boundTimeDerivativeApprox} 
and bounding the transport term similarly as at the end of the proof of \eqref{eq:boundTimeDerivativeApprox}.
This concludes the proof of Lemma~\ref{lem:classicalSolution}. \qed

\subsection{Proof of Proposition~\ref{prop:L1contraction} 
{\normalfont (Contraction principle up to time-dependent shift for truncated
Wong--Zakai type approximations)}}
\label{subsec:proofContraction}
Fix $\kappa>0$ and let $\zeta_\delta$ denote the smooth and convex approximation
to the positive part truncation $r\mapsto r_+:=r\vee 0$ on scale $\delta>0$ such that 
\eqref{eq:auxUniformBound1}--\eqref{auxConvergence21} hold true. Define
$\zeta^\delta_\kappa(r):=\kappa + \zeta_\delta(r{-}\kappa)$ which is a smooth and convex
approximation to the truncation $r\mapsto r\vee\kappa$. Finally, fix $\varepsilon,\hat\varepsilon\leq\frac{\kappa}{2}$
and abbreviate for what follows $v_\varepsilon:=\zeta^{\varepsilon^q}_\kappa\circ u_{\varepsilon}^\leftarrow$
resp.\ $v_{\hat\varepsilon}:=\zeta^{\hat\varepsilon^q}_\kappa\circ u_{\hat\varepsilon}^\leftarrow$,
where $q>1$ will be a large exponent to be specified later on in the proof. 
See \eqref{eq:shift} for the definition of the shifted densities. Finally, fix a compact
set $K\subset D$.

We aim to derive an estimate for $\sup_{T\in [0,T^*]}\E\|v_\varepsilon{-}v_{\hat\varepsilon}\|_{L^1(K)}$ 
of the same type as the asserted bound \eqref{eq:L1contraction}. The proof of this proceeds in several steps.
For the sake of better readability, let us occasionally break the proof into intermediate results.

\begin{lemma}
\label{lemma:equations}
Let the assumptions and notation of Subsection~\ref{subsec:proofContraction} be in place.
We next choose a smooth, even and convex map $\eta\colon\R\to [0,\infty)$
such that $\eta(r)=|r|-1$ for $|r|\geq 2$ and $\eta(r)\leq |r|$ for all $r\in\R$. 
Define $\eta_\delta(r):=\delta\eta(\frac{r}{\delta})$.
There is a constant $C>0$ independent of $\delta$ such that
\begin{align}
\label{eq:auxUniformBound}
&\sup_{r\in\R} |(\eta_\delta)'(r)| + |r|(\eta_\delta)''(r) \leq C,
\\\label{eq:auxConvergence}
&(\eta_\delta)'(r) \to \mathrm{sign}(r) \text{ and }
\eta_\delta(r) \nearrow |r| \text{ as } \delta\to 0,
\\\label{auxConvergence2}
&|\eta_\delta(r){-}|r||\leq C\delta\text{ for all } r\in\R,\text{ and }
\eta_\delta(r) = |r|-\delta\text{ for all } |r|\geq 2\delta.
\end{align}
Then the following ``entropy estimate'' holds true
\begin{align}
\nonumber
&-\int_0^{T^*}\int_{\R^d} \eta_\delta\big(v_\varepsilon(x,t){-}v_{\hat\varepsilon}(y,s)\big)
\partial_t\xi(x,t) \dx \dt
\\&\label{eq:weakFormShiftChainRule8}
\leq-\int_0^{T^*}\int_{\R^d}|\nabla v_\varepsilon^m(x,t)|^2
(\eta_\delta)''\big(v_\varepsilon^m(x,t){-}v_{\hat\varepsilon}^m(y,s)\big)
\xi(x,t) \dx \dt
\\&~~~\nonumber
+\int_0^{T^*}\int_{\R^d}\eta_\delta\big(v_\varepsilon^m(x,t){-}v_{\hat\varepsilon}^m(y,s)\big)
\Delta\xi(x,t) \dx \dt
\\&~~~\nonumber
-\int_0^{T^*}\int_{\R^d}\frac{1}{2}\nu^2|\nabla v_\varepsilon(x,t)|^2
(\eta_\delta)''\big(v_\varepsilon(x,t){-}v_{\hat\varepsilon}(y,s)\big)
\xi(x,t) \dx \dt
\\&~~~\nonumber
+\int_0^{T^*}\int_{\R^d}\frac{1}{2}\nu^2
\eta_\delta\big(v_\varepsilon(x,t){-}v_{\hat\varepsilon}(y,s)\big)
\Delta\xi(x,t) \dx \dt
\\&~~~\nonumber
-\int_0^{T^*}\int_{\R^d} \nu \eta_\delta\big(v_\varepsilon(x,t){-}v_{\hat\varepsilon}(y,s)\big)
\nabla\xi(x,t) \dx \,\mathrm{d}B_t
\\&~~~\nonumber
+\int_0^{T^*}\int_{\R^d} \Delta v_\varepsilon^m(x,t)\xi(x,t) 
\\&~~~~~~~~~~~~~~~~~~\nonumber
\times\big\{(\eta_\delta)'\big(v_\varepsilon(x,t){-}v_{\hat\varepsilon}(y,s)\big)
-(\eta_\delta)'\big(v_\varepsilon^m(x,t){-}v_{\hat\varepsilon}^m(y,s)\big)
\big\}\dx \dt
\\&~~~\nonumber
-\int_0^{T^*}\int_{\R^d}m\big((u_\varepsilon^\leftarrow)^{m-1}{-}v_\varepsilon^{m-1}\big)(x,t)
\\&~~~~~~~~~~~~~~~~~~\nonumber
\times\nabla v_\varepsilon(x,t)\cdot(\eta_\delta)'\big(v_\varepsilon(x,t){-}v_{\hat\varepsilon}(y,s)\big) 
\nabla\xi(x,t)\dx \dt
\end{align}
for all $\xi\in C^\infty_{\mathrm{cpt}}(\R^d\times (0,T^*);[0,\infty))$ and all $(y,s)\in \R^d{\times}(0,T^*)$.
A corresponding estimate holds true switching the roles of $v_\varepsilon$ and $v_{\hat\varepsilon}$,
see~\eqref{eq:weakFormShiftChainRule9} below.
\end{lemma}

\begin{proof}
\textit{Step 1 (Equation for $u_{\varepsilon}^\leftarrow$ and $u_{\hat\varepsilon}^\leftarrow$):} 
The first step is to derive the equation for the shifted densities
$u_{\varepsilon}^\leftarrow$ and $u_{\hat\varepsilon}^\leftarrow$, respectively.
To this end, we aim to apply It\^{o}'s formula
with respect to $\int_{\R^d}  u_\varepsilon(x,t)\eta(x{-}\nu(B_t{-}B^\varepsilon_t),t)\dx$
for each test function $\eta\in C^\infty_{\mathrm{cpt}}(\R^d\times (0,T^*))$.
Note that $\partial_t u_\varepsilon\equiv 0$ on the lateral boundary $\partial D\times (0,T^*)$.
Hence, it holds $\partial_tu_\varepsilon\in C(\R^d\times(0,T^*))$ and we thus obtain from
an application of It\^{o}'s formula for each test function $\eta\in C^\infty_{\mathrm{cpt}}(D\times (0,T^*))$
with probability one
\begin{align}
\nonumber
&-\int_0^{T^*}\int_{\R^d} u_\varepsilon(x,t)\partial_t\eta(x{-}\nu(B_t{-}B^\varepsilon_t),t) \dx \dt
\\&\label{eq:weakFormShift1}
= \int_0^{T^*}\int_{\R^d} \partial_tu_\varepsilon(x,t)\eta(x{-}\nu(B_t{-}B^\varepsilon_t),t) \dx \dt
\\&~~~\nonumber
+\int_0^{T^*}\int_{\R^d} u_\varepsilon(x,t)\frac{\mathrm{d}}{\mathrm{d}t}B^\varepsilon_t\cdot
\nabla\eta(x{-}\nu(B_t{-}B^\varepsilon_t),t) \dx \dt
\\&~~~\nonumber
+\int_0^{T^*}\int_{\R^d} u_\varepsilon(x,t)\frac{1}{2}\nu^2
\Delta\eta(x{-}\nu(B_t{-}B^\varepsilon_t),t) \dx \dt
\\&~~~\nonumber
-\int_0^{T^*}\int_{\R^d} u_\varepsilon(x,t)\nu
\nabla\eta(x{-}\nu(B_t{-}B^\varepsilon_t),t) \dx \,\mathrm{d}B_t.
\end{align}
Since $u_\varepsilon$ solves the equation \eqref{eq:densityPME2} classically in $D\times (0,T^*)$,
and is by definition constant outside of it, we may compute
\begin{equation}
\begin{aligned}
\label{eq:weakFormShift2}
&\int_0^{T^*}\int_{\R^d} \partial_tu_\varepsilon(x,t)\eta(x{-}\nu(B_t{-}B^\varepsilon_t),t) \dx \dt
\\&
=\int_0^{T^*}\int_{\R^d\setminus\partial D} \Delta u_\varepsilon^m(x,t)\eta(x{-}\nu(B_t{-}B^\varepsilon_t),t) \dx \dt
\\&~~~
+\int_0^{T^*}\int_{\R^d\setminus\partial D} \eta(x{-}\nu(B_t{-}B^\varepsilon_t),t)
\frac{\mathrm{d}}{\mathrm{d}t}B^\varepsilon_t\cdot\nabla u_\varepsilon(x,t) \dx \dt.
\end{aligned}
\end{equation}
Integrating by parts in the second term on the right hand side of the latter identity
as well as performing a change of variables $x\mapsto x+\nu(B_t{-}B^\varepsilon_t)$ yields
\begin{equation}
\begin{aligned}
\label{eq:weakFormShift3}
&\int_0^{T^*}\int_{\R^d} \partial_tu_\varepsilon(x,t)\eta(x{-}\nu(B_t{-}B^\varepsilon_t),t) \dx \dt
\\&
+\int_0^{T^*}\int_{\R^d} u_\varepsilon(x,t)\frac{\mathrm{d}}{\mathrm{d}t}B^\varepsilon_t\cdot
\nabla\eta(x{-}\nu(B_t{-}B^\varepsilon_t),t) \dx \dt
\\&
=\int_0^{T^*}\int_{\R^d\setminus({-}\nu(B_t{-}B^\varepsilon_t)+\partial D)} 
\Delta (u_{\varepsilon}^\leftarrow)^m(x,t)\eta(x,t) \dx \dt.
\end{aligned}
\end{equation}
Note that there is no boundary integral appearing from the integration by parts
in the second term on the right hand side of \eqref{eq:weakFormShift2} since 
$u_\varepsilon\in C(\R^d\times(0,T^*))$. We compute analogously
\begin{equation}
\begin{aligned}
\label{eq:weakFormShift4}
&\int_0^{T^*}\int_{\R^d} u_\varepsilon(x,t)\frac{1}{2}\nu^2
\Delta\eta(x{-}\nu(B_t{-}B^\varepsilon_t),t) \dx \dt
\\&
=-\int_0^{T^*}\int_{\R^d\setminus({-}\nu(B_t{-}B^\varepsilon_t)+\partial D)} 
\frac{1}{2}\nu^2\nabla u_{\varepsilon}^\leftarrow(x,t)\cdot\nabla\eta(x,t) \dx \dt
\end{aligned}
\end{equation}
and
\begin{equation}
\begin{aligned}
\label{eq:weakFormShift100}
&-\int_0^{T^*}\int_{\R^d} u_\varepsilon(x,t)\nu
\nabla\eta(x{-}\nu(B_t{-}B^\varepsilon_t),t) \dx \,\mathrm{d}B_t.
\\&
=\int_0^{T^*}\int_{\R^d\setminus({-}\nu(B_t{-}B^\varepsilon_t)+\partial D)} 
\nu\nabla u_{\varepsilon}^\leftarrow(x,t) \eta(x,t) \dx \,\mathrm{d}B_t.
\end{aligned}
\end{equation}
From \eqref{eq:weakFormShift1}, \eqref{eq:weakFormShift3}, \eqref{eq:weakFormShift4} and \eqref{eq:weakFormShift100}
we infer that for each $\eta\in C^\infty_{\mathrm{cpt}}(\R^d\times (0,T^*))$ it holds with probability one 
\begin{equation}
\begin{aligned}
\label{eq:weakFormShift5}
&-\int_0^{T^*}\int_{\R^d} u_{\varepsilon}^\leftarrow(x,t)\partial_t\eta(x,t) \dx \dt
\\&
=\int_0^{T^*}\int_{\R^d\setminus({-}\nu(B_t{-}B^\varepsilon_t)+\partial D)} 
\Delta (u_{\varepsilon}^\leftarrow)^m(x,t)\eta(x,t) \dx \dt
\\&~~~
-\int_0^{T^*}\int_{\R^d\setminus({-}\nu(B_t{-}B^\varepsilon_t)+\partial D)} 
\frac{1}{2}\nu^2\nabla u_{\varepsilon}^\leftarrow(x,t)\cdot\nabla\eta(x,t) \dx \dt
\\&~~~
+\int_0^{T^*}\int_{\R^d} \nu \nabla u_{\varepsilon}^\leftarrow(x,t)\eta(x,t) \dx \,\mathrm{d}B_t.
\end{aligned}
\end{equation}
Analogously one derives the equation for $u_{\hat\varepsilon}^\leftarrow$.

\textit{Step 2 (Convex approximation to $r\mapsto r\vee\kappa$ as test function):}
In the next step we aim to derive the equation for $v_\varepsilon=\zeta^{\varepsilon^q}_\kappa\circ u_\varepsilon^\leftarrow$
based on the equation for the shifted density derived in \eqref{eq:weakFormShift5}. The idea
is to test the equation \eqref{eq:weakFormShift5} with the test function 
$\eta:=\big((\zeta_\kappa^{\varepsilon^q})'\circ u_\varepsilon^\leftarrow\big)\xi$, 
where $\xi\in C^\infty_{\mathrm{cpt}}(\R^d\times (0,T^*);[0,\infty))$
is arbitrary. However, since the shifted density $u_\varepsilon^\leftarrow$ is only H\"{o}lder continuous
in the time variable we have to regularize first. To this end,
we make use of the Steklov average $\eta_h(x,t):=\frac{1}{h}\int_{t}^{t+h}\eta(x,s)\,\mathrm{d}s$
which is an admissible test function for \eqref{eq:weakFormShift5} for all sufficiently small
$h>0$ (depending only on the support of $\xi$). Since $\partial_t\eta_h(x,t)=\frac{\eta(x,t+h){-}\eta(x,t)}{h}$
we obtain by a simple change of variables
\begin{equation}
\label{eq:aux1ChainRule}
\begin{aligned}
&-\int_0^{T^*}\int_{\R^d} u_\varepsilon^\leftarrow(x,t)\partial_t\eta_h(x,t) \dx \dt
\\&
=-\int_0^{T^*}\int_{\R^d} \frac{1}{h}
\big(u_\varepsilon^\leftarrow(x,t{-}h){-}u_\varepsilon^\leftarrow(x,t)\big)
(\zeta^{\varepsilon^q}_\kappa)'(u_\varepsilon^\leftarrow(x,t))\xi(x,t) \dx \dt
\end{aligned}
\end{equation}
for each $h>0$ and each $\xi\in C^\infty_{\mathrm{cpt}}(\R^d{\times} (0,T^*);[0,\infty))$ 
almost surely. We then deduce from the smoothness and convexity of $\zeta_\kappa^{\varepsilon^q}$
as well as by reverting the change of variables the bound
\begin{align*}
&-\int_0^{T^*}\int_{\R^d} u_\varepsilon^\leftarrow(x,t)\partial_t\eta_h(x,t) \dx \dt 
\\&
\geq -\int_0^{T^*}\int_{\R^d} \frac{1}{h}
\big(\zeta^{\varepsilon^q}_\kappa(u_\varepsilon^\leftarrow(x,t{-}h))
-\zeta^{\varepsilon^q}_\kappa(u_\varepsilon^\leftarrow(x,t))\big)
\xi(x,t) \dx \dt
\\&
= -\int_0^{T^*}\int_{\R^d} \frac{1}{h}
\big(\xi(x,t{+}h){-}\xi(x,t)\big)
\zeta^{\varepsilon^q}_\kappa(u_\varepsilon^\leftarrow(x,t)) \dx \dt
\end{align*}
for each $h>0$ and each $\xi\in C^\infty_{\mathrm{cpt}}(D{\times} (0,T^*);[0,\infty))$ 
almost surely. Hence, we may infer from this, \eqref{eq:weakFormShift5} and standard properties 
of the Steklov average after letting $h\to 0$ the estimate 
\begin{align}
\nonumber
&-\int_0^{T^*}\int_{\R^d} \zeta^{\varepsilon^q}_\kappa(u_\varepsilon^\leftarrow(x,t))\partial_t\xi(x,t) \dx \dt
\\&\label{eq:weakFormShiftChainRule}
\leq\int_0^{T^*}\int_{\R^d\setminus({-}\nu(B_t{-}B^\varepsilon_t)+\partial D)} 
\Delta (u_\varepsilon^\leftarrow)^m(x,t)(\zeta^{\varepsilon^q}_\kappa)'
(u_\varepsilon^\leftarrow(x,t))\xi(x,t) \dx \dt
\\&~~~\nonumber
-\int_0^{T^*}\int_{\R^d\setminus({-}\nu(B_t{-}B^\varepsilon_t)+\partial D)} 
\frac{1}{2}\nu^2\nabla u_\varepsilon^\leftarrow(x,t)\cdot\nabla
\big((\zeta^{\varepsilon^q}_\kappa)'(u_\varepsilon^\leftarrow(x,t))\xi(x,t)\big) \dx \dt
\\&~~~\nonumber
+\int_0^{T^*}\int_{\R^d} \nu\nabla u_\varepsilon^\leftarrow(x,t)
(\zeta^{\varepsilon^q}_\kappa)'(u_\varepsilon^\leftarrow(x,t))\xi(x,t) \dx \,\mathrm{d}B_t,
\end{align}
which is valid for each $\xi\in C^\infty_{\mathrm{cpt}}(\R^d{\times} (0,T^*);[0,\infty))$ 
on a set with probability one. Note that on $\R^d\setminus({-}\nu(B_t{-}B^\varepsilon_t)+\partial D)$
we may apply the chain rule to compute that $\nabla u_\varepsilon^\leftarrow(x,t)
(\zeta^{\varepsilon^q}_\kappa)'(u_\varepsilon^\leftarrow(x,t))
=\nabla(\zeta^{\varepsilon^q}_\kappa\circ u_\varepsilon^\leftarrow)(x,t)$. Furthermore, note that
$(\zeta^{\varepsilon^q}_\kappa)'(r)=0$ holds true for all $r\leq\kappa$. In particular, because of 
$u_\varepsilon^\leftarrow\in C(\R^d{\times}[0,T^*])$ and the choice $\varepsilon\leq\frac{\kappa}{2}$
we have for all $t\in [0,T^*]$ that 
$(\zeta^{\varepsilon^q}_\kappa)'\circ u_\varepsilon^\leftarrow(\cdot,t) \equiv 0$ in a
neighborhood of the interface $({-}\nu(B_t{-}B^\varepsilon_t)+\partial D)$.
This in turn means that we can integrate by parts in the first term
on the right hand side of \eqref{eq:weakFormShiftChainRule} without producing an additional surface integral.
Taking all of these information together yields the bound
\begin{align}
\nonumber
&-\int_0^{T^*}\int_{\R^d} (\zeta^{\varepsilon^q}_\kappa\circ u_\varepsilon^\leftarrow)\partial_t\xi \dx \dt
\\&\label{eq:weakFormShiftChainRule2}
\leq-\int_0^{T^*}\int_{\R^d}m(u_\varepsilon^\leftarrow)^{m-1}|\nabla u_\varepsilon^\leftarrow|^2
\big((\zeta^{\varepsilon^q}_\kappa)''\circ u_\varepsilon^\leftarrow\big)\xi \dx \dt
\\&~~~\nonumber
-\int_0^{T^*}\int_{\R^d}m(u_\varepsilon^\leftarrow)^{m-1}
\nabla(\zeta^{\varepsilon^q}_\kappa\circ u_\varepsilon^\leftarrow)
\cdot\nabla\xi \dx \dt
\\&~~~\nonumber
-\int_0^{T^*}\int_{\R^d}\frac{1}{2}\nu^2|\nabla u_\varepsilon^\leftarrow|^2
\big((\zeta^{\varepsilon^q}_\kappa)''\circ u_\varepsilon^\leftarrow\big)\xi \dx \dt
\\&~~~\nonumber
-\int_0^{T^*}\int_{\R^d}\frac{1}{2}\nu^2\nabla(\zeta^{\varepsilon^q}_\kappa\circ u_\varepsilon^\leftarrow)
\cdot\nabla\xi \dx \dt
\\&~~~\nonumber
+\int_0^{T^*}\int_{\R^d} \nu
\nabla(\zeta^{\varepsilon^q}_\kappa\circ u_\varepsilon^\leftarrow) 
\xi \dx \,\mathrm{d}B_t
\end{align}
for each $\xi\in C^\infty_{\mathrm{cpt}}(\R^d{\times} (0,T^*);[0,\infty))$ almost surely.
Exploiting the sign $(\zeta^{\varepsilon^q}_\kappa)''\geq 0$ and making use of the abbreviation
$v_\varepsilon=\zeta^{\varepsilon^q}_\kappa\circ u_\varepsilon^\leftarrow$ we arrive at the estimate
\begin{align}
\nonumber
&-\int_0^{T^*}\int_{\R^d} v_\varepsilon(x,t)\partial_t\xi(x,t) \dx \dt
\\&\label{eq:weakFormShiftChainRule3}
\leq-\int_0^{T^*}\int_{\R^d}\nabla v_\varepsilon^m(x,t) \cdot\nabla\xi(x,t) \dx \dt
\\&~~~\nonumber
-\int_0^{T^*}\int_{\R^d}\frac{1}{2}\nu^2\nabla v_\varepsilon(x,t) \cdot\nabla\xi(x,t) \dx \dt
\\&~~~\nonumber
+\int_0^{T^*}\int_{\R^d} \nu \nabla v_\varepsilon(x,t)\xi(x,t) \dx \,\mathrm{d}B_t
\\&~~~\nonumber
-\int_0^{T^*}\int_{\R^d}m\big((u_\varepsilon^\leftarrow)^{m-1}{-}v_\varepsilon^{m-1}\big)(x,t)
\nabla v_\varepsilon(x,t)\cdot\nabla\xi(x,t) \dx \dt
\end{align}
which is valid for all $\xi\in C^\infty_{\mathrm{cpt}}(\R^d{\times} (0,T^*);[0,\infty))$ almost surely.
An analogous estimate also holds true for the pair $(u_{\hat\varepsilon}^\leftarrow,
v_{\hat\varepsilon}=\zeta^{\hat\varepsilon^q}_\kappa\circ u_{\hat\varepsilon}^\leftarrow)$, i.e.,
\begin{align}
\nonumber
&-\int_0^{T^*}\int_{\R^d} v_{\hat\varepsilon}(y,s)\partial_t\tilde\xi(y,s) \dy \ds
\\&\label{eq:weakFormShiftChainRule4}
\leq-\int_0^{T^*}\int_{\R^d}\nabla v_{\hat\varepsilon}^m(y,s) \cdot\nabla\tilde\xi(y,s) \dy \ds
\\&~~~\nonumber
-\int_0^{T^*}\int_{\R^d}\frac{1}{2}\nu^2\nabla v_{\hat\varepsilon}(y,s) \cdot\nabla\tilde\xi(y,s) \dy \ds
\\&~~~\nonumber
+\int_0^{T^*}\int_{\R^d} \nu \nabla v_{\hat\varepsilon}(y,s)\tilde\xi(y,s) \dy \,\mathrm{d}B_s
\\&~~~\nonumber
-\int_0^{T^*}\int_{\R^d}m\big((u_{\hat\varepsilon}^\leftarrow)^{m-1}{-}v_{\hat\varepsilon}^{m-1}\big)(y,s)
\nabla v_{\hat\varepsilon}(y,s)\cdot\nabla\tilde\xi(y,s) \dy \ds
\end{align}
for all $\tilde\xi\in C^\infty_{\mathrm{cpt}}(D{\times} (0,T^*);[0,\infty))$ almost surely.

\textit{Step 3 (Convex approximation $\eta_\delta$ to $r\mapsto |r|$ as test function):}
We proceed by testing the inequality \eqref{eq:weakFormShiftChainRule3} with
$(\eta_\delta)'(v_\varepsilon(x,t){-}v_{\hat\varepsilon}(y,s))\xi(x,t)$ 
where the test function $\xi\in C^\infty_{\mathrm{cpt}}(\R^d\times (0,T^*);[0,\infty))$ 
and $(y,s)\in \R^d{\times}(0,T^*)$ are arbitrary. This again incorporates
several integration by parts which we do not want to produce any additional
surface integrals. We reiterate that this will indeed not be the case
since neither the Dirichlet data nor the Neumann data for $v_\varepsilon$
jump across the interfaces ${-}\nu(B_t{-}B^\varepsilon_t)+\partial D$ for all $t\in [0,T^*]$.
Hence, arguing similar to the previous step using in particular the Steklov average 
and the convexity of $\eta_\delta$ we obtain the estimate
\begin{align}
\nonumber
&-\int_0^{T^*}\int_{\R^d} \eta_\delta\big(v_\varepsilon(x,t){-}v_{\hat\varepsilon}(y,s)\big)
\partial_t\xi(x,t) \dx \dt
\\&\label{eq:weakFormShiftChainRule5}
\leq-\int_0^{T^*}\int_{\R^d}\nabla v_\varepsilon^m(x,t) \cdot\nabla
\big((\eta_\delta)'\big(v_\varepsilon(x,t){-}v_{\hat\varepsilon}(y,s)\big)\xi(x,t)\big)  \dx \dt
\\&~~~\nonumber
-\int_0^{T^*}\int_{\R^d}\frac{1}{2}\nu^2\nabla v_\varepsilon(x,t) \cdot\nabla
\big((\eta_\delta)'(v_\varepsilon(x,t){-}v_{\hat\varepsilon}(y,s))\xi(x,t)\big) \dx \dt
\\&~~~\nonumber
-\int_0^{T^*}\int_{\R^d} \nu \eta_\delta\big(v_\varepsilon(x,t){-}v_{\hat\varepsilon}(y,s)\big)
\nabla\xi(x,t) \dx \,\mathrm{d}B_t
\\&~~~\nonumber
-\int_0^{T^*}\int_{\R^d}m\big((u_\varepsilon^\leftarrow)^{m-1}{-}v_\varepsilon^{m-1}\big)(x,t)
\\&~~~~~~~~~~~~~~~~~~\nonumber
\times|\nabla v_\varepsilon(x,t)|^2(\eta_\delta)''
\big(v_\varepsilon(x,t){-}v_{\hat\varepsilon}(y,s)\big)\xi(x,t)\dx \dt
\\&~~~\nonumber
-\int_0^{T^*}\int_{\R^d}m\big((u_\varepsilon^\leftarrow)^{m-1}{-}v_\varepsilon^{m-1}\big)(x,t)
\\&~~~~~~~~~~~~~~~~~~\nonumber
\times\nabla v_\varepsilon(x,t)\cdot(\eta_\delta)'\big(v_\varepsilon(x,t){-}v_{\hat\varepsilon}(y,s)\big) 
\nabla\xi(x,t)\dx \dt
\end{align}
for all $\xi\in C^\infty_{\mathrm{cpt}}(\R^d\times (0,T^*);[0,\infty))$ 
and all $(y,s)\in \R^d{\times}(0,T^*)$. As a preparation for what follows,
we postprocess the right hand side of the latter inequality. Integrating
by parts, adding zero and using the chain rule we may rewrite the non-linear 
diffusion term as follows
\begin{align}
\nonumber
&-\int_0^{T^*}\int_{\R^d}\nabla v_\varepsilon^m(x,t) \cdot\nabla
\big((\eta_\delta)'\big(v_\varepsilon(x,t){-}v_{\hat\varepsilon}(y,s)\big)\xi(x,t)\big)  \dx \dt
\\&\label{eq:weakFormShiftChainRule6}
=-\int_0^{T^*}\int_{\R^d}|\nabla v_\varepsilon^m(x,t)|^2
(\eta_\delta)''\big(v_\varepsilon^m(x,t){-}v_{\hat\varepsilon}^m(y,s)\big)
\xi(x,t) \dx \dt
\\&~~~\nonumber
+\int_0^{T^*}\int_{\R^d}\eta_\delta\big(v_\varepsilon^m(x,t){-}v_{\hat\varepsilon}^m(y,s)\big)
\Delta\xi(x,t) \dx \dt
\\&~~~\nonumber
+\int_0^{T^*}\int_{\R^d} \Delta v_\varepsilon^m(x,t)\xi(x,t)
\\&~~~~~~~~~~~~~~~~~~\nonumber
\times
\big\{(\eta_\delta)'\big(v_\varepsilon(x,t){-}v_{\hat\varepsilon}(y,s)\big)
-(\eta_\delta)'\big(v_\varepsilon^m(x,t){-}v_{\hat\varepsilon}^m(y,s)\big)
\big\} \dx \dt.
\end{align}
Analogously one obtains for the Stratonovich correction term
\begin{align}
\nonumber
&-\int_0^{T^*}\int_{\R^d}\frac{1}{2}\nu^2\nabla v_\varepsilon(x,t) \cdot\nabla
\big((\eta_\delta)'(v_\varepsilon(x,t){-}v_{\hat\varepsilon}(y,s))\xi(x,t)\big) \dx \dt
\\&\label{eq:weakFormShiftChainRule7}
=-\int_0^{T^*}\int_{\R^d}\frac{1}{2}\nu^2|\nabla v_\varepsilon(x,t)|^2
(\eta_\delta)''\big(v_\varepsilon(x,t){-}v_{\hat\varepsilon}(y,s)\big)
\xi(x,t) \dx \dt
\\&~~~\nonumber
+\int_0^{T^*}\int_{\R^d}\frac{1}{2}\nu^2
\eta_\delta\big(v_\varepsilon(x,t){-}v_{\hat\varepsilon}(y,s)\big)
\Delta\xi(x,t)\dx \dt.
\end{align}

Since $\chi_{\supp\nabla v_\varepsilon}v_\varepsilon\leq \chi_{\supp\nabla v_\varepsilon}u_\varepsilon^\leftarrow$ 
by the definition of $v_\varepsilon=\zeta^{\varepsilon^q}_\kappa\circ u_\varepsilon^\leftarrow$ 
and the truncation $\zeta^{\varepsilon^q}_\kappa$ we observe that the penultimate term in \eqref{eq:weakFormShiftChainRule5}
has a favorable sign. Together with the two identities \eqref{eq:weakFormShiftChainRule6} and
\eqref{eq:weakFormShiftChainRule7} the bound \eqref{eq:weakFormShiftChainRule5} thus yields
the asserted estimate~\eqref{eq:weakFormShiftChainRule8}.

Testing the inequality \eqref{eq:weakFormShiftChainRule4} with
$(\eta_\delta)'(v_{\hat\varepsilon}(y,s){-}v_{\varepsilon}(x,t))\tilde\xi(y,s)$, 
where the test function $\tilde\xi\in C^\infty_{\mathrm{cpt}}(\R^d\times (0,T^*);[0,\infty))$ 
and $(x,t)\in \R^d{\times}(0,T^*)$ are arbitrary, yields along the same lines the bound
\begin{align}
\nonumber
&-\int_0^{T^*}\int_{\R^d} \eta_\delta\big(v_{\hat\varepsilon}(y,s){-}v_{\varepsilon}(x,t)\big)
\partial_t\tilde\xi(y,s) \dy \ds
\\&\label{eq:weakFormShiftChainRule9}
\leq-\int_0^{T^*}\int_{\R^d}|\nabla v_{\hat\varepsilon}^m(y,s)|^2
(\eta_\delta)''\big(v_{\hat\varepsilon}^m(y,s){-}v_{\varepsilon}^m(x,t)\big)
\tilde\xi(y,s) \dy \ds
\\&~~~\nonumber
+\int_0^{T^*}\int_{\R^d}\eta_\delta\big(v_{\hat\varepsilon}^m(y,s){-}v_{\varepsilon}^m(x,t)\big)
\Delta\tilde\xi(y,s) \dy \ds
\\&~~~\nonumber
-\int_0^{T^*}\int_{\R^d}\frac{1}{2}\nu^2|\nabla v_{\hat\varepsilon}(y,s)|^2
(\eta_\delta)''\big(v_{\hat\varepsilon}(y,s){-}v_{\varepsilon}(x,t)\big)
\tilde\xi(y,s) \dy \ds
\\&~~~\nonumber
+\int_0^{T^*}\int_{\R^d}\frac{1}{2}\nu^2
\eta_\delta\big(v_{\hat\varepsilon}(y,s){-}v_{\varepsilon}(x,t)\big)
\Delta\tilde\xi(y,s) \dy \ds
\\&~~~\nonumber
-\int_0^{T^*}\int_{\R^d} \nu \eta_\delta\big(v_{\hat\varepsilon}(y,s){-}v_{\varepsilon}(x,t)\big)
\nabla\tilde\xi(y,s) \dy \,\mathrm{d}B_s
\\&~~~\nonumber
+\int_0^{T^*}\int_{\R^d} \Delta v_{\hat\varepsilon}^m(y,s)\tilde\xi(y,s)
\\&~~~~~~~~~~~~~~~~~~\nonumber
\times\big\{(\eta_\delta)'\big(v_{\hat\varepsilon}(y,s){-}v_{\varepsilon}(x,t)\big)
-(\eta_\delta)'\big(v_{\hat\varepsilon}^m(y,s){-}v_{\varepsilon}^m(x,t)\big)
\big\} \dy \ds
\\&~~~\nonumber
-\int_0^{T^*}\int_{\R^d}m\big((u_{\hat\varepsilon}^\leftarrow)^{m-1}{-}v_{\hat\varepsilon}^{m-1}\big)(y,s)
\\&~~~~~~~~~~~~~~~~~~\nonumber
\times\nabla v_{\hat\varepsilon}(y,s)\cdot(\eta_\delta)'\big(v_{\hat\varepsilon}(y,s){-}v_{\varepsilon}(x,t)\big) 
\nabla\tilde\xi(y,s)\dy \ds.
\end{align}
This concludes the proof of Lemma~\ref{lemma:equations}. 
\end{proof}

We continue with the proof of Proposition~\ref{prop:L1contraction}.
The next step takes care of the proper choice of test functions $\xi(x,t)$ resp.\ $\tilde\xi(y,s)$ 
in the latter two estimates. After that, we start merging them by integration over the respective 
independent variables (and the probability space) and summing the two resulting inequalities.

Consider the mollifier $\rho\in C^\infty_{\mathrm{cpt}}((0,1);[0,\infty))$ 
with $\int_\R\rho(r)\,\mathrm{d}r$ already used in \eqref{eq:approxBrownianMotion},
and define for $\tau>0$ the scaled kernel $\rho_\tau:=\frac{1}{\tau}\rho(\frac{\cdot}{\tau})$.
Let $\varphi\in C^\infty_{\mathrm{cpt}}((0,T^*);[0,1])$ and fix another even mollifier
$\gamma\in C^\infty_{\mathrm{cpt}}(B_1;[0,\infty))$ such that $\int_{B_1}\gamma(x)\dx=1$.
For $\theta>0$ let $\gamma_\theta:=\frac{1}{\theta^d}\gamma(\frac{\cdot}{\theta})$.
Now, since $K\subset D$ is compact we can find a scale $s_c\in (0,1)$ such that $K$ is contained
in $D_{s_c}:=\{x\in D\colon \mathrm{dist}(x,\partial D)>s_c\}$. Moreover, because $D$ has a $C^2$ boundary $\partial D$
there exists (cf.\ ~\cite[Lemma 5.4]{Andreianov:2007}) a sequence $(\bar\xi_h)_h$ and a constant $C=C(D)$ such that
\begin{itemize}
\item[i)] $\bar\xi_h\in H^1_0(D;[0,1])$, $\bar\xi_h=\chi_{D}$ on 
				  $\{x\in D\colon\mathrm{dist}(x,\partial D)\geq h\}$,
\item[ii)] it holds $\int_{D}\nabla\phi\cdot\nabla\bar\xi_h\dx\geq 0$ for all $\phi \in H^1_0(D;[0,\infty))$,
\item[iii)] $\supp\nabla\bar\xi_h\subset\{x\in D\colon\mathrm{dist}(x,\partial D)< h\}$, and we have the bounds
\begin{align}
\label{eq:BoundSubharmonic}
C^{-1}\leq\int_{D}|\nabla\bar\xi_h|\dx\leq C,\quad\int_{D}|\nabla\bar\xi_h|^2\dx\leq Ch^{-1}.
\end{align}
\end{itemize}
We then fix once and for all a scale $h\in (0,s_c)$, and set $\bar\xi:=\bar\xi_h$. Note that
$\bar\xi=\chi_K$ on $K$ by the choice of $s_c$ and $h$. For purely technical reasons, we actually consider
in the following a mollified version of $\bar\xi$. Let $\bar\xi_l:=\gamma_{l}*\bar\xi$ for $l>0$.

Let now $y\in \R^d$ and $s\in (0,T^*)$ be fixed. We then define the test function
\begin{align}
\label{eq:defTestFunction}
\xi(x,t,y,s) := \rho_\tau(t{-}s)\varphi\Big(\frac{t{+}s}{2}\Big)
\gamma_\theta(x{-}y)\bar\xi_l\Big(\frac{x{+}y}{2}\Big),
\quad (x,t)\in \R^d{\times}(0,T^*).
\end{align}
For this to be an admissible choice in \eqref{eq:weakFormShiftChainRule8}
we need to restrict the range of the various parameters. Assuming that
\begin{align}
\label{eq:choiceScales}
2\tau<\min\{\inf\supp\varphi,T^*{-}\sup\supp\varphi\} \text{ and }
\theta\vee l<\frac{s_c}{4} 
\end{align}
we observe $\xi(\cdot,\cdot,y,s)\in 
C^\infty_{\mathrm{cpt}}(\R^d{\times}(0,T^*);[0,\infty))$ and is thus admissible
for \eqref{eq:weakFormShiftChainRule8}. Moreover, for every $x\in \R^d$ and
$t\in (0,T^*)$ the test function
\begin{align}
\label{eq:defTestFunction2}
\tilde\xi(y,s,x,t) := \xi(x,t,y,s),\quad (y,s)\in \R^d{\times}(0,T^*)
\end{align}
then also represents an admissible choice for \eqref{eq:weakFormShiftChainRule9} under the
same restrictions \eqref{eq:choiceScales} on the parameters. We have everything in place
to merge \eqref{eq:weakFormShiftChainRule8} and \eqref{eq:weakFormShiftChainRule9}.

Testing \eqref{eq:weakFormShiftChainRule8} with the admissible test functions $\xi(\cdot,\cdot,y,s)$
from \eqref{eq:defTestFunction} for every $(y,s)\in \R^d{\times} (0,T^*)$, taking expectation and
integrating over $(y,s)\in \R^d{\times} (0,T^*)$, then repeating everything with \eqref{eq:weakFormShiftChainRule9}
based on the admissible test functions $\tilde\xi(\cdot,\cdot,x,t)$ from \eqref{eq:defTestFunction2} 
for every $(x,t)\in \R^d{\times} (0,T^*)$ and finally summing the two resulting inequalities
(using in particular that $\eta_\delta$ is even) yields an estimate of the form
\begin{align}
\label{eq:intermediateSummary}
&\E R_{\mathrm{dt}} \leq \E R_{\mathrm{porMed}} + \E R_{\mathrm{corr}} 
+ \E R_{\mathrm{noise}} + \E R_{\mathrm{error}}
\end{align}
for all $(\tau,\theta,l)$ subject to \eqref{eq:choiceScales}, all $\delta>0$ and all 
$\varepsilon,\hat\varepsilon\leq\frac{\kappa}{2}$. Here, we introduced for convenience
the abbreviations 
\begin{align*}
&R_{\mathrm{dt}} := 
-\int_0^{T^*}\int_{\R^d}\int_0^{T^*}\int_{\R^d} 
\eta_\delta\big(v_\varepsilon(x,t){-}v_{\hat\varepsilon}(y,s)\big)
(\partial_t{+}\partial_s)\xi(x,t,y,s) \dy \ds \dx \dt,
\\&
R_{\mathrm{porMed}}
\\&
:=-\int_0^{T^*}\int_{\R^d}\int_0^{T^*}\int_{\R^d}
|\nabla v_\varepsilon^m(x,t)|^2
(\eta_\delta)''\big(v_\varepsilon^m(x,t){-}v_{\hat\varepsilon}^m(y,s)\big)
\\&~~~~~~~~~~~~~~~~~~~~~~~~~~~~~~~~~~~~~~~~~~~~~~~~~~~\nonumber
\times\xi(x,t,y,s) \dx \dt \dy \ds
\\&~~~~
+\int_0^{T^*}\int_{\R^d}\int_0^{T^*}\int_{\R^d}
\eta_\delta\big(v_\varepsilon^m(x,t){-}v_{\hat\varepsilon}^m(y,s)\big)
\Delta_x\xi(x,t,y,s) \dx \dt \dy \ds
\\&~~~~
-\int_0^{T^*}\int_{\R^d}\int_0^{T^*}\int_{\R^d}
|\nabla v_{\hat\varepsilon}^m(x,t)|^2
(\eta_\delta)''\big(v_\varepsilon^m(x,t){-}v_{\hat\varepsilon}^m(y,s)\big)
\\&~~~~~~~~~~~~~~~~~~~~~~~~~~~~~~~~~~~~~~~~~~~~~~~~~~~\nonumber
\times\xi(x,t,y,s) \dy \ds \dx \dt
\\&~~~~
+\int_0^{T^*}\int_{\R^d}\int_0^{T^*}\int_{\R^d}
\eta_\delta\big(v_\varepsilon^m(x,t){-}v_{\hat\varepsilon}^m(y,s)\big)
\Delta_y\xi(x,t,y,s) \dy \ds \dx \dt,
\\&
R_{\mathrm{corr}}
:=-\int_0^{T^*}\int_{\R^d}\int_0^{T^*}\int_{\R^d} 
\frac{1}{2}\nu^2|\nabla v_\varepsilon(x,t)|^2
(\eta_\delta)''\big(v_\varepsilon(x,t){-}v_{\hat\varepsilon}(y,s)\big)
\\&~~~~~~~~~~~~~~~~~~~~~~~~~~~~~~~~~~~~~~~~~~~~~~~~~~~~~~~~~~~~~\nonumber
\times\xi(x,t,y,s) \dx \dt \dy \ds
\\&~~~~
+\int_0^{T^*}\int_{\R^d}\int_0^{T^*}\int_{\R^d}
\frac{1}{2}\nu^2\eta_\delta\big(v_\varepsilon(x,t){-}v_{\hat\varepsilon}(y,s)\big)
\Delta_x\xi(x,t,y,s) \dx \dt \dy \ds
\\&~~~~
-\int_0^{T^*}\int_{\R^d}\int_0^{T^*}\int_{\R^d} 
\frac{1}{2}\nu^2|\nabla v_{\hat\varepsilon}(x,t)|^2
(\eta_\delta)''\big(v_\varepsilon(x,t){-}v_{\hat\varepsilon}(y,s)\big)
\\&~~~~~~~~~~~~~~~~~~~~~~~~~~~~~~~~~~~~~~~~~~~~~~~~~~~\nonumber
\times\xi(x,t,y,s) \dy \ds \dx \dt
\\&~~~~
+\int_0^{T^*}\int_{\R^d}\int_0^{T^*}\int_{\R^d}
\frac{1}{2}\nu^2\eta_\delta\big(v_\varepsilon(x,t){-}v_{\hat\varepsilon}(y,s)\big)
\Delta_y\xi(x,t,y,s) \dy \ds \dx \dt,
\\&
R_{\mathrm{noise}} :=
-\int_0^{T^*}\int_{\R^d} \int_0^{T^*}\int_{\R^d} 
\nu \eta_\delta\big(v_\varepsilon(x,t){-}v_{\hat\varepsilon}(y,s)\big)
\nabla_x\xi(x,t,y,s) \dx \,\mathrm{d}B_t \dy \ds
\\&~~~~~~~~~~~~
-\int_0^{T^*}\int_{\R^d} \int_0^{T^*}\int_{\R^d} 
\nu \eta_\delta\big(v_\varepsilon(x,t){-}v_{\hat\varepsilon}(y,s)\big)
\nabla_y\xi(x,t,y,s) \dy \,\mathrm{d}B_s \dx \dt,
\end{align*}
as well as
\begin{align*}
&R_{\mathrm{error}}
\\&
:=\int_0^{T^*}\int_{\R^d} \int_0^{T^*}\int_{\R^d} 
\big\{(\eta_\delta)'\big(v_\varepsilon(x,t){-}v_{\hat\varepsilon}(y,s)\big)
{-}(\eta_\delta)'\big(v_\varepsilon^m(x,t){-}v_{\hat\varepsilon}^m(y,s)\big)\big\} 
\\&~~~~~~~~~~~~~~~~~~~~~~~~~~~~~~~~~~~~~~~~~~~~~~~~~~~~~~\nonumber
\times(\Delta_x v_\varepsilon^m)(x,t)\xi(x,t,y,s)
\dx \dt \dy \ds
\\&~~~~
+\int_0^{T^*}\int_{\R^d} \int_0^{T^*}\int_{\R^d}
\big\{(\eta_\delta)'\big(v_{\hat\varepsilon}(y,s){-}v_{\varepsilon}(x,t)\big)
-(\eta_\delta)'\big(v_{\hat\varepsilon}^m(y,s){-}v_{\varepsilon}^m(x,t)\big)\big\} 
\\&~~~~~~~~~~~~~~~~~~~~~~~~~~~~~~~~~~~~~~~~~~~~~~~~~~~~~~\nonumber
\times(\Delta_y v_{\hat\varepsilon}^m)(y,s)\tilde\xi(y,s,x,t) \dy \ds \dx \dt
\\&~~~~
-\int_0^{T^*}\int_{\R^d}\int_0^{T^*}\int_{\R^d}
m\big((u_\varepsilon^\leftarrow)^{m-1}{-}v_\varepsilon^{m-1}\big)(x,t)\nabla_x\xi(x,t,y,s)
\\&~~~~~~~~~~~~~~~~~~~~~~~~~~~~~~~~~\nonumber
\cdot\nabla v_\varepsilon(x,t)
(\eta_\delta)'\big(v_\varepsilon(x,t){-}v_{\hat\varepsilon}(y,s)\big) 
\dx \dt \dy \ds
\\&~~~~
-\int_0^{T^*}\int_{\R^d}\int_0^{T^*}\int_{\R^d}
m\big((u_{\hat\varepsilon}^\leftarrow)^{m-1}{-}v_{\hat\varepsilon}^{m-1}\big)(y,s)\nabla_y\tilde\xi(y,s,x,t)
\\&~~~~~~~~~~~~~~~~~~~~~~~~~~~~~~~~~\nonumber
\cdot\nabla v_{\hat\varepsilon}(y,s)
(\eta_\delta)'\big(v_{\hat\varepsilon}(y,s){-}v_{\varepsilon}(x,t)\big) 
\dy \ds \dx \dt.
\end{align*}
Before we move on with removing the doubling in the time variable by studying 
the limit $\tau\to 0$ let us first perform some computations on the non-linear
diffusion term $R_{\mathrm{porMed}}$ and the correction term $R_{\mathrm{corr}}$.
Exploiting that $(\eta_\delta)''\geq 0$ and $\xi\geq 0$, completing the square
$|\nabla v_{\varepsilon}^m|^2{+}|\nabla v_{\hat\varepsilon}^m|^2
=|\nabla v_{\varepsilon}^m{-}\nabla v_{\hat\varepsilon}^m|^2
+2\nabla v_{\varepsilon}^m\cdot\nabla v_{\hat\varepsilon}^m$ and
integrating by parts entails that
\begin{align*}
&-\int_{\R^d}\int_{\R^d} |\nabla v_\varepsilon^m(x,t)|^2
(\eta_\delta)''\big(v_\varepsilon^m(x,t){-}v_{\hat\varepsilon}^m(y,s)\big)
\xi \dx \dy 
\\&
-\int_{\R^d}\int_{\R^d}|\nabla v_{\hat\varepsilon}^m(x,t)|^2
(\eta_\delta)''\big(v_\varepsilon^m(x,t){-}v_{\hat\varepsilon}^m(y,s)\big)
\xi \dx \dy 
\\&
\leq \int_{\R^d}\int_{\R^d} 2\xi (\nabla_y\cdot\nabla_x)
\eta_\delta\big(v_\varepsilon^m(x,t){-}v_{\hat\varepsilon}^m(y,s)\big) \dx \dy
\\&
= \int_{\R^d}\int_{\R^d} \eta_\delta\big(v_\varepsilon^m(x,t){-}v_{\hat\varepsilon}^m(y,s)\big)
2(\nabla_x\cdot\nabla_y)\xi \dx \dy.
\end{align*}
Since $(\Delta_x{+}\Delta_y)(\gamma_\theta(x{-}y)\bar\xi_l(\frac{x{+}y}{2}))
=\gamma_\theta(x{-}y)\frac{1}{2}\Delta\bar\xi_l(\frac{x{+}y}{2})
{+}\bar\xi_l(\frac{x{+}y}{2})2\Delta\gamma_\theta(x{-}y)$ and
$(2\nabla_x\cdot\nabla_y)(\gamma_\theta(x{-}y)\bar\xi_l(\frac{x{+}y}{2}))
=\gamma_\theta(x{-}y)\frac{1}{2}\Delta\bar\xi_l(\frac{x{+}y}{2})
{-}\bar\xi_l(\frac{x{+}y}{2})2\Delta\gamma_\theta(x{-}y)$, we thus obtain the estimate
\begin{align*}
R_{\mathrm{porMed}}
&\leq \int_0^{T^*}\int_{\R^d}\int_0^{T^*}\int_{\R^d}
\eta_\delta\big(v_\varepsilon^m(x,t){-}v_{\hat\varepsilon}^m(y,s)\big)
\\&~~~~~~~~~~~~~~~~~~~~~~~~~~~~\nonumber
\times\rho_\tau(t{-}s)\varphi\Big(\frac{t{+}s}{2}\Big)
\gamma_\theta(x{-}y)\Delta\bar\xi_l\Big(\frac{x{+}y}{2}\Big) \dx \dt \dy \ds
\\&
=: R^{(1)}_{\mathrm{porMed}}.
\end{align*}

The idea eventually is---after letting $\tau\to 0$, $\delta\to 0$ and removing
the doubling in the spatial variables (the latter by fine-tuning the scales $\theta>0$ and
$l>0$ as suitably chosen powers of $\varepsilon\vee\hat\varepsilon$)---to integrate by
parts and to use the sign in condition~\textit{ii)} for the spatial test function $\bar\xi$.
We will make this precise together with all the required error estimates in a later 
stage of the proof. For the moment, we only wish to mention that the same argument
also applies to the term $R_{\mathrm{corr}}$; one simply substitutes $v_{\varepsilon}$
for $v_{\varepsilon}^m$ and $v_{\hat\varepsilon}$ for $v_{\hat\varepsilon}^m$, respectively.
This shows that
\begin{align*}
R_{\mathrm{corr}}
&\leq \int_0^{T^*}\int_{\R^d}\int_0^{T^*}\int_{\R^d}  \frac{1}{2}\nu^2
\eta_\delta\big(v_\varepsilon(x,t){-}v_{\hat\varepsilon}(y,s)\big)
\\&~~~~~~~~~~~~~~~~~~~~~~~~~~~~\nonumber
\times\rho_\tau(t{-}s)\varphi\Big(\frac{t{+}s}{2}\Big)
\gamma_\theta(x{-}y)\Delta\bar\xi_l\Big(\frac{x{+}y}{2}\Big) \dx \dt \dy \ds
\\&
=: R^{(1)}_{\mathrm{corr}}.
\end{align*}
The task therefore reduces to post-process the bound
\begin{align}
\label{eq:intermediateSummary2}
&\E R_{\mathrm{dt}} \leq \E R_{\mathrm{porMed}}^{(1)} + \E R_{\mathrm{corr}}^{(1)} 
+ \E R_{\mathrm{noise}} + \E R_{\mathrm{error}}
\end{align}
with the remaining three terms left unchanged from the estimate \eqref{eq:intermediateSummary}.
In a first step we aim to remove the doubling in the time variable by letting $\tau\to 0$.

\begin{lemma}
\label{lemma:limitTimeVariable}
Let the assumptions and notation of Subsection~\ref{subsec:proofContraction} 
until this point be in place. Define the quantities
\begin{align}
\label{eq:vanishingThetaRdt}
&R_{\mathrm{dt}}^{(1)}:=-\int_0^{T^*}\int_{\R^d}\int_{\R^d}
\eta_\delta\big(v_\varepsilon(x,t){-}v_{\hat\varepsilon}(y,t)\big)
\gamma_\theta(x{-}y)\bar\xi_l\Big(\frac{x{+}y}{2}\Big)\ddt\varphi(t)
\dy \dx \dt,
\\&\label{eq:vanishingThetaRporMed}
R_{\mathrm{porMed}}^{(2)}:=\int_0^{T^*}\int_{\R^d}\int_{\R^d} 
\eta_\delta\big(v_\varepsilon^m(x,t){-}v_{\hat\varepsilon}^m(y,t)\big)
\varphi(t)\gamma_\theta(x{-}y)\Delta\bar\xi_l\Big(\frac{x{+}y}{2}\Big) \dy \dx \dt,
\\&\label{eq:vanishingThetaRcorr}
R_{\mathrm{corr}}^{(2)}:=\int_0^{T^*}\int_{\R^d}\int_{\R^d} 
\frac{1}{2}\nu^2\eta_\delta\big(v_\varepsilon(x,t){-}v_{\hat\varepsilon}(y,t)\big)
\varphi(t)\gamma_\theta(x{-}y)\Delta\bar\xi_l\Big(\frac{x{+}y}{2}\Big) \dy \dx \dt,
\\&\label{eq:vanishingThetaRerror}
R_{\mathrm{error}}^{(1)}:=
\int_0^{T^*}\int_{\R^d}\int_{\R^d}
\big\{(\eta_\delta)'\big(v_\varepsilon(x,t){-}v_{\hat\varepsilon}(y,t)\big)
{-}(\eta_\delta)'\big(v_\varepsilon^m(x,t){-}v_{\hat\varepsilon}^m(y,t)\big)\big\} 
\\&~~~~~~~~~~~~~~~~~~~~~~~~~~~~~~~~~~~~~\nonumber
\times(\Delta_x v_\varepsilon^m)(x,t)\gamma_\theta(x{-}y)
\bar\xi_l\Big(\frac{x{+}y}{2}\Big)\varphi(t)\dy \dx \dt
\\&~~~~\nonumber
+\int_0^{T^*}\int_{\R^d}\int_{\R^d}
\big\{(\eta_\delta)'\big(v_{\hat\varepsilon}(y,t){-}v_{\varepsilon}(x,t)\big)
-(\eta_\delta)'\big(v_{\hat\varepsilon}^m(y,t){-}v_{\varepsilon}^m(x,t)\big)\big\} 
\\&~~~~~~~~~~~~~~~~~~~~~~~~~~~~~~~~~~~~~\nonumber
\times(\Delta_y v_{\hat\varepsilon}^m)(y,t)\gamma_\theta(x{-}y)
\bar\xi_l\Big(\frac{x{+}y}{2}\Big)\varphi(t)\dy \dx \dt
\\&~~~~\nonumber
-\int_0^{T^*}\int_{\R^d}\int_{\R^d}
m\big((u_\varepsilon^\leftarrow)^{m-1}{-}v_\varepsilon^{m-1}\big)(x,t)
\nabla_x\Big(\gamma_\theta(x{-}y)
\bar\xi_{l}\Big(\frac{x{+}y}{2}\Big)\Big)\varphi(t)
\\&~~~~~~~~~~~~~~~~~~~~~~~~~~~~~~~~~~~~~~~~~~~\nonumber
\cdot\nabla v_\varepsilon(x,t)
(\eta_\delta)'\big(v_\varepsilon(x,t){-}v_{\hat\varepsilon}(y,t)\big) 
\dy \dx \dt
\\&~~~~\nonumber
-\int_0^{T^*}\int_{\R^d}\int_{\R^d}
m\big((u_{\hat\varepsilon}^\leftarrow)^{m-1}{-}v_{\hat\varepsilon}^{m-1}\big)(y,t)
\nabla_y\Big(\gamma_\theta(x{-}y)
\bar\xi_{l}\Big(\frac{x{+}y}{2}\Big)\Big)\varphi(t)
\\&~~~~~~~~~~~~~~~~~~~~~~~~~~~~~~~~~~~~~~~~~~~\nonumber
\cdot\nabla v_{\hat\varepsilon}(y,t)
(\eta_\delta)'\big(v_{\hat\varepsilon}(y,t){-}v_{\varepsilon}(x,t)\big) 
\dy \dx \dt.
\end{align}
Then the estimate
\begin{align}
\label{eq:intermediateSummary3}
&\E R_{\mathrm{dt}}^{(1)} \leq \E R_{\mathrm{porMed}}^{(2)} 
+ \E R_{\mathrm{corr}}^{(2)} + \E R_{\mathrm{error}}^{(1)}
\end{align}
holds true for all $(\theta,l)$ subject to \eqref{eq:choiceScales}, all $\delta>0$ and all 
$\varepsilon,\hat\varepsilon\leq\frac{\kappa}{2}$.
\end{lemma}

\begin{proof}
It follows from \eqref{eq:defTestFunction} that $(\partial_t{+}\partial_s)\xi
=\varphi'(\frac{t{+}s}{2})\rho_\tau(t{-}s)\gamma_\theta(x{-}y)\bar\xi_l(\frac{x{+}y}{2})$.
In particular, the singular terms as $\tau\to 0$ cancel. Hence, it follows by
Lebesgue's dominated convergence based on the regularity and the bounds for the Wong--Zakai approximation
$u_\varepsilon$ from Lemma~\ref{lem:classicalSolution} and definition~\eqref{eq:shift} of the
shifted densities~$u_\varepsilon^\leftarrow$ that
\begin{align}
\label{eq:vanishingThetaRdt2}
\E R_{\mathrm{dt}} \to \E R_{\mathrm{dt}}^{(1)}
\text{ as } \tau\to 0.
\end{align}
Relying again on Lebesgue's dominated convergence due to the regularity and the 
a priori estimates for the Wong--Zakai approximation
$u_\varepsilon$ from Lemma~\ref{lem:classicalSolution} and the definition \eqref{eq:shift} of the
shifted densities $u_\varepsilon^\leftarrow$, 
we may also easily pass to the limit $\tau\to 0$ in all the terms on the right hand side of \eqref{eq:intermediateSummary2}
except for the noise term~$R_{\mathrm{noise}}$. More precisely, we obtain
\begin{align}
\label{eq:vanishingThetaRporMed2}
\E R_{\mathrm{porMed}}^{(1)} &\to \E R_{\mathrm{porMed}}^{(2)} &&
\text{ as } \tau\to 0,
\\\label{eq:vanishingThetaRcorr2}
\E R_{\mathrm{corr}}^{(1)} &\to \E R_{\mathrm{corr}}^{(2)} &&
\text{ as } \tau\to 0,
\\\label{eq:vanishingThetaRerror2}
\E R_{\mathrm{error}} &\to \E R_{\mathrm{error}}^{(1)} &&
\text{ as } \tau\to 0.
\end{align} 

We proceed with the discussion of the noise term $R_{\mathrm{noise}}$. To this end, we define
for $\lambda>0$ the cut-off $\bar\rho_\lambda:=\frac{1}{\lambda}\bar\rho(\frac{\cdot}{\lambda})$
by means of a standard even cut-off function $\bar\rho\in C^\infty_{\mathrm{cpt}}((-1,1);[0,\infty))$
such that $\int_\R\bar\rho(r)\,\mathrm{d}r=1$. Using that the mollifier $\rho$ is supported on the 
positive real axis together with adding zero yields
for all $\tau>0$ subject to \eqref{eq:choiceScales}
\begin{align*}
&\E R_{\mathrm{noise}}
\\&
=-\E\int_0^{T^*}\int_{\R^d} \int_s^{s+\tau}\int_{\R^d}\int_\R 
\nu\bar\rho_\lambda(v_{\hat\varepsilon}(y,s){-}a) 
\eta_\delta\big(v_\varepsilon(x,t){-}a\big)
\nabla_x\xi \da \dx \,\mathrm{d}B_t \dy \ds
\\&~~~
-\E\int_0^{T^*}\int_{\R^d} \int_{t-\tau}^{t}\int_{\R^d}\int_\R 
\nu\bar\rho_\lambda(b{-}v_{\varepsilon}(x,t)) 
\eta_\delta\big(b{-}v_{\hat\varepsilon}(y,s)\big)
\nabla_y\xi \db \dy \,\mathrm{d}B_s \dx \dt
\\&
=: \E R_{\mathrm{noise}}^{(1)} + \E R_{\mathrm{noise}}^{(2)}. 
\end{align*}
Observe that $\E R_{\mathrm{noise}}^{(1)}=0$ since the bounded random variable $v_{\hat\varepsilon}(y,s)$
is measurable with respect to $\mathcal{F}_s$. Similarly, due to the fact that $v_{\varepsilon}(x,t{-}\tau)$
is bounded and measurable with respect to $\mathcal{F}_{t-\tau}$ we may actually write
\begin{align*}
\E R_{\mathrm{noise}}^{(2)}
&=-\E\int_0^{T^*}\int_{\R^d} \int_{t-\tau}^{t}\int_{\R^d}\int_\R 
\nu\big\{\bar\rho_\lambda(b{-}v_{\varepsilon}(x,t)) 
{-}\bar\rho_\lambda(b{-}v_{\varepsilon}(x,t{-}\tau)) \big\}
\\&~~~~~~~~~~~~~~~~~~~~~~~~~~~~~~~~~~~~~~~
\times\eta_\delta\big(b{-}v_{\hat\varepsilon}(y,s)\big)
\nabla_y\xi \db \dy \,\mathrm{d}B_s \dx \dt.
\end{align*}
Hence, it follows by an application of the Burkholder--Davis--Gundy inequality and then
Lebesgue's dominated convergence relying again on the regularity and bounds for the Wong--Zakai approximation
$u_\varepsilon$ from Lemma~\ref{lem:classicalSolution} and the definition \eqref{eq:shift} of the
shifted densities $u_\varepsilon^\leftarrow$ 
that
\begin{align*}
\E R_{\mathrm{noise}}^{(2)} \to 0 \text{ as } \tau \to 0.
\end{align*}

In summary, we obtain from this together with \eqref{eq:intermediateSummary2}, \eqref{eq:vanishingThetaRdt2},
\eqref{eq:vanishingThetaRporMed2}, \eqref{eq:vanishingThetaRcorr2} and
\eqref{eq:vanishingThetaRerror2} in the limit $\tau\to 0$ the desired estimate \eqref{eq:intermediateSummary3}. 
This concludes the proof of Lemma~\ref{lemma:limitTimeVariable}.
\end{proof}

We continue with the proof of Proposition~\ref{prop:L1contraction}
taking care in the next step of the error terms. More precisely,
we may derive the following bound.

\begin{lemma}
\label{lemma:boundErrorTerm}
Let the assumptions and notation of Subsection~\ref{subsec:proofContraction} 
until this point be in place. In particular, recall the definition of the
error term $R^{(1)}_{\mathrm{error}}$ from \eqref{eq:vanishingThetaRerror}. 
We then have the estimate
\begin{align}
\label{eq:boundRerror}
\E R^{(1)}_{\mathrm{error}} \leq C(\kappa)l^{-1}\theta^{-1}(\varepsilon\vee\hat\varepsilon)^q + O_{\delta\to 0}(1)
\end{align}
for all $(\theta,l)$ subject to \eqref{eq:choiceScales}, all $\delta>0$ and all 
$\varepsilon,\hat\varepsilon\leq\frac{\kappa}{2}$.
\end{lemma}

\begin{proof} 
We start estimating by H\"{o}lder's inequality
\begin{align*}
&\E\bigg|\int_0^{T^*}\int_{\R^d}\int_{\R^d} (\Delta_x v_\varepsilon^m)(x,t)
\gamma_\theta(x{-}y)\bar\xi_l\Big(\frac{x{+}y}{2}\Big)\varphi(t)
\\&~~~~~~~~~~~~~~~~~~~~~
\times\big\{(\eta_\delta)'\big(v_\varepsilon(x,t){-}v_{\hat\varepsilon}(y,t)\big)
{-}(\eta_\delta)'\big(v_\varepsilon^m(x,t){-}v_{\hat\varepsilon}^m(y,t)\big)\big\}
\dy \dx \dt\bigg|
\\&
\leq \bigg(\E\int_{0}^{T^*}\int_{\R^d}\int_{\R^d}
\big|(\eta_\delta)'\big(v_\varepsilon(x,t){-}v_{\hat\varepsilon}(y,t)\big)
{-}(\eta_\delta)'\big(v_\varepsilon^m(x,t){-}v_{\hat\varepsilon}^m(y,t)\big)\big|^2
\dy \dx \dt\bigg)^\frac{1}{2}
\\&~~~~
\times\frac{\|\varphi\|_{L^\infty(0,T^*)}}{(\inf\supp\varphi)^\frac{1}{2}}
\bigg(\E\int_0^{T^*}\int_{\R^d} t|(\Delta_x v_\varepsilon^m)(x,t)|^2 \dx\dt\bigg)^\frac{1}{2}
\|\bar\xi_l\|_{L^\infty(\R^d)}\|\gamma_\theta\|_{L^\infty(\R^d)}.
\end{align*}
However, by the bounds \eqref{eq:boundLaplacianApprox} and \eqref{eq:auxUniformBound}, 
the convergence \eqref{eq:auxConvergence} and the fact that
$\mathrm{sign}(a{-}b)=\mathrm{sign}(a^m{-}b^m)$ due to the monotonicity of
$r\mapsto r^m$ we infer by an application of Lebesgue's dominated convergence theorem
based on the regularity of the Wong--Zakai approximation $u_\varepsilon$ from
Lemma~\ref{lem:classicalSolution} that the term on the right hand side of the
latter bound vanishes as $\delta\to 0$.

Recall that $\kappa\leq v_\varepsilon\leq u_\varepsilon^\leftarrow$ on $\supp\nabla v_\varepsilon$.
Hence, we can estimate by means of \eqref{auxConvergence21}
\begin{align*}
&\chi_{\supp\nabla v_\varepsilon}
\big|(u_\varepsilon^\leftarrow)^{m-1}(x,t){-}v_\varepsilon^{m-1}(x,t)\big|
\\&
\leq \chi_{\supp\nabla v_\varepsilon}
\sup_{r\in [v_\varepsilon(x,t),u^\leftarrow_\varepsilon(x,t)]} 
 (m{-}1)r^{m-2} |v_\varepsilon(x,t){-}u_\varepsilon^\leftarrow(x,t)|
\\&
\leq C(\kappa)\chi_{\supp\nabla v_\varepsilon}\varepsilon^q.
\end{align*}
We may then estimate using also \eqref{eq:auxUniformBound1}, the definition \eqref{eq:shift}
of the shifted densities $u^\leftarrow_\varepsilon$, $\|\nabla\bar\xi_l\|_{L^\infty(\R^d)}\leq
\|\nabla\gamma_l\|_{L^1(\R^d)}\|\bar\xi\|_{L^\infty(\R^d)}\leq Cl^{-1}$ (which follows from Young's inequality and 
condition~\textit{i)} for the spatial cut-off $\bar\xi$),
$\|\varphi\|_{L^\infty(0,T^*)}\leq 1$, a change of variables $x\mapsto x{-}\nu(B_t{-}B^\varepsilon_t)$,
the fact that $u_\varepsilon^\leftarrow\geq \kappa$ on $\supp\nabla v_\varepsilon$,
and finally the a priori estimate \eqref{eq:energyInequalityApprox}
\begin{align*}
&\E\bigg|\int_{0}^{T^*}\int_{\R^d}\int_{\R^d} 
\big((u_\varepsilon^\leftarrow)^{m-1}(x,t){-}v_\varepsilon^{m-1}(x,t)\big)
\\&~~~~~~~~~~~~~~~~~~~~~
\times\nabla v_\varepsilon(x,t)\cdot\nabla_x\Big(\gamma_\theta(x{-}y)
\bar\xi_{l}\Big(\frac{x{+}y}{2}\Big)\Big)\varphi(t)\dy \dx \dt\bigg|
\\&
\leq Cl^{-1}\theta^{-1}\E\int_0^{T^*}\int_{\R^d}\chi_{\supp\nabla v_\varepsilon}
\big|(u_\varepsilon^\leftarrow)^{m-1}(x,t){-}v_\varepsilon^{m-1}(x,t)\big|
|\nabla u_\varepsilon^\leftarrow(x,t)| \dx \dt
\\&
\leq C(\kappa)l^{-1}\theta^{-1}\varepsilon^q\bigg(
\E\int_0^{T^*}\int_{D} mu_\varepsilon^{m-1}(x,t)|\nabla u_\varepsilon(x,t)|^2 \dx\dt\bigg)^\frac{1}{2}
\\&
\leq C(\kappa)l^{-1}\theta^{-1}\varepsilon^q.
\end{align*}

Since the other two terms of $R^{(1)}_{\mathrm{error}}$ can be treated along the same lines, 
we obtain in total the asserted estimate \eqref{eq:boundRerror}.
This concludes the proof of Lemma~\ref{lemma:boundErrorTerm}.
\end{proof}

We continue with the proof of Proposition~\ref{prop:L1contraction}.
We have by now everything in place to let $\delta\to 0$ in \eqref{eq:intermediateSummary3}.
A straightforward application of Lebesgue's dominated
convergence theorem based on the convergence in \eqref{eq:auxConvergence} and the regularity and bounds
for the Wong--Zakai approximations $u_\varepsilon$ from Lemma~\ref{lem:classicalSolution} shows that
by letting $\delta\to 0$ in \eqref{eq:intermediateSummary3} and using \eqref{eq:boundRerror} it holds
\begin{align}
\label{eq:intermediateSummary4}
&\E R_{\mathrm{dt}}^{(2)} \leq \E R_{\mathrm{porMed}}^{(3)} 
+ \E R_{\mathrm{corr}}^{(3)} + C(\kappa)l^{-1}\theta^{-1}(\varepsilon\vee\hat\varepsilon)^q
\end{align}
for all $(\theta,l)$ subject to \eqref{eq:choiceScales}, and all 
$\varepsilon,\hat\varepsilon\leq\frac{\kappa}{2}$. The updated terms in this inequality 
are given by
\begin{align*}
R_{\mathrm{dt}}^{(2)} &:= -\int_0^{T^*}\int_{\R^d}\int_{\R^d}
\big|v_\varepsilon(x,t){-}v_{\hat\varepsilon}(y,t)\big|
\gamma_\theta(x{-}y)\bar\xi_l\Big(\frac{x{+}y}{2}\Big)\ddt\varphi(t)
\dy \dx \dt,
\\
R_{\mathrm{porMed}}^{(3)} &:= \int_0^{T^*}\int_{\R^d}\int_{\R^d}
\big|v_\varepsilon^m(x,t){-}v_{\hat\varepsilon}^m(y,t)\big|
\varphi(t)\gamma_\theta(x{-}y)\Delta\bar\xi_l\Big(\frac{x{+}y}{2}\Big) \dy \dx \dt,
\\
R_{\mathrm{corr}}^{(3)} &:= \int_0^{T^*}\int_{\R^d}\int_{\R^d} 
\frac{1}{2}\nu^2\big|v_\varepsilon(x,t){-}v_{\hat\varepsilon}(y,t)\big|
\varphi(t)\gamma_\theta(x{-}y)\Delta\bar\xi_l\Big(\frac{x{+}y}{2}\Big) \dy \dx \dt,
\end{align*}
respectively. We may proceed with the next step which consists of removing 
the doubling in the spatial variables. More precisely, the following holds true.

\begin{lemma}
\label{lemma:removeDoublingSpace}
Let the assumptions and notation of Subsection~\ref{subsec:proofContraction} 
until this point be in place. Define the quantities
\begin{align*}
R_{\mathrm{dt}}^{(3)} &:= -\int_0^{T^*}\int_{\R^d}\int_{\R^d}
\big|v_\varepsilon(x,t){-}v_{\hat\varepsilon}(y,t)\big|
\gamma_\theta(x{-}y)\bar\xi_l(x)\ddt\varphi(t)
\dy \dx \dt,
\\
R_{\mathrm{porMed}}^{(5)} &:= \int_0^{T^*}\varphi(t)\int_{\R^d} 
\big|v_\varepsilon^m(x,t){-}v_{\hat\varepsilon}^m(x,t)\big|
\Delta\bar\xi_l(x) \dx \dt,
\\
R_{\mathrm{corr}}^{(5)} &:= \int_0^{T^*}\varphi(t)\int_{\R^d}\frac{1}{2}\nu^2
\big|v_\varepsilon(x,t){-}v_{\hat\varepsilon}(x,t)\big|
\Delta\bar\xi_l(x) \dx \dt.
\end{align*}
We then have the estimate
\begin{align}
\label{eq:intermediateSummary7}
\E R_{\mathrm{dt}}^{(3)} &\leq \E R_{\mathrm{porMed}}^{(5)} 
+ \E R_{\mathrm{corr}}^{(5)} + C(\kappa)l^{-1}\theta^{-1}(\varepsilon\vee\hat\varepsilon)^q
+ C(\kappa)l^{-3}\theta
\end{align}
for all $\varphi\in C^\infty_{\mathrm{cpt}}((0,T^*);[0,1])$ 
with $\|\varphi\|_{W^{1,1}(0,T^*)}\leq \bar C$, all $(\theta,l)$ subject to \eqref{eq:choiceScales} 
and all $\varepsilon,\hat\varepsilon\leq\frac{\kappa}{2}$.
\end{lemma}

\begin{proof}
By a straightforward estimate, we may replace in all three terms $\bar\xi_l(\frac{\cdot{-}y}{2})$ 
by $\bar\xi_l(\cdot)$ and therefore update the estimate \eqref{eq:intermediateSummary4} to
\begin{align}
\label{eq:intermediateSummary5}
&\E R_{\mathrm{dt}}^{(3)} \leq \E R_{\mathrm{porMed}}^{(4)} 
+ \E R_{\mathrm{corr}}^{(4)} + C(\kappa)l^{-1}\theta^{-1}(\varepsilon\vee\hat\varepsilon)^q
+ Cl^{-3}\theta\|\varphi\|_{W^{1,1,}(0,T^*)}
\end{align}
for all $(\theta,l)$ subject to \eqref{eq:choiceScales} and all 
$\varepsilon,\hat\varepsilon\leq\frac{\kappa}{2}$, where 
\begin{align*}
R_{\mathrm{porMed}}^{(4)} &:= \int_0^{T^*}\int_{\R^d}\int_{\R^d}
\big|v_\varepsilon^m(x,t){-}v_{\hat\varepsilon}^m(y,t)\big|
\varphi(t)\gamma_\theta(x{-}y)\Delta\bar\xi_l(x) \dy \dx \dt,
\\
R_{\mathrm{corr}}^{(4)} &:= \int_0^{T^*}\int_{\R^d}\int_{\R^d} 
\frac{1}{2}\nu^2\big|v_\varepsilon(x,t){-}v_{\hat\varepsilon}(y,t)\big|
\varphi(t)\gamma_\theta(x{-}y)\Delta\bar\xi_l(x) \dy \dx \dt,
\end{align*}
respectively. We estimate using $\|\Delta\bar\xi_l\|_{L^\infty(\R^d)}\leq
\|\nabla^2\gamma_l\|_{L^1(\R^d)}\|\bar\xi\|_{L^\infty(\R^d)}\leq Cl^{-2}$ (which follows from Young's inequality and 
condition~\textit{i)} for the spatial cut-off $\bar\xi$),
a change of variables $y\mapsto y{+}x$, $u_\varepsilon^\leftarrow\geq \kappa$ on $\supp\nabla v_\varepsilon$,
and the bounds \eqref{eq:auxUniformBound1} resp.\ \eqref{eq:energyInequalityApprox}
\begin{align*}
&\bigg|\E R_{\mathrm{corr}}^{(4)} - \E\int_0^{T^*}\int_{\R^d}
\frac{1}{2}\nu^2\big|v_\varepsilon(x,t){-}v_{\hat\varepsilon}(x,t)\big|
\varphi(t)\Delta\bar\xi_l(x)\dx \dt \bigg|
\\&
\leq C\E\int_0^{T^*}\int_{\R^d}\int_{\R^d}\big|v_{\hat\varepsilon}(y,t){-}v_{\hat\varepsilon}(x,t)\big| 
\gamma_\theta(x{-}y)|\Delta\bar\xi_l(x)|\varphi(t)\dy \dx \dt
\\&
\leq Cl^{-2}\theta\E\int_0^{T^*}\int_{\R^d}\int_{\R^d} \gamma_\theta(x{-}y)
\int_0^1 |\nabla v_{\hat\varepsilon}(rx{+}(1{-}r)y,t)| \,\mathrm{d}r \dy \dx \dt
\\&
\leq C(\kappa)l^{-2}\theta \E \int_0^1\int_0^{T^*}\int_{\R^d}\int_{\R^d} \gamma_\theta(y)
(u_{\hat\varepsilon}^\leftarrow)^\frac{m-1}{2}(x{-}ry,t)
|\nabla u_{\hat\varepsilon}^\leftarrow(x{-}ry,t)| \dx \dy \dt  \,\mathrm{d}r
\\&
\leq C(\kappa)l^{-2}\theta \bigg(\E \int_0^{T^*}\int_D
m u_{\hat\varepsilon}^{m-1}(x,t)|\nabla u_{\hat\varepsilon}(x,t)|^2 \dx \dt\bigg)^\frac{1}{2} 
\\&
\leq C(\kappa)l^{-2}\theta
\end{align*}
for all $(\theta,l)$ subject to \eqref{eq:choiceScales} and all 
$\varepsilon,\hat\varepsilon\leq\frac{\kappa}{2}$. Arguing
along the same lines we also get the estimate
\begin{align*}
\bigg|\E R_{\mathrm{porMed}}^{(4)} - \E\int_0^{T^*}\int_{\R^d} 
\big|v_\varepsilon^m(x,t){-}v_{\hat\varepsilon}^m(x,t)\big|
\varphi(t)\Delta\bar\xi_l(x) \dx \dt\bigg|
\leq C(\kappa)l^{-2}\theta
\end{align*}
for all $(\theta,l)$ subject to \eqref{eq:choiceScales} and all 
$\varepsilon,\hat\varepsilon\leq\frac{\kappa}{2}$.
Summarizing we then obtain the desired inequality \eqref{eq:intermediateSummary7}.
This concludes the proof of Lemma~\ref{lemma:removeDoublingSpace}.
\end{proof}

We continue with the proof of Proposition~\ref{prop:L1contraction}.
The next step takes care of estimating the non-linear diffusion term and the correction term.

\begin{lemma}
\label{lemma:estimateDiffusionCorrection}
Let the assumptions and notation of Subsection~\ref{subsec:proofContraction} 
until this point be in place. In particular, recall from the statement of Lemma~\ref{lemma:removeDoublingSpace}
the definition of the quantities $R_{\mathrm{porMed}}^{(5)}$ and $R_{\mathrm{corr}}^{(5)}$.
Then there exists some $l(\bar\xi)>0$ small enough and some absolute constant $\bar C>0$ 
such that we have the estimate
\begin{align}
\label{eq:intermediateSummary80}
\E R_{\mathrm{porMed}}^{(5)} + \E R_{\mathrm{corr}}^{(5)}
\leq \bar{C}\kappa + C(\kappa)l^{-2}(\varepsilon\vee\hat\varepsilon)^\alpha
\end{align}
for all $(\theta,l)$ subject to \eqref{eq:choiceScales}
resp.\ $l<l(\bar\xi)$, and all $\varepsilon,\hat\varepsilon\leq\frac{\kappa}{2}$. 
\end{lemma}

\begin{proof}
We first aim to replace $v_\varepsilon$ resp.\ $v_{\hat\varepsilon}$
by $\zeta^{\varepsilon^q}_\kappa\circ u_{\varepsilon}$ resp.\ 
$\zeta^{\hat\varepsilon^q}_\kappa\circ u_{\hat\varepsilon}$ in the two terms
$R_{\mathrm{porMed}}^{(5)}$ and $R_{\mathrm{corr}}^{(5)}$. This can be done
by an estimation similar to the proof of Lemma~\ref{lemma:removeDoublingSpace}, this time using in particular
\eqref{eq:approxBrownianMotion}, \eqref{eq:HoelderContBM} and $\E C_\alpha^2<\infty$,
which yields the bound
\begin{align*}
&\E R_{\mathrm{porMed}}^{(5)} + \E R_{\mathrm{corr}}^{(5)} 
\\&
\leq-\E\int_0^{T^*}\varphi(t)\int_{\R^d}
|(\zeta^{\varepsilon^q}_\kappa\circ u_{\varepsilon})^m(x,t){-}
(\zeta^{\hat\varepsilon^q}_\kappa\circ u_{\hat\varepsilon})^m(x,t)|
\Delta\bar\xi_l(x) \dx \dt
\\&~~~
-\E\int_0^{T^*}\varphi(t)\int_{\R^d}\frac{1}{2}\nu^2
|(\zeta^{\varepsilon^q}_\kappa\circ u_{\varepsilon})(x,t)
{-}(\zeta^{\hat\varepsilon^q}_\kappa\circ u_{\hat\varepsilon})(x,t)|
\Delta\bar\xi_l(x) \dx \dt
\\&~~~
+C(\kappa)l^{-2}(\varepsilon\vee\hat\varepsilon)^\alpha.
\end{align*}
However, it follows from $\varepsilon,\hat\varepsilon\leq\frac{\kappa}{2}$
that $(\zeta^{\varepsilon^q}_\kappa\circ u_{\varepsilon})^m{-}
(\zeta^{\hat\varepsilon^q}_\kappa\circ u_{\hat\varepsilon})^m\in H^1_0(D)$
as well as $(\zeta^{\varepsilon^q}_\kappa\circ u_{\varepsilon}){-}
(\zeta^{\hat\varepsilon^q}_\kappa\circ u_{\hat\varepsilon})\in H^1_0(D)$.
An integration by parts together with condition~\textit{ii)}
of the spatial cut-off function $\bar\xi$ then entails
\begin{align*}
&\E R_{\mathrm{porMed}}^{(5)} + \E R_{\mathrm{corr}}^{(5)} 
\\&
\leq-\E\int_0^{T^*}\varphi(t)\int_{\R^d}
\nabla|(\zeta^{\varepsilon^q}_\kappa\circ u_{\varepsilon})^m(x,t){-}
(\zeta^{\hat\varepsilon^q}_\kappa\circ u_{\hat\varepsilon})^m(x,t)|
\cdot\nabla\bar\xi_l(x) \dx \dt
\\&~~~
-\E\int_0^{T^*}\varphi(t)\int_{\R^d}\frac{1}{2}\nu^2
\nabla|(\zeta^{\varepsilon^q}_\kappa\circ u_{\varepsilon})(x,t)
{-}(\zeta^{\hat\varepsilon^q}_\kappa\circ u_{\hat\varepsilon})(x,t)|
\cdot\nabla\bar\xi_l(x) \dx \dt
\\&~~~
+C(\kappa)l^{-2}(\varepsilon\vee\hat\varepsilon)^\alpha
\\&
\leq -\E\int_0^{T^*}\varphi(t)\int_{\R^d} 
\nabla|(\zeta^{\varepsilon^q}_\kappa\circ u_{\varepsilon})^m(x,t){-}
(\zeta^{\hat\varepsilon^q}_\kappa\circ u_{\hat\varepsilon})^m(x,t)|
\cdot\nabla(\bar\xi_l{-}\bar\xi)(x) \dx \dt
\\&~~~
-\E\int_0^{T^*}\varphi(t)\int_{\R^d}\frac{1}{2}\nu^2
\nabla|(\zeta^{\varepsilon^q}_\kappa\circ u_{\varepsilon})(x,t)
{-}(\zeta^{\hat\varepsilon^q}_\kappa\circ u_{\hat\varepsilon})(x,t)|
\cdot\nabla(\bar\xi_l{-}\bar\xi)(x) \dx \dt
\\&~~~
+C(\kappa)l^{-2}(\varepsilon\vee\hat\varepsilon)^\alpha.
\end{align*}
It remains to bound the two terms featuring the difference $\nabla\bar\xi_l{-}\nabla\bar\xi$.
By continuity of translations in $L^2$, however, together with the a priori estimate 
\eqref{eq:energyInequalityApprox} we find some small enough $l(\bar\xi)>0$ such that
\begin{align}
\label{eq:condition2Scales}
l<l(\bar\xi) \quad\Rightarrow\quad \|\bar\xi_l{-}\bar\xi\|_{H^1(\R^d)}\leq \kappa^{\frac{m+1}{2}},
\end{align}
holds true, and therefore in particular the desired estimate~\eqref{eq:intermediateSummary80}
for some absolute constant $\bar C>0$. 
\end{proof}

We continue with the proof of Proposition~\ref{prop:L1contraction}.
It is straightforward to estimate
\begin{align*}
&\bigg|{-}\E R_{\mathrm{dt}}^{(3)}{-}\E\int_0^{T^*}\ddt\varphi(t)\int_{\R^d}\int_{\R^d}
|v_\varepsilon(x,t){-}v_{\hat\varepsilon}(x,t)|\gamma_\theta(x{-}y)\bar\xi(x) \dy \dx \dt\bigg|
\leq \bar C\kappa 
\end{align*}
for all $l$ subject to \eqref{eq:condition2Scales} and all
$\varphi\in C^\infty_{\mathrm{cpt}}((0,T^*);[0,1])$ 
with $\|\varphi\|_{W^{1,1}(0,T^*)}\leq \bar C$. In summary, we thus obtain
together with \eqref{eq:intermediateSummary80} the estimate
\begin{align}
\nonumber
-\E R_{\mathrm{dt}}^{(4)}&:=-\E\int_0^{T^*}\ddt\varphi(t)\int_{\R^d}\int_{\R^d}
|v_\varepsilon(x,t){-}v_{\hat\varepsilon}(y,t)|
\gamma_\theta(x{-}y)\bar\xi(x) \dy \dx \dt
\\&~\label{eq:intermediateSummary8}
\leq C(\kappa)l^{-1}\theta^{-1}(\varepsilon\vee\hat\varepsilon)^q
+ C(\kappa)l^{-2}(\varepsilon\vee\hat\varepsilon)^\alpha
+ C(\kappa)l^{-3}\theta
+ \bar C\kappa
\end{align}
for all $\varphi\in C^\infty_{\mathrm{cpt}}((0,T^*);[0,1])$ 
with $\|\varphi\|_{W^{1,1}(0,T^*)}\leq \bar C$, all $(\theta,l)$ 
subject to \eqref{eq:choiceScales} resp.\ \eqref{eq:condition2Scales}, 
and all $\varepsilon,\hat\varepsilon\leq\frac{\kappa}{2}$.

In the next step we take care of the term $\E R_{\mathrm{dt}}^{(4)}$.
Employing the same argument leading to~\cite[(4.20)]{Dareiotis:2018}
(instead of using~\cite[Lemma 3.2]{Dareiotis:2018} to treat the initial
condition we can also rely on the continuity down to $t=0$ thanks to Lemma~\ref{lem:classicalSolution})
we infer that the estimate \eqref{eq:intermediateSummary8} entails the bound
(recall that the initial condition is deterministic)
\begin{align}
\nonumber
&\E\int_{D}\int_{D}
|v_\varepsilon(x,T){-}v_{\hat\varepsilon}(y,T)|
\gamma_\theta(x{-}y)\bar\xi(x) \dy \dx
\\&\label{eq:intermediateSummary9}
\leq \int_{D}\int_{D}
|v_\varepsilon(x,0){-}v_{\hat\varepsilon}(y,0)|
\gamma_\theta(x{-}y)\bar\xi(x) \dy \dx
\\&~~~\nonumber
+ C(\kappa)l^{-1}\theta^{-1}(\varepsilon\vee\hat\varepsilon)^q
+ C(\kappa)l^{-2}(\varepsilon\vee\hat\varepsilon)^\alpha
+ C(\kappa)l^{-3}\theta
+ \bar C\kappa
\end{align}
for all $T\in [0,T^*]$, all $(\theta,l)$ subject to \eqref{eq:choiceScales} resp.\ \eqref{eq:condition2Scales}, 
and all $\varepsilon,\hat\varepsilon\leq\frac{\kappa}{2}$. We then estimate for all $T\in[\kappa,T^*]$ 
by means of \eqref{eq:boundTimeDerivativeApprox}, \eqref{eq:auxUniformBound1} and $\|\bar\xi\|_{L^\infty(D)}\leq 1$
\begin{align*}
&\bigg|\E\int_{D}\int_{D}
|v_\varepsilon(x,T){-}v_{\hat\varepsilon}(y,T)|
\gamma_\theta(x{-}y)\bar\xi(x) \dy \dx
{-}\E\int_{D}|v_\varepsilon(x,T){-}v_{\hat\varepsilon}(x,T)|\bar\xi(x)\dx\bigg|
\\&
\leq\E\int_{D}\int_{D}|v_{\hat\varepsilon}(y{+}x,T){-}v_{\hat\varepsilon}(x,T)|
\gamma_\theta(y)\bar\xi(x)\dx\dy
\\&
\leq C\theta\E \int_0^1\int_{\R^d}\int_{\R^d} 
|\nabla u_{\hat\varepsilon}^\leftarrow(x{-}ry,T)|
\gamma_\theta(y) \dx \dy \,\mathrm{d}r
\\&
\leq C(\kappa)\theta\E\int_{D} u_{\hat\varepsilon}^{m-1}|\nabla u_{\hat\varepsilon}| \dx
\leq C(\kappa)\hat\varepsilon^{-\beta}\theta.
\end{align*}
Since we may assume in this argument without loss of generality 
that $\varepsilon\leq\hat\varepsilon$ (otherwise, switch the roles of $v_{\hat\varepsilon}$
and $v_\varepsilon$ in the previous estimate) we obtain together with
an analogous estimate based on the regularity of the initial condition
\begin{align}
\nonumber
&\E\int_{D}|v_\varepsilon(x,T){-}v_{\hat\varepsilon}(x,T)|\bar\xi(x) \dy \dx
\\&\label{eq:intermediateSummary10}
\leq \int_{D}|v_\varepsilon(x,0){-}v_{\hat\varepsilon}(x,0)|\dx
+ C(\kappa)l^{-1}\theta^{-1}(\varepsilon\vee\hat\varepsilon)^q
+ C(\kappa)l^{-2}(\varepsilon\vee\hat\varepsilon)^\alpha
\\&~~~\nonumber
+ C(\kappa)(\varepsilon\vee\hat\varepsilon)^{-\beta}\theta
+ C(\kappa)l^{-3}\theta
+ \bar C\kappa
\end{align}
for all $T\in [\kappa,T^*]$, all $(\theta,l)$ subject to \eqref{eq:choiceScales} resp.\ \eqref{eq:condition2Scales}, 
and all $\varepsilon,\hat\varepsilon\leq\frac{\kappa}{2}$. As a consequence of \eqref{auxConvergence21}, we may finally
switch from $v_\varepsilon=\zeta^{\varepsilon^q}_\kappa\circ u_\varepsilon^\leftarrow$ to 
$\kappa\vee u_{\varepsilon}^\leftarrow$ which yields the bound
\begin{align}
\nonumber
&\E\int_{D}|\kappa\vee u_{\varepsilon}^\leftarrow(x,T){-}
\kappa\vee u_{\hat\varepsilon}^\leftarrow(x,T)|\bar\xi(x) \dy \dx
\\&\label{eq:intermediateSummary11}
\leq \int_{D}|\kappa\vee u_{\varepsilon}^\leftarrow(x,0)
{-}\kappa\vee u_{\hat\varepsilon}^\leftarrow(x,0)|\dx
+ C(\kappa)l^{-1}\theta^{-1}(\varepsilon\vee\hat\varepsilon)^q
\\&~~~\nonumber
+ C(\kappa)l^{-2}(\varepsilon\vee\hat\varepsilon)^\alpha
+ C(\kappa)(\varepsilon\vee\hat\varepsilon)^{-\beta}\theta
+ C(\kappa)l^{-3}\theta
+ \bar C\kappa
\end{align}
for all $T\in [\kappa,T^*]$, all $(\theta,l)$ subject to \eqref{eq:choiceScales} resp.\ \eqref{eq:condition2Scales}, 
and all $\varepsilon,\hat\varepsilon\leq\frac{\kappa}{2}$. We eventually arrived at the last step of the proof.

In light of the right hand side terms in \eqref{eq:intermediateSummary11} we 
first define $\theta:=(\varepsilon\vee\hat\varepsilon)^{\beta+1}$,
then fix $\bar\vartheta>0$ and $q>0$ such that $2\bar\vartheta<\alpha$, $3\bar\vartheta<\beta+1$
as well as $\beta+1+\bar\vartheta<q$, and finally define $l:=(\varepsilon\vee\hat\varepsilon)^{\bar\vartheta}$.
Choosing $\varepsilon_0>0$ small enough, we can ensure that $(\theta,l)$
satisfy \eqref{eq:choiceScales} resp.\ \eqref{eq:condition2Scales}. Hence, these choices
guarantee that we may infer from \eqref{eq:intermediateSummary11} a bound of the type \eqref{eq:L1contraction}
with right hand side terms $C(\kappa)(\varepsilon\vee\hat\varepsilon)^{2\vartheta}$ for some suitable $\vartheta>0$.
Choosing $\varepsilon_0$ even smaller, if needed, we can avoid the dependence of the constant on the data by
sacrificing a power $\vartheta$, which together with the fact that $\bar\xi=\chi_D$ on $K$ entails the 
desired bound \eqref{eq:L1contraction}. 
This concludes the proof of Proposition~\ref{prop:L1contraction}. \qed

\subsection{Proof of Corollary~\ref{cor:L1convergence} 
{\normalfont ($L^1$ convergence of shifted densities)}}
Let $\delta>0$ be fixed but arbitrary. By the triangle inequality we may estimate
\begin{align*}
&\sup_{T\in [\kappa,T^*]}\E\int_D|u^\leftarrow_{\varepsilon}(T){-}u^\leftarrow_{\hat\varepsilon}(T)|\dx
\\&
\leq \sup_{T\in [\kappa,T^*]}\E\int_D|u^\leftarrow_{\varepsilon}(T)|\chi_{\{u_\varepsilon(T)<\kappa\}}
+|u^\leftarrow_{\hat\varepsilon}(T)|\chi_{\{u_{\hat\varepsilon}(T)<\kappa\}}\dx
\\&~~~
+\sup_{T\in [\kappa,T^*]} \E\int_D |\kappa\vee u^\leftarrow_\varepsilon(x,T)
{-}\kappa\vee u^\leftarrow_{\hat\varepsilon}(x,T)|\dx
\end{align*}
for all $\kappa\in (0,T^*\wedge 1)$ and all $\varepsilon,\hat\varepsilon\leq\frac{\kappa}{2}$. Hence, it follows from
splitting $D=K\cup(D\setminus K)$ together with the bounds \eqref{boundMaximumPrinciple} and \eqref{eq:L1contraction} 
as well as the definition \eqref{eq:shift} of the shifted densities that
\begin{align*}
&\sup_{T\in [\kappa,T^*]}\E\int_D|u^\leftarrow_{\varepsilon}(T){-}u^\leftarrow_{\hat\varepsilon}(T)|\dx
\\&
\leq 2\mathcal{L}^d(D)\kappa
+ 2\mathcal{L}^d(D\setminus K)\big(\kappa\vee(1{+}\|u_0\|_{L^\infty(D)})\big)
+ \bar C(\varepsilon\vee\hat\varepsilon)^\vartheta + \bar C\kappa
\\&~~~
+\E\int_D |\kappa\vee(u_0(x){+}\varepsilon){-}\kappa\vee(u_0(x){+}\hat\varepsilon)|\dx.
\end{align*}
The term with the initial data is estimated similarly by
\begin{align*}
&\E\int_D |\kappa\vee(u_0(x){+}\varepsilon){-}\kappa\vee(u_0(x){+}\hat\varepsilon)|\dx
\leq 2\mathcal{L}^d(D)\kappa + \mathcal{L}^d(D)|\varepsilon{-}\hat\varepsilon|
\end{align*}
for all $\kappa>0$ and all $\varepsilon,\hat\varepsilon\leq\frac{\kappa}{2}$.
Choosing first $\kappa<\tau$ sufficiently small such that 
$(4\mathcal{L}^d(D){+}\bar C)\kappa\leq\frac{\delta}{2}$,
we may then fix a large enough compact set $K\subset D$ and some small enough $\varepsilon_0(\kappa,K)$ 
such that the bound
\begin{align*}
\sup_{T\in [\tau,T^*]}\E\int_D|u^\leftarrow_{\varepsilon}(T){-}u^\leftarrow_{\hat\varepsilon}(T)|\dx \leq \delta
\end{align*}
holds true for all $\varepsilon,\hat\varepsilon\leq\varepsilon_0$. This proves that
the sequence of shifted densities $u_\varepsilon^\leftarrow$ is a Cauchy sequence in the space 
$C([\tau,T^*];L^1(\Omega{\times}D,\Prob\otimes\mathcal{L}^d))$ for all $\tau>0$. The
corresponding assertion in the space $L^1([0,T^*];L^1(\Omega{\times}D,\Prob\otimes\mathcal{L}^d))$
is proved similarly based on the additional estimate
\begin{align*}
&\E\int_0^{T^*}\int_D|u^\leftarrow_{\varepsilon}{-}u^\leftarrow_{\hat\varepsilon}|\dx\dt
\\&
\leq 2\mathcal{L}^d(D)\kappa(1{+}\|u_0\|_{L^\infty(D)})
+(T^*{-}\kappa)\sup_{T\in [\kappa,T^*]}\E\int_D|u^\leftarrow_{\varepsilon}(T){-}u^\leftarrow_{\hat\varepsilon}(T)|\dx.
\end{align*}
Let us denote the corresponding limit in $L^1([0,T^*];L^1(\Omega{\times}D,\Prob\otimes\mathcal{L}^d))$ by $u$.

On the other side, it follows immediately from the bound \eqref{boundMaximumPrinciple}
that the sequence of densities $u_\varepsilon$ has a weak limit in the
space $L^{m+1}(\Omega_{T^*},\mathcal{P}_{T^*};L^{m+1}(D))$, which we denote by $\bar u$. 
It remains to verify that $u=\bar u$. To this end, let $\phi\in C^\infty_{\mathrm{cpt}}(D\times [0,T^*])$
and $A\in\mathcal{F}_{T^*}$ be fixed. We then have
\begin{align*}
\E\chi_A\int_0^{T^*}\int_D (u{-}\bar u)\phi \dx\dt
=\lim_{\varepsilon\to 0}\,\E\chi_A\int_0^{T^*}\int_D (u_\varepsilon^\leftarrow{-}u_\varepsilon)\phi \dx\dt.
\end{align*}
By a simple change of variables and the definition \eqref{eq:shift} we may write
\begin{align*}
&\E\int_0^{T^*}\int_{D} (u_\varepsilon^\leftarrow{-}u_\varepsilon)\phi \dx\dt
\\&
=\E\int_0^{T^*}\int_{\R^d} u_\varepsilon(x,t)\big(\phi(x{-}\nu(B_t{-}B_t^\varepsilon),t){-}\phi(x,t)\big) \dx\dt.
\end{align*}
Exploiting \eqref{eq:approxBrownianMotion}, \eqref{eq:HoelderContBM} and \eqref{boundMaximumPrinciple}
we further estimate
\begin{align*}
&\bigg|\E\int_0^{T^*}\int_{\R^d} u_\varepsilon(x,t)\big(\phi(x{-}\nu(B_t{-}B_t^\varepsilon),t){-}\phi(x,t)\big) \dx\dt\bigg|
\\&
\leq C\|\nabla\phi\|_{L^\infty(D\times [0,T^*])}\varepsilon^\alpha\E C_\alpha
\leq C\varepsilon^\alpha.
\end{align*}
Hence, we may infer that
\begin{align*}
\E\chi_A\int_0^{T^*}\int_D (u{-}\bar u)\phi \dx\dt = 0
\end{align*}
holds true for all $\phi\in C^\infty_{\mathrm{cpt}}(D\times [0,T^*])$
and $A\in\mathcal{F}_{T^*}$. This shows that $u=\bar u$ and thus concludes the proof
of Corollary~\ref{cor:L1convergence}.
\qed

\subsection{Proof of Proposition~\ref{prop:weakSolutionByWongZakaiApproximation} 
{\normalfont (Recovering the unique weak solution)}}
Let a test function $\phi\in C^\infty_{\mathrm{cpt}}(D)$ be fixed,
and define $K:=\mathrm{supp}\,\phi$. Fix also an integer $M\geq 1$, and
let $\mathcal{C}_\alpha$ be the square integrable random variable of \eqref{eq:HoelderContBM}.
Let $\delta\in (0,1)$ be such that $\{x\in\R^d\colon\mathrm{dist}(x,K)\leq \delta\}\subset D$.
Let $\varepsilon'=\varepsilon'(M,\delta)$ be the constant from \eqref{eq:supportAfterShift}.
By It\^{o}'s formula, the definition \eqref{eq:shift} and the fact that the 
$u_\varepsilon$ solve \eqref{eq:densityPME2} classically we have for all 
$\varepsilon\leq\varepsilon'$ and all measurable $A\in\mathcal{F}_{T^*}$
(cf.\ the argument in the first step of the proof of Proposition~\ref{prop:L1contraction})
\begin{equation}
\begin{aligned}
\label{eq:weakFormShift2a}
&\E \chi_{\{\mathcal{C}_\alpha\leq M\}} \chi_A \int_D u_\varepsilon^\leftarrow(x,T)\phi(x) \dx 
- \E \chi_{\{\mathcal{C}_\alpha\leq M\}} \chi_A \int_D  (u_0(x){+}\varepsilon)\phi(x)\dx
\\&
=\E \chi_{\{\mathcal{C}_\alpha\leq M\}} \chi_A 
\int_0^{T}\int_D \Big(\Delta (u_{\varepsilon}^\leftarrow)^m(x,t) 
{+} \frac{1}{2}\nu^2\Delta u_{\varepsilon}^\leftarrow(x,t)\Big)\phi(x) \dx \dt
\\&~~~
+ \E \chi_{\{\mathcal{C}_\alpha\leq M\}} \chi_A
\int_0^{T}\int_D \nu\nabla u_{\varepsilon}^\leftarrow(x,t)\phi(x) \dx \,\mathrm{d}B_t
\end{aligned}
\end{equation}
for all $T\in (0,T^*)$. Define the shifted test function
\begin{align*}
\phi^\leftarrow_\varepsilon(x,t) := \phi(x{-}\nu(B_t{-}B_t^\varepsilon)),
\end{align*} 
so that we obtain by a simple change of variables
\begin{equation}
\begin{aligned}
\label{eq:weakFormShift2b}
&\E \chi_{\{\mathcal{C}_\alpha\leq M\}} \chi_A \int_D u_\varepsilon^\leftarrow(x,T)\phi(x) \dx 
- \E \chi_{\{\mathcal{C}_\alpha\leq M\}} \chi_A \int_D  (u_0(x){+}\varepsilon)\phi(x)\dx
\\&
=\E \chi_{\{\mathcal{C}_\alpha\leq M\}} \chi_A 
\int_0^{T}\int_D \Big(\Delta u_{\varepsilon}^m(x,t) 
{+} \frac{1}{2}\nu^2\Delta u_{\varepsilon}(x,t)\Big)\phi^\leftarrow_\varepsilon(x,t) \dx \dt
\\&~~~
+ \E \chi_{\{\mathcal{C}_\alpha\leq M\}} \chi_A
\int_0^{T}\int_D \nu\nabla u_{\varepsilon}(x,t)\phi^\leftarrow_\varepsilon(x,t) \dx \,\mathrm{d}B_t.
\end{aligned}
\end{equation}
Note that as a consequence of \eqref{eq:supportAfterShift} we have 
almost surely on $\{\mathcal{C}_\alpha\leq M\}$ for all $\varepsilon\leq\varepsilon'$
and all $t\in [0,T^*]$ that $\mathrm{supp}\,\phi^\leftarrow_\varepsilon(\cdot,t)
\subset\subset D$. Hence, integrating by parts on the right hand side in \eqref{eq:weakFormShift2b}
does not produce any boundary integrals so that after reversing the change of variables we obtain the identity
\begin{equation}
\begin{aligned}
\label{eq:weakFormShift2c}
&\E \chi_{\{\mathcal{C}_\alpha\leq M\}} \chi_A \int_D u_\varepsilon^\leftarrow(x,T)\phi(x) \dx 
- \E \chi_{\{\mathcal{C}_\alpha\leq M\}} \chi_A \int_D  (u_0(x){+}\varepsilon)\phi(x)\dx
\\&
=\E \chi_{\{\mathcal{C}_\alpha\leq M\}} \chi_A 
\int_0^{T}\int_D \Big((u_{\varepsilon}^\leftarrow)^m(x,t) 
{+} \frac{1}{2}\nu^2 u_{\varepsilon}^\leftarrow(x,t)\Big)
\Delta\phi(x) \dx \dt
\\&~~~
- \E \chi_{\{\mathcal{C}_\alpha\leq M\}} \chi_A
\int_0^{T}\int_D \nu u_{\varepsilon}^\leftarrow(x,t)\nabla\phi(x) \dx \,\mathrm{d}B_t.
\end{aligned}
\end{equation}
We aim to pass to the limit $\varepsilon_0\geq\varepsilon\to 0$ in all four terms.
This is possible by the convergence of the shifted densities $u_\varepsilon^\leftarrow$
in $C([\tau,T^*];L^1(\Omega{\times}D,\Prob{\otimes}\mathcal{L}^d))$ as well as
in $L^1([0,T^*];L^1(\Omega{\times}D,\Prob{\otimes}\mathcal{L}^d))$, see Corollary~\ref{cor:L1convergence}.
Because of the uniform bound \eqref{boundMaximumPrinciple} the convergence in $L^1$
can actually be lifted to convergence in any $L^q([0,T^*];L^q(\Omega{\times}D,\Prob{\otimes}\mathcal{L}^d)),\,q\in (1,\infty),$
which makes the limit passage possible in the non-linear diffusion term as well as the noise term (using for the latter, e.g.,
the Burkholder--Davis--Gundy inequality). In summary, we obtain from letting $\varepsilon_0\geq\varepsilon\to 0$ 
the identity
\begin{align*}
&\E \chi_{\{\mathcal{C}_\alpha\leq M\}} \chi_A \int_D u(x,T)\phi(x) \dx 
- \E \chi_{\{\mathcal{C}_\alpha\leq M\}} \chi_A \int_D  u_0(x)\phi(x)\dx
\\&
=\E \chi_{\{\mathcal{C}_\alpha\leq M\}} \chi_A 
\int_0^{T}\int_D \Big(u^m(x,t) {+} \frac{1}{2}\nu^2u(x,t)\Big)\Delta\phi(x) \dx \dt
\\&~~~
- \E \chi_{\{\mathcal{C}_\alpha\leq M\}} \chi_A
\int_0^{T}\int_D \nu u(x,t) \nabla\phi(x) \dx \,\mathrm{d}B_t
\end{align*}
for all $T\in (0,T^*)$. Since $M\geq 1$ as well as 
$A\in\mathcal{F}_{T^*}$ were arbitrary, and the random variable $\mathcal{C}_\alpha$ is integrable, 
we thus recover \eqref{eq:weakSol}. This concludes the proof of 
Proposition~\ref{prop:weakSolutionByWongZakaiApproximation} since the asserted bounds for $u$
follow immediately from \eqref{boundMaximumPrinciple}. \qed

\section{Finite time extinction property}
\label{sec:finiteTimeExtinction}

\subsection{Viscosity theory: Maximal subsolution in a rough domain}
\label{sec:maximalSubsolution}
From now on we will restrict ourselves to the one-dimensional setting $d=1$.
To start we recall some language from viscosity theory (cf.\ ~\cite{Crandall:1992}). In particular, 
we aim to make precise what we mean by a subsolution to the Cauchy--Dirichlet problem 
\eqref{eq:pressPMErough}--\eqref{eq:pressPMElateral} of the porous medium equation
with two copies of a Brownian path as lateral boundary; and analogously for the corresponding
regularized problem \eqref{eq:pressPMErough2}--\eqref{eq:pressPMElateral2}. 

In a first step, we introduce the relevant notions for the full space problem. 
The porous medium operator in terms of the pressure variable is encoded by the functional
\begin{align}
\label{eq:functionalPorMed}
F\colon [0,\infty)\times\R\times\R \to \R,
\quad (r,q,X) \mapsto (m{-}1)rX+|q|^2.
\end{align}
Let $T^*\in (0,\infty]$ be a time horizon and $p\colon \R\times (0,T^*)\to \R$
be a function. The parabolic semijet $\Jplus p(x,t)$ of $p$ at $(x,t)\in \R\times (0,T^*)$ 
is the set of all $(q,a,X)\in\R^3$ such that
\begin{equation}
\begin{aligned}
\label{eq:parSemijet}
p(y,s) &\leq p(x,t) + a(s-t) + q(y-x) + \frac{1}{2}X(y-x)^2 
\\&~~~ 
+ o(|s{-}t|+|y{-}x|^2)\quad\text{as}\quad
\R\times(0,t] \ni (y,s) \to (x,t).
\end{aligned}
\end{equation}
Note that the condition in \eqref{eq:parSemijet} is a slightly weaker test
than the one of~\cite[(8.1)]{Crandall:1992} since we are only allowing
for $s\leq t$. This minor technicality turns out to be convenient
proving that subsolutions to, say, \eqref{eq:pressPMErough}--\eqref{eq:pressPMElateral} 
in the sense of Definition~\ref{def:subsolutionCauchyDirichlet} below are also
subsolutions to \eqref{eq:fullSpacePorMed}--\eqref{eq:fullSpaceInitialCond}.
Moreover, \eqref{eq:parSemijet} is in accordance with the definition of
subsolutions in~\cite[Definition 1]{Caffarelli:1999}. 

\begin{definition}
\label{def:subsolutionFullSpace}
Let $T^*\in (0,\infty]$ be a time horizon and
$p_0\in C^\infty_{\mathrm{cpt}}(\R;[0,\infty))$ be an initial pressure. An upper-semicontinuous
function $p\colon \R\times [0,T^*) \to [0,\infty)$ is called \emph{a subsolution for the
Cauchy problem of the porous medium equation with initial pressure $p_0$}
\begin{align}
\label{eq:fullSpacePorMed}
\partial_tp &= (m{-}1)p\partial_{xx}p + |\partial_x p|^2,
\quad (x,t)\in \R\times (0,T^*),
\\\label{eq:fullSpaceInitialCond}
p(x,0) &= p_0(x), \quad\hspace*{2.98cm} x\in\R,
\end{align}
if it holds
\begin{align}
\label{eq:subsolutionFullSpace}
a-F(p(x,t),q,X) \leq 0 
\end{align}
for all $(x,t)\in \R\times (0,T^*)$ and all $(q,a,X)\in \Jplus p(x,t)$, as well as 
if it holds $p(x,0) \leq p_0(x)$ for all $x\in\R$.
\end{definition}

In the language of~\cite{Crandall:1992}, the functional $a-F(r,p,X)$ encoding
the porous medium equation in terms of the pressure variable is not proper.
But note that it is at least degenerate elliptic. Since we will only deal with subsolutions
(or for the regularized problem with classical solutions), this is of no concern for us 
(cf.\ ~\cite{Caffarelli:1999} or~\cite{Vazquez:2005} for a viscosity theory of
the Cauchy problem of the deterministic porous medium equation).

We continue with the notion of subsolutions to \eqref{eq:pressPMErough}--\eqref{eq:pressPMElateral}
(resp.\ \eqref{eq:pressPMErough2}--\eqref{eq:pressPMElateral2}). Since we have to
incorporate Brownian paths as the lateral boundary, the following constructions are of 
course random. However, all results of this section turn out to be purely deterministic 
consequences of the probabilistic facts \eqref{eq:uniformBoundDyadicApproximation},
\eqref{eq:uniformConvergenceDyadicApproxiamtion}, \eqref{eq:HoelderContBM} and \eqref{eq:supportAfterShift}. 
In other words, we proceed with constructions to be understood in a pathwise sense.

The functional for the lateral boundary condition of the limit problem is simply defined by
$C\colon \R \to \R,\,r\mapsto r$. The one for the regularized problem is given by
$C_\varepsilon\colon \R\to\R,\, r\mapsto r-\frac{m}{m-1}\varepsilon^{m-1}$. We then introduce
a lower-semicontinuous functional 
\begin{align}
\label{eq:functionalCauchyDirichletLimit}
&G^-\colon \R\times (0,T^*)\times [0,\infty) \times \R^3 \to \R
\\\nonumber
&(x,t,r,q,a,X) \mapsto
\begin{cases}
a - F(r,q,X), & (x,t)\in \bigcup_{t\in (0,T^*)} (\nu B_t{+}I)\times \{t\}, \\
(a - F(r,q,X))\wedge C(r), & (x,t) \in \bigcup_{t\in (0,T^*)} (\nu B_t{+}\partial I)\times \{t\}, \\
C(r), & \text{else},
\end{cases}
\end{align}
encoding the equations \eqref{eq:pressPMErough} and \eqref{eq:pressPMElateral}. The corresponding
functional for the regularized problem, i.e., encoding \eqref{eq:pressPMErough2} and \eqref{eq:pressPMElateral2},
is given by
\begin{align}
\label{eq:functionalCauchyDirichletApprox}
&G^-_\varepsilon\colon \R\times (0,T^*)\times [0,\infty) \times \R^3 \to \R
\\\nonumber
&(x,t,r,q,a,X) \mapsto
\begin{cases}
a - F(r,q,X), 
& (x,t) \in \bigcup_{t\in (0,T^*)} (\nu B^\varepsilon_t{+}I)\times \{t\}, \\
C_\varepsilon(r), & \text{else}.
\end{cases}
\end{align}

\begin{definition}
\label{def:subsolutionCauchyDirichlet}
Let $T^*\in (0,\infty]$ be a time horizon, $I\subset\R$ be a bounded interval and
$\bar p_0\in C^\infty_{\mathrm{cpt}}(I;[0,\infty))$ be an initial pressure. An upper-semicontinuous
function $\bar p\colon \R\times [0,T^*) \to [0,\infty)$ is called a \emph{subsolution for the
Cauchy--Dirichlet problem} \eqref{eq:pressPMErough}--\eqref{eq:pressPMElateral} 
\emph{with initial pressure $\bar p_0$} if
\begin{align}
\label{eq:subsolutionCauchyDirichletLimit}
G^-(x,t,\bar p(x,t),q,a,X) \leq 0
\end{align}
holds true for all $(x,t)\in \R\times (0,T^*)$ and all $(q,a,X)\in \Jplus \bar p(x,t)$, as well as if
$\bar p(x,0) \leq \bar p_0(x)$ is satisfied for all $x\in \R$.

Analogously, we call an upper-semicontinuous function $\bar p_\varepsilon\colon \R\times [0,T^*) \to [0,\infty)$
a \emph{subsolution for the Cauchy--Dirichlet problem} \eqref{eq:pressPMErough2}--\eqref{eq:pressPMElateral2} if
\begin{align}
\label{eq:subsolutionCauchyDirichletApprox}
G^-_\varepsilon(x,t,\bar p_\varepsilon(x,t),q,a,X) \leq 0
\end{align}
holds true for all $(x,t)\in \R\times (0,T^*)$ and all $(q,a,X)\in \Jplus \bar p_\varepsilon(x,t)$, as well as if
$\bar p_\varepsilon(x,0) \leq \bar p_{0,\varepsilon}(x)$ is satisfied for all $x\in \R$.
\end{definition}

The reasons for relaxing \eqref{eq:pressPMErough}--\eqref{eq:pressPMElateral}
(resp.\ \eqref{eq:pressPMErough2}--\eqref{eq:pressPMElateral2}) to the full space
setting are twofold. On one hand, working in this framework turns out to be convenient 
when studying (any sort of) convergence of the subsolutions $\bar p_\varepsilon$
for the regularized problem. On the other hand, we also want to exploit
the established results from the viscosity theory~\cite{Caffarelli:1999}
(see also~\cite{Vazquez:2005}) for the Dirichlet problem of the deterministic porous medium equation
\eqref{eq:fullSpacePorMed}--\eqref{eq:fullSpaceInitialCond}.

We proceed with a list of intermediate results needed to prove the main result
of this work, Theorem~\ref{theorem:mainResult}. The corresponding proofs will be
provided afterwards. The first result concerns the construction of a subsolution 
to the regularized Cauchy--Dirichlet problem \eqref{eq:pressPMErough2}--\eqref{eq:pressPMElateral2}.

\begin{lemma}
\label{lem:subsolutionPressureApprox}
Given $\varepsilon > 0$, let $u_\varepsilon$ denote the unique weak solution to 
\emph{\eqref{eq:densityPME2}--\eqref{eq:densityPMElateral2}} with initial density 
$u_0\in C^\infty_{\mathrm{cpt}}(I;[0,\infty))$ in the sense of Lemma~\ref{lem:classicalSolution}. 
Define an associated pressure function as follows:
\begin{align}
\label{eq:definitionPressureApprox}
\bar p_\varepsilon\colon \R\times[0,T^*) &\to [0,\infty)
\\\nonumber
(x,t) &\mapsto	\begin{cases}
								\frac{m}{m-1}u_\varepsilon(x{-}\nu B^\varepsilon_t,t)^{m-1}, &
								(x,t) \in \bigcup_{t\in [0,T^*)} (\nu B^\varepsilon_t{+}I)\times \{t\}, \\
								\frac{m}{m-1}\varepsilon^{m-1}, & \text{else}.
								\end{cases} 
\end{align}
On a set with probability one the following then holds true: 

For each $\varepsilon>0$, the 
associated pressure $\bar p_\varepsilon$ is continuous and a subsolution of the problem
\emph{\eqref{eq:pressPMErough2}--\eqref{eq:pressPMElateral2}}
with initial pressure $\bar p_{0,\varepsilon}(x) := \frac{m}{m-1}(u_0(x){+}\varepsilon)^{m-1}$ 
in the sense of Definition~\ref{def:subsolutionCauchyDirichlet}. 
Moreover, we have the bounds
\begin{align}
\label{eq:boundPressureMaxPrinciple}
\frac{m}{m-1}\varepsilon^{m-1} \leq \bar p_\varepsilon(x,t) 
\leq \frac{m}{m-1}(\varepsilon{+}\|u_0\|_{L^\infty(I)})^{m-1} 
\end{align}
for all $(x,t)\in\R\times [0,T^*)$. Finally, $\bar p_\varepsilon$ is also a 
subsolution of the Cauchy problem \emph{\eqref{eq:fullSpacePorMed}--\eqref{eq:fullSpaceInitialCond}} 
with initial pressure $p_{0,\varepsilon}(x):=\frac{m}{m-1}(u_0(x){+}\varepsilon)^{m-1}$
in the sense of Definition~\ref{def:subsolutionFullSpace}.
\end{lemma}

In a next step, we construct on a set of probability one the maximal subsolution 
(also referred to as Perron's solution) $\bar p_{\max}$ of the limit Cauchy--Dirichlet
problem \eqref{eq:pressPMErough}--\eqref{eq:pressPMElateral} in the sense of 
Definition~\ref{def:subsolutionCauchyDirichlet}. Using standard arguments from
viscosity theory, the main issue is to establish the existence of \textit{a}
subsolution to \eqref{eq:pressPMErough}--\eqref{eq:pressPMElateral}. This
will be done by means of the technique of semi-relaxed limits.

\begin{proposition}
\label{prop:maxSubsolution}
For each $\varepsilon > 0$ let $\bar p_\varepsilon$ denote the subsolution
of \emph{\eqref{eq:pressPMErough2}--\eqref{eq:pressPMElateral2}} as constructed
in Lemma~\ref{lem:subsolutionPressureApprox}. Define the upper semi-relaxed limit
(with respect to backwards parabolic cylinders)
\begin{equation}
\begin{aligned}
\label{def:semirelaxedLimitPressure}
&\bar p_{\mathrm{semi-rel}}(x,t)
\\&
:=\lim_{\varepsilon\to 0}\sup_{\hat\varepsilon\leq\varepsilon}
\big\{\bar p_{\hat\varepsilon}(y,s)\colon (y,s)\in\R{\times}[0,T^*),\,
(y,s)\in (x{-}\varepsilon,x{+}\varepsilon){\times}(t{-}\varepsilon^2,t]\big\}
\end{aligned}
\end{equation}
for all $(x,t)\in\R\times [0,T^*)$. Then the following holds true almost surely: 

The upper-semicontinuous envelope of $\bar p_{\mathrm{semi-rel}}$ is a subsolution
of \emph{\eqref{eq:pressPMErough}--\eqref{eq:pressPMElateral}}
with initial pressure $\bar p_0(x) := \frac{m}{m-1}u_0(x)^{m-1}$ in the sense of 
Definition~\ref{def:subsolutionCauchyDirichlet}. Define 
\begin{align}
\label{def:maximalSubsol}
\bar p_{\max}(x,t):=\sup\{\bar p(x,t)\colon \bar p \text{ is a subsol.\ of }
\emph{\eqref{eq:pressPMErough}--\eqref{eq:pressPMElateral}},(x,t)\in\R{\times} [0,T^*)\}.
\end{align}
Then $\bar p_{\max}(x,t)<\infty$ for all $(x,t)\in\R\times [0,T^*)$, and
$\bar p_{\max}$ is a subsolution of \emph{\eqref{eq:pressPMErough}--\eqref{eq:pressPMElateral}}
with initial pressure $\bar p_0(x) := \frac{m}{m-1}u_0(x)^{m-1}$ in the sense of 
Definition~\ref{def:subsolutionCauchyDirichlet}. 
Denoting by $p_{\mathrm{visc}}\in C(\R\times [0,T^*))$ the viscosity
solution of \eqref{eq:fullSpacePorMed}--\eqref{eq:fullSpaceInitialCond}
with initial pressure $p_0(x):=\frac{m}{m-1}u_0(x)^{m-1}$ in the sense 
of~\cite[Definition 4]{Caffarelli:1999} we have
\begin{align}
\label{eq:upperBoundPerronSolution}
\bar p_{\max}(x,t) \leq p_{\mathrm{visc}}(x,t)
\end{align}
for all $(x,t)\in \R\times [0,T^*)$.
\end{proposition}

A crucial estimate for the proof of Theorem~\ref{theorem:mainResult}
is the content of the following result. The asserted bound is important in the sense that it serves to
close the loop between the pathwise constructions performed in this section and the 
unique weak solution of the SPME \eqref{eq:SPME}--\eqref{eq:lateralBC}.

\begin{proposition}
\label{prop:comparisonTransformedWeakSol}
Let $u\in\Hsp^{-1}_{m+1}(I)$ denote the unique weak solution of 
the Cauchy--Dirichlet problem \emph{\eqref{eq:SPME}--\eqref{eq:lateralBC}} 
with initial density $u_0\in C^\infty_{\mathrm{cpt}}(I;[0,\infty))$
in the sense of Definition~\ref{def:weakSolutions}. Define an associated pressure 
function as follows:
\begin{align}
\label{eq:definitionPressure}
p\colon \R\times(0,T^*) &\to [0,\infty)
\\\nonumber
(x,t) &\mapsto	\begin{cases}
								\frac{m}{m-1}u(x{-}\nu B_t,t)^{m-1}, &
								(x,t) \in \bigcup_{t\in [0,T^*)} (\nu B_t{+}I)\times \{t\}, \\
								0, & \text{else}.
								\end{cases} 
\end{align}
Let $\bar p_{\max}$ be the associated maximal subsolution of the problem 
\emph{\eqref{eq:pressPMErough}--\eqref{eq:pressPMElateral}}
as constructed in Proposition~\ref{prop:maxSubsolution}.
Then, for all $t\in (0,T^*)$ the bound
\begin{align}
\label{eq:comparisonTransformedWeakSol}
p(\cdot,t) \leq \bar p_{\max}(\cdot,t)
\end{align}
holds true almost surely almost everywhere in $\R$.
\end{proposition}

We have by now everything in place to proceed with the proofs.

\subsection{Proof of Lemma~\ref{lem:subsolutionPressureApprox} 
{\normalfont (Subsolution for the regularized problem)}}
The assertion that $\bar p_\varepsilon\in C(\R\times [0,T^*))$ 
follows from the definition \eqref{eq:definitionPressureApprox},
the regularity of the Wong--Zakai approximation 
$u_\varepsilon\in C(\bar{I}\times [0,T^*))$
and that $u_\varepsilon$ satisfies the lateral boundary condition
\eqref{eq:densityPMElateral2} (cf.\ Lemma~\ref{lem:classicalSolution})
pointwise. Moreover, $\bar p_{\varepsilon}(x,0)=\bar p_{0,\varepsilon}(x)$
holds true for all $x\in\R$ because of \eqref{eq:definitionPressureApprox}, 
$\mathrm{supp}\,u_0\subset I$ and $u_\varepsilon(x,0)=u_0(x)+\varepsilon$.
The upper and lower bound of \eqref{eq:boundPressureMaxPrinciple}
is a direct consequence of again the definition \eqref{eq:definitionPressureApprox}
and the upper and lower bound of \eqref{boundMaximumPrinciple}.

We next have to show that \eqref{eq:subsolutionCauchyDirichletApprox}
is satisfied for all $(x,t)\in \R\times (0,T^*)$ and all $(q,a,X)\in \Jplus \bar p_\varepsilon(x,t)$.
The claim is trivial for $(x,t) \in \bigcup_{t\in (0,T^*)} (\nu B^\varepsilon_t{+}I)\times \{t\}$.
Indeed, since $u_\varepsilon$ solves \eqref{eq:densityPME2}--\eqref{eq:densityPMElateral2}
classically in $I\times (0,T^*)$ and $B^\varepsilon$ is smooth we infer from 
\eqref{eq:definitionPressureApprox}, the chain rule and elementary computations that $\bar p_\varepsilon$
satisfies the porous medium equation 
$\partial_t\bar p_\varepsilon = \bar p_\varepsilon\partial_{xx}\bar p_\varepsilon + |\partial_x\bar p_\varepsilon|^2$
classically in the open space-time domain $\bigcup_{t\in (0,T^*)} (\nu B^\varepsilon_t{+}I)\times \{t\}$.
If $(x,t)\in\bigcup_{t\in (0,T^*)} (\nu B^\varepsilon_t{+}(\R\setminus I))\times \{t\}$, the 
claim is again trivial by the definition \eqref{eq:definitionPressureApprox} of $\bar p_\varepsilon$
and \eqref{eq:functionalCauchyDirichletApprox}. This proves \eqref{eq:subsolutionCauchyDirichletApprox}.

We finally have to show that \eqref{eq:subsolutionFullSpace} holds true
for all $(x,t)\in \R\times (0,T^*)$ and all $(q,a,X)\in \Jplus \bar p_\varepsilon(x,t)$.
By the previous reasoning, it just remains to consider the case of a space-time point on the
lateral boundary $(x,t)\in\bigcup_{t\in (0,T^*)} (\nu B^\varepsilon_t{+}\partial I)\times \{t\}$.
Without loss of generality we may assume that $(x,t)$ sits on the upper part of the lateral
boundary. Let $(q,a,X)\in \Jplus \bar p_\varepsilon(x,t)$ be fixed. Since $\bar p_\varepsilon$ satisfies 
the lateral boundary condition \eqref{eq:pressPMElateral2} pointwise, it follows from the definition 
of the parabolic semijet \eqref{eq:parSemijet} that
\begin{align*}
\bar p_\varepsilon(x{-}1/k,t) \leq 
\frac{m}{m-1}\varepsilon^{m-1} - \frac{q}{k} + \frac{X}{2k^2} + o(1/k^2)
\end{align*}
as well as
\begin{align*}
\frac{m}{m-1}\varepsilon^{m-1} = \bar p_\varepsilon(x{+}1/k,t) \leq 
\frac{m}{m-1}\varepsilon^{m-1} + \frac{q}{k} + \frac{X}{2k^2} + o(1/k^2)
\end{align*}
for all sufficiently large $k\geq 1$. Due to the lower bound in \eqref{eq:boundPressureMaxPrinciple}
we infer by adding both inequalities that $0\leq X/k^2+o(1/k^2)$ for all sufficiently large $k\geq 1$.
From this we deduce that $X\geq 0$. Using once more \eqref{eq:parSemijet} as well as that 
$\bar p_\varepsilon$ satisfies the lateral boundary condition \eqref{eq:pressPMElateral2} 
pointwise it also holds
\begin{align*}
\bar p_\varepsilon(x,t{-}1/k) \leq \frac{m}{m-1}\varepsilon^{m-1} - \frac{a}{k} + o(1/k)
\end{align*}
for all sufficiently large $k\geq 1$. Hence, by another application of the lower bound in
\eqref{eq:boundPressureMaxPrinciple} we infer that $0\leq -a/k + o(1/k)$ for all sufficiently 
large $k\geq 1$. In other words, it holds $a\leq 0$. To summarize we have shown that
\begin{align*}
a - F(\bar p_\varepsilon(x,t),q,X) = a - (m{-}1)\bar p_{\varepsilon}(x,t)X - |q|^2 \leq 0
\end{align*}
is satisfied for all $(x,t)\in\bigcup_{t\in (0,T^*)} (\nu B^\varepsilon_t{+}\partial I)\times \{t\}$
and all $(q,a,X)\in \Jplus \bar p_\varepsilon(x,t)$. This proves \eqref{eq:subsolutionFullSpace} 
and thus concludes the proof of Lemma~\ref{lem:subsolutionPressureApprox}. \qed

\subsection{Proof of Proposition~\ref{prop:maxSubsolution} 
{\normalfont (Perron's solution for the limit problem)}}
First note that $0\leq \bar p_{\text{semi-rel}}(x,t)<\infty$ is satisfied 
for all $(x,t)\in \R\times [0,T^*)$ as a consequence of \eqref{def:semirelaxedLimitPressure} and 
\eqref{eq:boundPressureMaxPrinciple}. Since $\bar p_\varepsilon(x,0) = \bar p_{0,\varepsilon}(x)
=\frac{m}{m-1}(u_0(x){+}\varepsilon)^{m-1}$ for all $x\in\R$ we deduce that
$\bar p_{\text{semi-rel}}(x) = \frac{m}{m-1}u_0(x)^{m-1}=\bar p_0(x)$ holds true
for all $x\in\R$ since $u_0$ is continuous and the semi-relaxed limit 
\eqref{def:semirelaxedLimitPressure} is defined in terms of backwards parabolic cylinders.

We aim to show that $\bar p_{\text{semi-rel}}$ satisfies \eqref{eq:subsolutionCauchyDirichletLimit}
for all $(x,t)\in \R\times (0,T^*)$ and all $(q,a,X)\in \Jplus \bar p(x,t)$. For this,
we will proceed in two steps as follows. Consider the lower-semicontinuous functional
\begin{align}
\label{eq:functionalCauchyDirichletApproxLSC}
&G^-_{\varepsilon,\text{lsc}}\colon \R\times (0,T^*)\times [0,\infty) \times \R^3 \to \R
\\\nonumber
&(x,t,r,q,a,X) \mapsto
\begin{cases}
a - F(r,q,X), & (x,t)\in \bigcup_{t\in (0,T^*)} (\nu B^\varepsilon_t{+}I)\times \{t\}, \\
(a - F(r,q,X))\wedge C_\varepsilon(r), & 
(x,t) \in \bigcup_{t\in (0,T^*)} (\nu B^\varepsilon_t{+}\partial I)\times \{t\}, \\
C_\varepsilon(r), & \text{else}.
\end{cases}
\end{align}
Since $\bar p_\varepsilon$ is a subsolution of \eqref{eq:pressPMErough2}--\eqref{eq:pressPMElateral2}
in the sense of Definition~\ref{def:subsolutionCauchyDirichlet}, it of course also satisfies
\begin{align}
\label{eq:subsolutionCauchyDirichletApproxLSC}
G^-_{\varepsilon,\text{lsc}}(x,t,\bar p_\varepsilon(x,t),q,a,X) \leq 0
\end{align}
for all $(x,t)\in \R\times (0,T^*)$ and all $(q,a,X)\in \Jplus \bar p_\varepsilon(x,t)$.
Now, define the lower semi-relaxed limit functional
\begin{equation}
\begin{aligned}
\label{def:semirelaxedLimitFunctional}
& G^-_{\mathrm{semi-rel}}(x,t,r,q,a,X)
\\&
:=\lim_{\varepsilon\to 0}\inf_{\hat\varepsilon\leq\varepsilon}
\big\{G^-_{\hat\varepsilon,\text{lsc}}(y,s,\hat r,\hat q,\hat a,\hat X)
\colon (y,s)\in\R{\times}(0,T^*){\times} [0,\infty) {\times} \R^3, \\&\hspace*{2cm}
(y,s)\in (x{-}\varepsilon,x{+}\varepsilon){\times}(t{-}\varepsilon^2,t],\,
|(r{-}\hat r,q{-}\hat q,a{-}\hat a,X{-}\hat X)|<\varepsilon\big\}.
\end{aligned}
\end{equation}
In a first step, we check that $\bar p_{\text{semi-rel}}$ satisfies
\begin{align}
\label{eq:subsolutionCauchyDirichletApproxSemirel}
G^-_{\text{semi-rel}}(x,t,\bar p_{\text{semi-rel}}(x,t),q,a,X) \leq 0
\end{align}
for all $(x,t)\in \R\times (0,T^*)$ and all $(q,a,X)\in \Jplus \bar p_{\text{semi-rel}}(x,t)$.
In a second step, we identify the lower semi-relaxed limit $G^-_{\mathrm{semi-rel}}$ with
the functional $G^-$ defined in \eqref{eq:functionalCauchyDirichletLimit}
as a consequence of the uniform convergence \eqref{eq:uniformConvergenceDyadicApproxiamtion}.

The validity of \eqref{eq:subsolutionCauchyDirichletApproxSemirel} is a consequence
of standard viscosity theory. More precisely, \eqref{eq:subsolutionCauchyDirichletApproxSemirel} 
follows as a combination of~\cite[Lemma 6.1]{Crandall:1992}, \cite[Remark 6.3]{Crandall:1992}
(for the purpose of subsolutions, the equations do not have to be proper) as well as
the applicability of~\cite[Proposition 4.3]{Crandall:1992}. Hence, let us show that 
$G^-_{\mathrm{semi-rel}}=G^-$.

Due to the uniform convergence \eqref{eq:uniformBoundDyadicApproximation} 
and the definitions \eqref{eq:functionalCauchyDirichletLimit},
\eqref{eq:functionalCauchyDirichletApproxLSC} and \eqref{def:semirelaxedLimitFunctional} 
the statement is clear for space-time points $(x,t)\notin\bigcup_{t\in (0,T^*)}(\nu B_t{+}\partial I){\times}\{t\}$.
Hence, let us fix a point $(x,t)\in\bigcup_{t\in (0,T^*)}(\nu B_t{+}\partial I){\times}\{t\}$
and an $\varepsilon>0$. Let 
\begin{align*}
\hat\varepsilon_*(\varepsilon) := 
\sup\{0<\hat\varepsilon\leq\varepsilon\colon \nu|B^{\hat\varepsilon}_t{-}B_t|<\varepsilon/2\}.
\end{align*}
Note that $\hat\varepsilon_*>0$ because of the uniform convergence \eqref{eq:uniformBoundDyadicApproximation}. 
We then have
\begin{align*}
\lim_{\varepsilon\to 0}\inf_{(\hat r,\hat q,\hat a,\hat X)} 
(\hat a {-} F(\hat r,\hat q,\hat X))\wedge C_\varepsilon(\hat r)
&\leq G^-_{\mathrm{semi-rel}}(x,t,r,q,a,X)
\\&
\leq \lim_{\varepsilon\to 0}\inf_{(\hat r,\hat q,\hat a,\hat X)} 
(\hat a {-} F(\hat r,\hat q,\hat X))\wedge C_{\hat\varepsilon_*(\varepsilon)}(\hat r),
\end{align*}
and where both inner infima run over all points $(\hat r,\hat q,\hat a,\hat X)\in [0,\infty)\times\R^3$ such that
$|(r{-}\hat r,q{-}\hat q,a{-}\hat a,X{-}\hat X)|<\varepsilon$. It follows that
$G^-_{\mathrm{semi-rel}}(x,t)=G^-(x,t)$ as claimed.

We have shown so far that $\bar p_{\text{semi-rel}}$ satisfies all conditions of
Definition~\ref{def:subsolutionCauchyDirichlet} except of being upper-semicontinuous
on $\R\times [0,T^*)$. Since $\bar p_{\text{semi-rel}}$ is subject to 
\eqref{eq:subsolutionCauchyDirichletLimit}, it is a classical fact of viscosity theory
that also its upper-semicontinuous envelope satisfies \eqref{eq:subsolutionCauchyDirichletLimit}.
(A rigorous argument consists of applying~\cite[Lemma 6.1]{Crandall:1992} to
a constant sequence.) It remains to check whether the upper-semicontinuous envelope
of $\bar p_{\text{semi-rel}}$ satisfies the initial condition with respect to
$\bar p_0:=\frac{m}{m-1}u_0^{m-1}$. This follows from the following reasoning.

Consider the classical solution (and therefore also viscosity solution in the sense 
of~\cite[Definition 4]{Caffarelli:1999}) $p^\varepsilon_{\text{visc}}$ of
\eqref{eq:fullSpacePorMed}--\eqref{eq:fullSpaceInitialCond} with initial pressure
$p_{0,\varepsilon}:=\frac{m}{m-1}(u_0{+}\varepsilon)^{m-1}$. Because of \eqref{eq:uniformBoundDyadicApproximation},
we can choose almost surely a space-time cylinder $Q=Q(B^\varepsilon)$ such that
the parabolic closure of $\bigcup_{t\in (0,T^*)}(\nu B^\varepsilon_t{+}I){\times}\{t\}$
is contained in $Q$. By the maximum principle, $p^\varepsilon_{\text{visc}}$ is also subject to
\eqref{eq:boundPressureMaxPrinciple} on $\R\times [0,T^*)$. In particular, $p^\varepsilon_{\text{visc}}$
dominates $\bar p_\varepsilon$ on the parabolic boundary of $Q$. Since $\bar p_\varepsilon$ is 
continuous and a subsolution of \eqref{eq:fullSpacePorMed}--\eqref{eq:fullSpaceInitialCond}
by Lemma~\ref{lem:subsolutionPressureApprox}, it follows from~\cite[Lemma 2.5]{Vazquez:2005} that
\begin{align*}
\bar p_\varepsilon(x,t) \leq p^\varepsilon_{\text{visc}}(x,t)
\end{align*}
for all $(x,t)$ in the parabolic closure of $Q$, hence for all $(x,t)\in\R\times [0,T^*)$
by expanding $Q$ to $\R\times [0,T^*)$. Since $p^\varepsilon_{\text{visc}}\searrow p_{\text{visc}}$
with $p_{\mathrm{visc}}\in C(\R\times [0,T^*))$ being the viscosity
solution of \eqref{eq:fullSpacePorMed}--\eqref{eq:fullSpaceInitialCond}
with initial pressure $p_0(x):=\frac{m}{m-1}u_0(x)^{m-1}$ in the sense 
of~\cite[Definition 4]{Caffarelli:1999}, cf.\ the proof of~\cite[Lemma 2.2]{Caffarelli:1999}, 
we obtain the bound 
\begin{align}
\label{eq:upperBoundSubsolutions}
\bar p_{\text{semi-rel}}(x,t) \leq p_{\mathrm{visc}}(x,t)
\end{align}
for all $(x,t)\in \R\times [0,T^*)$. However, since we already proved that
$\bar p_{\text{semi-rel}}(x,0)=\frac{m}{m-1}u_0(x)^{m-1}=p_{\text{visc}}(x,0)$ and 
$p_{\text{visc}}$ is continuous, it follows that also the upper-semicontinuous
envelope of $\bar p_{\text{semi-rel}}$ attains the initial condition. Hence,
it is a subsolution of \emph{\eqref{eq:pressPMErough}--\eqref{eq:pressPMElateral}}
in the sense of Definition~\ref{def:subsolutionCauchyDirichlet}.

We continue with the verification of the claims regarding $\bar p_{\text{max}}$.
To this end, consider first an arbitrary subsolution $\bar p$ of 
\eqref{eq:pressPMErough}--\eqref{eq:pressPMElateral}
with initial pressure $\bar p_0:=\frac{m}{m-1} u_0^{m-1}$
in the sense of Definition~\ref{def:subsolutionCauchyDirichlet}. 
The argument showing \eqref{eq:upperBoundSubsolutions} 
more generally proves that
\begin{align}
\label{eq:upperBoundAllSubsolutions}
\bar p(x,t) \leq p_{\mathrm{visc}}(x,t)
\end{align}
is satisfied for all $(x,t)\in \R\times [0,T^*)$. Hence,
since we already established the existence of \textit{a} subsolution
to \eqref{eq:pressPMErough}--\eqref{eq:pressPMElateral} the definition
of $\bar p_{\text{max}}$ is meaningful, and the asserted bound \eqref{eq:upperBoundPerronSolution}
then follows at once from \eqref{eq:upperBoundAllSubsolutions}. In particular,
$\bar p_{\max}<\infty$ is satisfied on $\R\times [0,T^*)$. That $\bar p_{\max}$ is subject to
\eqref{eq:subsolutionCauchyDirichletLimit} is once again a classical
fact from viscosity theory, see~\cite[Lemma 4.2]{Crandall:1992}. Moreover, as we have already
argued for the semi-relaxed limit $p_{\text{semi-rel}}$, it then follows that
the upper-semicontinuous envelope of $\bar p_{\max}$ constitutes a subsolution
of \eqref{eq:pressPMErough}--\eqref{eq:pressPMElateral}. However, by the definition of
$\bar p_{\max}$ we may then infer that $\bar p_{\max}$ is actually equal to its
upper-semicontinuous envelope; in particular, a subsolution of \eqref{eq:pressPMErough}--\eqref{eq:pressPMElateral}
in the sense of Definition~\ref{def:subsolutionCauchyDirichlet}.
This concludes the proof of  Proposition~\ref{prop:maxSubsolution}. \qed

\subsection{Proof of Proposition~\ref{prop:comparisonTransformedWeakSol}
{\normalfont (Comparison of transformed weak solution with viscosity solution)}}
It obviously suffices to prove that for all $t\in (0,T^*)$
\begin{align}
\label{eq:lowerBoundSemirelaxed}
p(\cdot,t) \leq \bar p_{\text{semi-rel}}(\cdot,t)
\end{align}
almost surely almost everywhere in $\R$. By definition of $p$, see \eqref{eq:definitionPressure}, 
and since we have $\bar p_{\text{semi-rel}}\geq 0$ it moreover suffices to show that  
\eqref{eq:lowerBoundSemirelaxed} is almost surely satisfied for all $t\in (0,T^*)$ almost everywhere in $\nu B_t+I$. 
 
We aim to exploit Corollary~\ref{cor:L1convergence}, i.e., that the unique weak solution $u$ of 
\eqref{eq:SPME}--\eqref{eq:lateralBC} can be recovered by means of the Wong--Zakai approximations 
$u_\varepsilon$ from Lemma~\ref{lem:classicalSolution}, at least after 
employing an additional time-dependent shift \eqref{eq:shift}. So fix an integer $M\geq 1$ 
as well as some $\kappa>0$. Let $\mathcal{C}_\alpha$ be the square integrable random variable of 
the estimate \eqref{eq:HoelderContBM}. Let finally $\varepsilon>0$ be fixed, and denote by
$\varepsilon'=\varepsilon'(M,\varepsilon)$ the constant from \eqref{eq:supportAfterShift}.

We then have by means of the estimate \eqref{eq:supportAfterShift} that
\begin{align}
\nonumber
&\frac{m}{m-1}u_{\hat\varepsilon}(x{-}\nu B_t,t)^{m-1} 
\\&\nonumber
=\frac{m}{m-1}u_{\hat\varepsilon}\big((x{-}\nu(B_t{-}B^{\hat\varepsilon}_t))-\nu B^{\hat\varepsilon}_t,t\big)^{m-1}
\\&\label{eq:aux}
\leq \sup_{\hat\varepsilon\leq\varepsilon}
\big\{\bar p_{\hat\varepsilon}(y,s)\colon (y,s)\in\R{\times}[0,T^*),\,
(y,s)\in (x{-}\varepsilon,x{+}\varepsilon){\times}(t{-}\varepsilon^2,t]\big\}
\\&\nonumber
=: \bar P_\varepsilon(x,t)
\end{align}
almost surely on $\{\mathcal{C}_\alpha\leq M\}$ for all $\hat\varepsilon\leq\varepsilon\wedge\varepsilon'$
and all $(x,t)\in \bigcup_{t\in (0,T^*)} (\nu B_t{+}I)\times \{t\}$.
Fix $t\in (0,T^*)$, let $\phi\in C_{\mathrm{cpt}}^\infty(\nu B_t{+}I;[0,\infty))$
be an arbitrary test function and let $A\in\mathcal{F}_{T^*}$. It follows from \eqref{eq:aux} that
\begin{equation}
\begin{aligned}
\label{eq:aux2}
0&\leq\E\chi_{A\cap\{C_\alpha{\leq}M\}}\int_\R \phi(x{+}\nu B_t^{\hat\varepsilon})
\frac{m}{m-1}u^\leftarrow_{\hat\varepsilon}(x,t)^{m-1} \dx
\\&
\leq \E\chi_{A\cap\{C_\alpha{\leq}M\}}\int_\R \phi(x)\bar P_\varepsilon(x,t) \dx
\end{aligned}
\end{equation}
for all $\hat\varepsilon\leq\varepsilon\wedge\varepsilon'$,
all $t\in (0,T^*)$, all $\phi\in C_{\mathrm{cpt}}^\infty(\nu B_t{+}I;[0,\infty))$ and all $A\in\mathcal{F}_{T^*}$.
As a consequence of the bound \eqref{eq:aux2}, the convergences in \eqref{eq:uniformConvergenceDyadicApproxiamtion}
resp.\ Corollary~\ref{cor:L1convergence} as well as the definition \eqref{def:semirelaxedLimitPressure} 
of the semi-relaxed limit $\bar p_{\text{semi-rel}}$, we thus obtain the estimate
\begin{equation}
\begin{aligned}
\label{eq:aux3}
0 &\leq \E\chi_{A\cap\{C_\alpha{\leq}M\}}\int_\R \phi(x)\frac{m}{m-1}u(x{-}\nu B_t,t)^{m-1} \dx
\\&
\leq \E\chi_{A\cap\{C_\alpha{\leq}M\}}\int_\R \phi(x)\bar p_{\text{semi-rel}}(x,t) \dx
\end{aligned}
\end{equation}
for all $t\in (0,T^*)$, all $\phi\in C_{\mathrm{cpt}}^\infty(\nu B_t{+}I;[0,\infty))$ and all $A\in\mathcal{F}_{T^*}$. 
Since the test function and $A$ are arbitrary, we deduce that the following bound holds true 
for all $t\in (0,T^*)$ almost surely on $\mathcal\{C_\alpha \leq M\}$ almost everywhere in $\nu B_t{+}I$:
\begin{align*}
p(\cdot,t) = \frac{m}{m-1}u(\cdot\,{-}\nu B_t,t)^{m-1} 
\leq \bar p_{\text{semi-rel}}(\cdot,t).
\end{align*}
Since $M\geq 1$ is arbitrary, and the random variable $\mathcal{C}_\alpha$
is integrable, we may infer from this the desired bound \eqref{eq:comparisonTransformedWeakSol}.
This concludes the proof of Proposition~\ref{prop:comparisonTransformedWeakSol}. \qed

\subsection{Proof of Theorem~\ref{theorem:mainResult}
{\normalfont (Finite time extinction with probability one)}}
\label{sec:proofMainResult}
Let $p_{\mathrm{visc}}\in C(\R\times [0,\infty))$ be the unique viscosity
solution of \eqref{eq:fullSpacePorMed}--\eqref{eq:fullSpaceInitialCond}
with initial pressure $p_0(x):=\frac{m}{m-1}u_0(x)^{m-1}$ in the sense 
of~\cite[Definition 4]{Caffarelli:1999}. Since the time horizon $T^*<\infty$
was arbitrary in Proposition~\ref{prop:maxSubsolution} and Proposition~\ref{prop:comparisonTransformedWeakSol},
we have constructed almost surely a (maximal) subsolution $\bar p_{\max}$ of 
\eqref{eq:pressPMErough}--\eqref{eq:pressPMElateral} on $\R\times [0,\infty)$ 
in the sense of Definition~\ref{def:subsolutionCauchyDirichlet} such that it holds almost surely
\begin{align}
\label{eq:upperBoundPerronSolution2}
\bar p_{\max}(x,t) \leq p_{\mathrm{visc}}(x,t)
\end{align}
for all $(x,t)\in \R\times [0,\infty)$, as well as for all $t\in (0,\infty)$
\begin{align}
\label{eq:comparisonTransformedWeakSol2}
p(x,t) \leq \bar p_{\max}(x,t)
\end{align}
almost surely almost everywhere in $\R$, where $p$ is
defined on $\R\times (0,\infty)$ via \eqref{eq:definitionPressure} based on the
unique weak solution of the Cauchy--Dirichlet problem \eqref{eq:SPME}--\eqref{eq:lateralBC} 
on $\R\times (0,\infty)$ with initial density $u_0$. 

Next, we choose a delayed Barenblatt solution $\mathcal{B}$ (written in terms of the pressure variable) 
with free boundary $\partial I$ at $t=0$ and which strictly dominates the initial density,
i.e., $\{\mathcal{B}(\cdot,0)=0\}=\partial I$ and $ u_0<\mathcal{B}(\cdot,0)$ on $I$.
The free boundary associated to the Barenblatt solution $\mathcal{B}$ can then be written as
$\bigcup_{t\in [0,\infty)}(\bar Mt^\frac{1}{m+1}{+}\partial I){\times}\{t\}$ 
for a constant $\bar M=\bar M(u_0)>0$ depending only on the initial density. 

We remark that the Barenblatt solution $\mathcal{B}$ is a viscosity solution 
of \eqref{eq:fullSpacePorMed}--\eqref{eq:fullSpaceInitialCond}
in the sense of~\cite[Definition 4]{Caffarelli:1999} strictly dominating $p_{\text{visc}}$ 
at $t=0$ in the sense of~\cite[Definition 5]{Caffarelli:1999}. Hence, it 
follows from the comparison principle~\cite[Theorem 4.1]{Caffarelli:1999} that
almost surely
\begin{align}
\label{eq:comparisonBarenblatt}
p_{\text{visc}}(x,t) \leq \mathcal{B}(x,t)
\end{align}
for all $(x,t)\in \R\times [0,\infty)$. 

Let $M>\bar M$ be fixed but otherwise arbitrary,
and define the stopping time $\hatText := \inf\{t\geq 0\colon |B_t| \geq \mathcal{L}^1(I){+}Mt^\frac{1}{m+1}\}$
as in the statement of Theorem~\ref{theorem:mainResult}. Furthermore, fix $T\in (0,\infty)$
and consider $\hat{\mathcal{A}}_T:=\{\hat T_{\mathrm{extinct}}\leq T\}$. Note that $\hatText>0$.
By almost sure continuity of $B_t$ and since $M>\bar M$, there is almost surely on $\hat{\mathcal{A}}_T$ some 
$t_q\in\mathbb{Q}\cap (0,T]$ such that $|B_{t_q}|\geq \mathcal{L}^1(I)+\bar Mt_q^{\frac{1}{m+1}}$.
As a consequence of \eqref{eq:upperBoundPerronSolution2} and the choice of $\bar M$, we deduce 
that almost surely on $\hat{\mathcal{A}}_T$ it holds $\bar p_{\mathrm{max}}(\cdot,t_q)\equiv 0$.
Moreover, \eqref{eq:comparisonTransformedWeakSol2} entails that almost surely we have
$p\leq p_{\text{max}}$ for all $t\in\mathbb{Q}\cap (0,\infty)$ almost everywhere in $\R$.
Hence, we may deduce that $\hat{\mathcal{A}}_T \subset \{\Text\leq T\}$ up to some $\Prob$
null set. This  proves the desired bound \eqref{eq:boundExtinctionTime}.  

Finite time extinction with probability one now follows from standard properties of Brownian motion. Indeed,
assume that we have $\Prob(\hatText<\infty) < 1$. There would then exist a constant $C<\infty$ 
such that on a set with non-vanishing probability it holds $0\leq t^{-\frac{1}{2}}|B_t|<C$ 
for all $t\in (1,\infty)$; a contradiction. Hence, $\Prob(\hatText<\infty) = 1$ which
also by \eqref{eq:boundExtinctionTime} entails that $\Prob(\Text<\infty) = 1$.

It remains to prove that $u(\cdot,T)$ vanishes almost surely almost everywhere in $I$ 
from the stopping time $\hatText$ onwards. So let $T \in (0,\infty)$ be fixed.
We infer from \eqref{eq:upperBoundPerronSolution2} together with \eqref{eq:comparisonBarenblatt} 
that almost surely $\bar p_{\max}(\cdot,\hatText)\equiv 0$. In addition, we note that $\bar p_{\max}$ is almost surely a 
subsolution of \eqref{eq:fullSpacePorMed}--\eqref{eq:fullSpaceInitialCond}
on $\R\times [0,\infty)$ with initial pressure $p_0(x):=\frac{m}{m-1}u_0(x)^{m-1}$ in 
the sense of Definition~\ref{def:subsolutionFullSpace}. For a proof, the only non-trivial case
is that of a lateral boundary point $(x,t)\in\bigcup_{t\in (0,\infty)}(\nu B_t{+}I){\times}\{t\}$
such that $C(\bar p_{\max}(x,t))=\bar p_{\max}(x,t)\leq 0$. Since $\bar p_{\max}(x,t)\geq 0$
by construction, it follows that $\bar p_{\max}(x,t)=0$. 
From this point onwards one may argue as for $\bar p_\varepsilon$ in the proof of Lemma~\ref{lem:subsolutionPressureApprox}. 

However, once a subsolution became trivial it stays trivial (by comparison with the trivial viscosity solution).
We deduce that $\bar p_{\max}(\cdot,T)\equiv 0$ is satisfied almost surely on $\{T\geq\hatText\}$.
The claim thus follows from another application of \eqref{eq:comparisonTransformedWeakSol2}.
This concludes the proof of Theorem~\ref{theorem:mainResult}.
\qed


\end{document}